\documentclass[preprint,11pt]{elsarticle}

\usepackage{ifpdf}
\usepackage{amsmath}
\usepackage{amsfonts}
\usepackage{amsthm}
\usepackage{mathrsfs}
\usepackage{color}
\usepackage{graphicx}
\usepackage{subfigure}
\usepackage{float}
\usepackage{overpic}
\usepackage{threeparttable}
\usepackage{epstopdf}
\usepackage{dcolumn}
\usepackage{multirow}
\usepackage{booktabs}
\biboptions{numbers,sort&compress}

\usepackage{graphicx,natbib,amssymb,lineno}
\ifpdf
\usepackage[%
  pdftitle={Instructions for use of the document class
    elsart},%
  pdfauthor={},%
  pdfsubject={The preprint document class elsart},%
  pdfkeywords={instructions for use, elsart, document class},%
  pdfstartview=FitH,%
  bookmarks=true,%
  bookmarksopen=true,%
  breaklinks=true,%
  colorlinks=true,%
  linkcolor=blue,anchorcolor=blue,%
  citecolor=blue,filecolor=blue,%
  menucolor=blue,pagecolor=blue,%
  urlcolor=blue]{hyperref}
\else
\usepackage[%
  breaklinks=true,%
  colorlinks=true,%
  linkcolor=blue,anchorcolor=blue,%
  citecolor=blue,filecolor=blue,%
  menucolor=blue,pagecolor=blue,%
  urlcolor=blue]{hyperref}
\fi

\makeatletter
\def\elsartstyle{%
    \def\normalsize{\@setfontsize\normalsize\@xiipt{14.5}}
    \def\small{\@setfontsize\small\@xipt{13.6}}
    \let\footnotesize=\small
    \def\large{\@setfontsize\large\@xivpt{18}}
    \def\Large{\@setfontsize\Large\@xviipt{22}}
    \skip\@mpfootins = 18\p@ \@plus 2\p@
    \normalsize
} \@ifundefined{square}{}{} \makeatother

\newtheorem{theorem}{Theorem}[section]
\newtheorem{lemma}[theorem]{Lemma}

\theoremstyle{definition}

\newtheorem{remark}[theorem]{Remark}

\makeatletter
\def\ps@pprintTitle{%
  \let\@oddhead\@empty
  \let\@evenhead\@empty
  \let\@oddfoot\@empty
  \let\@evenfoot\@oddfoot
}
\makeatother

\def\ps@pprintTitle{%
  \let\@oddhead\@empty
  \let\@evenhead\@empty
  \def\@oddfoot{\reset@font\hfil\thepage\hfil}
  \let\@evenfoot\@oddfoot
}

\begin{document}

\begin{frontmatter}

\title{Determining codimension of Bogdanov-Takens and Bautin 
bifurcations via simplest normal form computation}

\author[myaddress1,myaddress2]
{Pei Yu\corref{mycorrespondingauthor}}
\cortext[mycorrespondingauthor]{Corresponding author}
\ead{pyu@uwo.ca}
\author[myaddress2,myaddress3]
{Yanni Zeng}
\ead{yzeng@ttu.edu}
\author[myaddress1]
{Maoan Han}
\ead{mahan@zjnu.edu.cn}
\address[myaddress1]{School of Mathematical Sciences, Zhejiang Normal 
University, \\
Jinhua, Zhejiang Province, 321004 China \vspace*{0.05in}}
\address[myaddress2]{Department of Mathematics,
Western University, London, Ontario, N6A~5B7, Canada \vspace*{0.05in}}
\address[myaddress3]{Mathematics and Statistics, Texas Tech University,
Lubbock, TX 79409, USA}

\begin{abstract}
In solving real-world problems, determining the codimension of 
Bogdanov-Takens (BT) and Bautin (generalized Hopf) bifurcations can be very 
challenging, even for simple two-dimensional dynamical systems. 
This difficulty becomes particularly evident when the number of system 
parameters exceeds the codimension of the bifurcations. 
Such challenges are closely linked to analyzing complex dynamics, 
such as the bifurcation of multiple limit cycles and 
homoclinic/heteroclinic bifurcations. In this paper, we use two 
population systems to demonstrate a systematic approach for 
determining the conditions that define the codimension 
of BT and Bautin bifurcations.

\end{abstract}


\end{frontmatter}

\section{Introduction}

Limit cycles frequently arise in real-world problems, such as population 
models, giving rise to the well-known phenomenon of 
self-oscillations~\cite{GuckenheimerHolmes1993,Bazykin1998,HanYu2012}. 
The theory of limit cycles plays a central role in the study of 
complex nonlinear dynamical systems, particularly in connection 
with two fundamental bifurcations: 
the Bogdanov–Takens (BT) bifurcation and Bautin (or  generalized Hopf, 
simpliy GH) bifurcation. A key step in such analyses is the derivation 
of the corresponding normal form and the determination of the 
bifurcations codimension.

In practice, however, determining the codimension of BT and GH bifurcations 
can be extremely challenging, especially when physical constraints are 
imposed on system parameters, such as requiring them to be positive or 
confined within a specific range. These restrictions can significantly 
complicate the bifurcation analysis.

The principal mathematical tool for analyzing these bifurcations is normal 
form theory, often applied together with center manifold theory. 
This framework typically involves two levels of simplification: 
the conventional (or classical) normal form (CNF), and a further 
refinement known as the simplest (also called unique, minimal, or hypernormal) 
normal form (SNF), as discussed in~\cite{Baider1991,Algaba1998,Ushiki1984,
Yu1999,GY2010,GY2012,YL2003,GM2015,YuZhang2019}.
The basic idea is to express the system’s dynamics in terms of its 
parameters or coefficients, then solve the resulting multivariate 
polynomial system obtained from the normal form computation. The solutions
provide the conditions required to determine the codimension of BT and 
GH bifurcations. Generally, determining 
the codimension of a GH bifurcation is more difficult than for a BT 
bifurcation, partly because BT case involves one fewer 
parameter once the critical point is identified. The main difficulty, 
however, often lies not only in computing the normal form but also in solving 
the associated multivariate polynomial systems.

For higher-codimensional problems, even powerful algebraic methods,
such as Gr\"{o}bner bases~\cite{Buchberger1998}, 
regular chains~\cite{Aubry1999,Chen2013}, and Wu's method of characteristic 
sets~\cite{Wu2000}, often fail to yield the desired results, despite 
advanced computational tools like Maple. The challenge becomes especially 
harder when the number of parameters exceeds the bifurcation’s codimension.

For instance, consider a system with five real parameters. In theory, 
the maximum number of small-amplitude limit cycles that can bifurcate 
from a GH point equals the number, here, five. 
However, in systems arising 
from biology or physics, constraints on parameters (e.g., positivity or 
bounded ranges) often reduce this number to four, 
three, two, or even one. In such cases, determining the codimension 
of the GH bifurcation, that is, finding the maximum number of 
bifurcating limit cycles, becomes considerably more challenging.

For BT bifurcations, codimension-2 cases are standard in the 
literature~\cite{GuckenheimerHolmes1993,Kuznetsov1998}. 
By contrast, computing normal forms for codimension-3 or higher 
(degenerate) BT bifurcations is substantially more involved, particularly 
when aiming to derive explicit relationships between the original 
system parameters and the normal form.

A common approach in the literature for computing the normal forms of 
BT bifurcations employs multi-step transformations.
For example, Dumortier {\it et al.}~\cite{Dumortier1987} proposed 
a six-step transformation method, later refined in~\cite{HLY2018}, 
to handle degenerate cusp BT bifurcations. These transformations can 
be viewed as {\it backward transformations}, since the unfolding parameters in 
the normal form are expressed in terms of the original system parameters. 
At each step, only dominant terms are retained, and subdominant terms are 
discarded.

Recently, we developed a novel one-step transformation approach, based on 
the SNF theory, for analyzing higher-dimensional BT bifurcations. 
This method functions as a {\it forward transformation}, as it 
expresses the original system parameters in terms of the new bifurcation 
(unfolding) parameters of the normal form. It eliminates all nonessential 
terms up to a desired order, bypassing the need to first compute the 
classical normal form (CNF) before further simplification.
The SNF theory combines state transformations, parameter reparametrizations, 
and time rescalings to obtain the simplest possible form of 
a dynamical system. When all three are applied, the resulting normal form 
is simpler than those produced by other known techniques.

For example, consider the CNF of a dynamical system up to $5$th-order: 
\begin{equation}\label{Eqn1} 
\begin{array}{ll}  
\dfrac{d x_1}{dt} = x_2, \\[1.5ex] 
\dfrac{d x_2}{dt} =  a_2 x_1^2 + b_2 x_1 x_2 
+a_3 x_1^3 + b_3 x_1^2 x_2
+a_4 x_1^4 + b_4 x_1^3 x_2 
+a_5 x_1^5 + b_5 x_1^4 x_2,
\end{array} 
\end{equation}   
where we assume $a_2=0$ and $b_2 a_3 \ne 0$. Then, applying the 
state transformation 
$$
\begin{array}{rl}
x_1 = \!\!\! & y_1 -\dfrac{a_4}{5 a_3}\, y_1^2
-\dfrac{1}{100 a_3^2 b_2} (25 a_3^2 b_4-20 a_3 a_4 b_3+2 a_4^2 b_2)\, y_1^3
\\[2.0ex]
& -\, \dfrac{1}{500 a_3^3 b_2^2} (100 a_3^3 b_2 b_5-125 a_3^3 b_3 b_4
-75 a_3^2 a_4 b_2 b_4 +100 a_3^2 a_4 b_3^2-10 a_3 a_4^2 b_2 b_3
+2 a_4^3 b_2^2)\, y_1^4
\\[2.0ex] 
& +\, \dfrac{1}{50000 a_3^4 b_2^2} (5625 a_3^4 b_4^2+6400 a_3^3 a_4 b_2 b_5
-17000 a_3^3 a_4 b_3 b_4-300 a_3^2 a_4^2 b_2 b_4+10000 a_3^2 a_4^2 b_3^2
\\[2.0ex] 
& \qquad \qquad \qquad -\, 4240 a_3 a_4^3 b_2 b_3+708 a_4^4 b_2^2)\, y_1^5, 
\\[0.0ex]
x_2 = \!\!\! & y_2, 
\end{array} 
$$
and the time rescaling, $t = (1 + t_1) \tau$, with 
$$ 
\begin{array}{rl}
t_1 = \!\!\! & \dfrac{2 a_4}{5 a_3}\, y_1
+\dfrac{1}{100 b_2 a_3^2}(75 a_3^2 b_4-60 a_3 a_4 b_3+22 a_4^2 b_2)\, y_1^2
\\[2.0ex]
& +\dfrac{1}{125 b_2^2 a_3^3} (100 a_3^3 b_2 b_5-125 a_3^3 b_3 b_4
 +100 a_3^2 a_4 b_3^2-70 a_3 a_4^2 b_2 b_3+16 a_4^3 b_2^2)\, y_1^3, 
\end{array}
$$ 
yields the SNF:
\begin{equation}\label{Eqn2} 
\begin{array}{ll}  
\dfrac{d y_1}{d \tau} = \!\!\! & y_2, \\[1.5ex] 
\dfrac{d y_2}{d \tau} = \!\!\! & c_{11}\, y_1 y_2 + c_{30} \, y_1^3 
+ c_{21}\, y_1^2 y_2 + c_{50}\, y_1^5, 
\end{array} 
\end{equation}   
with 
$$ 
c_{11} = b_2, \quad c_{30} = a_3, \quad c_{21} = b_3 
- \dfrac{3 a_4 b_2}{5 a_3}, \quad 
c_{50} = a_5+ \dfrac{3 (4 a_4 b_3-5 a_3 b_4)}{10 b_2}
-\dfrac{24 a_4^2}{25 a_3}. 
$$ 
Kuznetsov~\cite{Kuznetsov2005} presents the SNF up to $3$rd order, along 
with its unfolding, the parametric SNF (PSNF),  
\begin{equation}\label{Eqn3} 
\begin{array}{ll}  
\dfrac{d y_1}{d \tau} = \!\!\! & y_2, \\[1.5ex]
\dfrac{d y_2}{d \tau} = \!\!\! & \beta_1 + \beta_2\, y_1 + \beta_3\,y_2  
+ c_{11}\, y_1 y_2 + c_{30} \, y_1^3 + c_{21}\, y_1^2 y_2, \\[1.5ex] 
\end{array}
\end{equation}   
where $\beta_k$, for $k=1,2,3$, are unfolding parameters. 

\begin{remark}\label{Rem1.1}
The normal form \eqref{Eqn3} has become the standard representation of
the codimension-three BT bifurcation and is widely used in the dynamical
systems literature (see, for example,~\cite{Baer2006,Perez2019,Saha2022}).
However, since the PSNF~\eqref{Eqn3} is truncated at third order,
the term $c_{21} y_1^2 y_2$ can be omitted without affecting the
qualitative bifurcation structure. This yields a simplified PSNF,
which we propose for the study of codimension-three BT bifurcations.
In particular, this simplification may lead to new bifurcation analyses
and simplier formulas for the codimension-three BT bifurcation problem.
\end{remark}

In contrast, a more general 5th-order SNF of \eqref{Eqn1} can be written as
\begin{equation}\label{Eqn4} 
\begin{array}{ll}  
\dfrac{d y_1}{d \tau} = \!\!\! & y_2, \\[1.5ex] 
\dfrac{d y_2}{d \tau} = \!\!\! & c_{11}\, y_1 y_2 + c_{30} \, y_1^3 
+ c_{40}\, y_1^4 + c_{50}\, y_1^5, 
\end{array} 
\end{equation}  
obtained via the transformation 
$$
\begin{array}{rl}
x_1 = \!\!\! & y_1 - \dfrac{b_3}{3 b_2}\, y_1^2
+\dfrac{1}{36 b_2^2} (10 b_3^2-9 b_2 b_4)\, y_1^3
-\dfrac{1}{270 b_2^3} (54 b_2^2 b_5-135 b_2 b_3 b_4+80 b_3^3)\, y_1^4
\\[2.0ex]
& +\, \dfrac{1}{3600 b_2^4} \big[ (768 b_2^2 b_3 b_5-b_4 b_2 
(2100 b_3^2-405 b_2 b_4)+980 b_3^4 \big]\, y_1^5,
\end{array} 
$$  
and time scaling, $t = (1 + t_1) \tau$ with 
$$ 
\begin{array}{rl}
t_1 = \!\!\! & \dfrac{2b_3}{3 b_2}\, y_1
-\dfrac{1}{36 b_2^2} (14 b_3^2-27 b_2 b_4)\, y_1^2
+\dfrac{1}{135 b_2^3} \big[ 108 b_2^2 b_5+5 b_3 (10 b_3^2-27 b_2 b_4) \big]
\, y_1^3. 
\end{array}
$$
In general, the SNF up to arbitrary order $k$ is 
\begin{equation}\label{Eqn5} 
\begin{array}{ll}  
\dfrac{d y_1}{d \tau} = \!\!\! & y_2, \\[1.5ex] 
\dfrac{d y_2}{d \tau} = \!\!\! & c_{11}\, y_1 y_2 
+ \displaystyle\sum_{j=3}^k c_{j0} \, y_1^j.  
\end{array} 
\end{equation}  
The corresponding PSNF up to third order is 
\begin{equation}\label{Eqn6} 
\begin{array}{ll}  
\dfrac{d y_1}{d \tau} = \!\!\! & y_2, \\[1.5ex]
\dfrac{d y_2}{d \tau} = \!\!\! & \beta_1 + \beta_2\, y_1 + \beta_3\, y_2 
+ c_{11}\, y_1 y_2 + c_{30} \, y_1^3. 
\end{array}
\end{equation}
It should be noted that an additional rescaling allows one to 
normalize either $c_{11}$ or $c_{30}$ to $\pm 1$; 
however, it is not possible to normalize both simultaneously.


We illustrate the computation of BT and GH codimensions using two 
predator–prey models from Bazykin’s classical monograph~\cite{Bazykin1998}.
For codimension-3 BT bifurcations, we apply our one-step SNF/PSNF approach.

\vspace*{0.10in} 
\noindent
{\bf Model 1} (Eq. (3.4.16) in~\cite{Bazykin1998}): 
$$
\begin{array}{ll}  
\dfrac{dx}{dt} = \!\!\! & a x \, \left(1 - \dfrac{x}{K} \right) - b x y, 
\\[2.0ex] 
\dfrac{dy}{dt} = \!\!\! & -\,c\,y + d\,x \, \dfrac{y^2}{N+y}, 
\end{array}
$$ 
which transforms into the dimensionless form 
(Eq. (3.4.17) in~\cite{Bazykin1998}):  
\begin{equation}\label{Eqn7} 
\begin{array}{ll}  
\dfrac{d X}{d \tau} = \!\!\! & X \big(1 - Y - e \, X \big), \\[2.0ex] 
\dfrac{d Y}{d \tau} = \!\!\! & Y \left( -\,g + \dfrac{X Y}{n+Y} 
\right), 
\end{array}
\end{equation}   
via $ x = \frac{a}{d}\,X$, $y=\frac{a}{b}\,Y$ and 
$ t = \frac{\tau}{a}$, where the new parameters are 
$$ 
g = \dfrac{c}{a}, \quad e = \frac{a}{dK}, \quad n = \dfrac{bN}{a}. 
$$

\vspace*{0.10in} 
\noindent 
{\bf Model 2} (Eq. (3.5.3) in~\cite{Bazykin1998}):
$$
\begin{array}{ll}  
\dfrac{dx}{dt} = \!\!\! & \alpha x - \epsilon x^2 - \dfrac{bxy}{1 + Ax}, 
\\[2.0ex] 
\dfrac{dy}{dt} = \!\!\! & -\,c\,y + \dfrac{\delta  xy}{1 + Ax} - h y^2, 
\end{array}
$$ 
which nondimensionlizes to 
\begin{equation}\label{Eqn8} 
\begin{array}{ll}  
\dfrac{d X}{d \tau} = \!\!\! & X \Big(1 -e\, X - \dfrac{Y}{1+a X} \Big), 
\\[2.0ex] 
\dfrac{d Y}{d \tau} = \!\!\! & Y \Big( -\,g + \dfrac{X}{1+a X} - d Y \Big), 
\end{array}
\end{equation}   
via $ x = \frac{\alpha}{\delta}\, X$, $y= \frac{\alpha}{b}\, Y$ and 
$t=\frac{\tau}{\alpha}$, where 
$$ 
g = \dfrac{c}{\alpha}, \quad a = \dfrac{A \alpha}{\delta}, \quad 
e = \dfrac{\epsilon}{\delta}, \quad d = \dfrac{h}{b}. 
$$ 

While the phase portraits of these two models were analysed in 
\cite{Bazykin1998} using linear analysis, our focus is on 
identifying and classifying saddle-node, BT and GH bifurcations. 
These bifurcations often arise even in simple ecological models 
and require explicit normal form computations. In the next section, we analyze 
system \eqref{Eqn7}, followed by system \eqref{Eqn8} in section 3. 
Concluding remarks are presented in Section 4.

\section{Bifurcation Analysis of System \eqref{Eqn7}}

The equilibrium solutions of system \eqref{Eqn7} are found 
by setting $\frac{d X}{d \tau}=\frac{d Y}{d \tau}=0$, yielding
\begin{equation}\label{Eqn9}
\begin{array}{ll} 
\textrm{Trivial equilibrium:} & {\rm E_0}= (0,\,0), \\[1.0ex] 
\textrm{Boundary equilibrium:} & {\rm E_1}= \Big( \dfrac{1}{e},\ 0 \Big),
\\[2.0ex]
\textrm{Positive equilibrium:} & {\rm E_2}= (X_2,\, 1- e X_2),
\end{array} 
\end{equation} 
where $X_2$ satisfies the quadratic equation,
\begin{equation}\label{Eqn10}
F_1(X_2) = e\, X_2^2-(e g+1)\, X_2+ g (1+n). 
\end{equation}
To facilitate the stability and bifurcation analysis of ${\rm E_2}$, 
we employ the method of {\it hierarchical parametric analysis}
\cite{ZengYu2023,ZengYuHan2024}. The core idea is to treat $X_2$ 
as a parameter and solve the equation $F_1 = 0$ for $n$, 
exploiting the fact that $n$ appears linearly. This gives 
$$ 
n = \dfrac{(1-e X_2) (X_2-g)}{g}.  
$$   
To ensure that $Y_2=1-e X_2>0$ and $n>0$, we require that  
$$ 
e < \dfrac{1}{X_2}, \quad g < X_2. 
$$ 

We summarize the conditions for existence and stability of 
the equilibria in the following lemma.

\begin{lemma}\label{Lem1}
The equilibria ${\rm E_0}$ and ${\rm E_1}$ exist for all 
positive parameter values. The positive equilibrium 
${\rm E_2}$ exits if $ X_2>0$, $g < X_2$, $e < \frac{1}{X_2}$, 
and $n=\frac{(1-e X_2) (X_2-g)}{g}$. 
\begin{itemize}
\item 
${\rm E_0}$ is always a saddle point.
\item 
${\rm E_1}$ is always a stable node.
\item 
${\rm E_2}$ is locally asymptotically stable $({\rm LAS})$ if
$$
g < X_2 \quad \text{and} \quad F_2(g) \stackrel{\triangle}{=}
g^3 - 3 X_2 g^2 + 2 X_2^2 g - X_2^2 < 0,
$$
which holds either when $X_2 \le \frac{3\sqrt{3}}{2}$, or when 
$X_2 > \frac{3\sqrt{3}}{2}$ and $g \in (0, g_1) \cup (g_2, X_2)$, 
where $g_1$ and $g_2$ are the two positive real roots of $F_2(g) = 0$.
\end{itemize}
\end{lemma}

\begin{proof}
The existence of the equilibria follows directly from \eqref{Eqn9} 
and the positivity conditions derived above.

To determine their stability, we compute the Jacobian matrix of 
\eqref{Eqn7}, given by 
\begin{equation}\label{Eqn11} 
J(X,Y) = \left[
\begin{array}{cc}
1 - Y - 2 \, e X & -\,X \\[1.0ex] 
\dfrac{Y^2}{n+Y} & -g + \dfrac{X Y(2n+Y)}{n+Y)^2}  
\end{array} 
\right]. 
\end{equation} 

Evaluating at each equilibrium:

\text{At ${\rm E_0}$}: Substituting $(0,0)$ yields eigenvalues 
$1$ and $-g$, so ${\rm E_0}$ is a saddle point.

\text{At ${\rm E_1}$}: Substituting $\left(\frac{1}{e}, 0\right)$ 
gives eigenvalues $-1$ and $-g$, implying ${\rm E_1}$ is a stable node.

\text{At ${\rm E_2}$}: Letting ${\rm E_2} = (X_2, Y_2 = 1 - e X_2)$, 
we compute the trace and determinant of the Jacobian:
\[
\operatorname{tr}(J({\rm E_2})) = -\dfrac{e X_2^2 - g X_2 + g^2}{X_2}, \quad
\det(J({\rm E_2})) = -2g (2e X_2 - g e - 1).
\]

By the Routh–Hurwitz criterion, ${\rm E_2}$ is LAS if
\[
\operatorname{tr}(J({\rm E_2})) < 0 \quad \text{and} \quad 
\det(J({\rm E_2})) > 0,
\]
which translates to that
\[
g < X_2, \qquad \dfrac{g (X_2 - g)}{X_2^2} < e < \dfrac{1}{2 X_2 - g}.
\]
This second inequality is equivalent to requiring $F_3(g) < 0$, where
\[
F_3(g) = g^3 - 3 X_2 g^2 + 2 X_2^2 g - X_2^2,
\]
which has the following properties:
\[
\begin{aligned}
& \lim_{g \to -\infty} F_3(g) = -\infty, \quad 
\lim_{g \to +\infty} F_3(g) = +\infty, \\[1.0ex]
& F_3(0) = F_3(X_2) = -X_2^2 < 0, \\[1.0ex] 
& \left. \dfrac{dF_3}{dg} \right|_{g = 0} = 2 X_2^2 > 0, \quad
\left. \dfrac{dF_3}{dg} \right|_{g = X_2} = -X_2^2 < 0.
\end{aligned}
\]

These results imply that $F_3(g)$ has a local minimum at
$g_{\min} = \big(1 - \frac{1}{\sqrt{3}}\big) X_2 < X_2$,
and a local maximum at $g_{\max} = \big(1 + \frac{1}{\sqrt{3}}\big) 
X_2 > X_2$. The maximum value of $F_3$ is
\[
F_{3,\max} = \dfrac{\sqrt{3}}{9} X_2^2\, (2 X_2 - 3\sqrt{3}),
\]
which is non-positive when $X_2 \le \frac{3\sqrt{3}}{2}$.

Thus, if $X_2 \le \frac{3\sqrt{3}}{2}$, then $F_3(g) < 0$ for all 
$g \in (0, X_2)$, ensuring that ${\rm E_2}$ is LAS. 
If $X_2 > \frac{3\sqrt{3}}{2}$, then $F_3(g)$ has two positive roots 
$0 < g_1 < g_2 < X_2$, and $F_3(g) < 0$ holds 
for $g \in (0, g_1) \cup (g_2, X_2)$.

The proof is complete.
\end{proof}

Next, we consider bifurcations of system \eqref{Eqn7} from the 
equilibrium ${\rm E_2}$, first saddle-node (SN) bifurcation, 
then BT bifurcation, followed by Hopf bifurcation.

\subsection{SN bifurcation}

The saddle-node (SN) bifurcation occurs when the discriminant of 
$F_1$, given by 
$$
\Delta_1 = (1-eg)^2 - 4 eg n =0 \quad \Longleftrightarrow \quad 
n_c = \dfrac{(1-eg)^2}{4 eg}, 
$$ 
defining the critical value $n_c$ at which the two roots coincide,
yielding the equilibrium:
$$
{\rm E_{sn}} = (X_{\rm 2sd},\, Y_{\rm 2sd}) = \Big(\dfrac{1+eg}{2e},\, 
\dfrac{1-eg}{2} \Big), \quad e<\dfrac{1}{g},  
$$
at which the Jacobian matrix has a zero eigenvalue. 

We now state the following result. 

\begin{theorem}\label{Thm1}
System \eqref{Eqn7} undergoes a saddle-node bifurcation at the 
critical point $n=n_c$. A cusp bifurcation does not occur. 
\end{theorem}

\begin{proof}
The trace of the Jacobian matrix evaluated at $\mathrm{E_{sn}}$ is 
$$ 
{\rm tr}(J({\rm E_{sn}})) = -\,\dfrac{1}{2(1+eg)} 
\left[ g^2 e^2+2 g (g+1) e+1-2 g\right].  
$$ 
Define the critical value, 
\begin{equation}\label{Eqn12}
e_c =  \dfrac{1}{g} \left[-(g+1)+ \sqrt{g (g+4)} \right], \quad 
\textrm{for} \ \  g> \dfrac{1}{2}. \\[1.0ex] 
\end{equation}
We have that $e_c < \frac{1}{g}$ for $g>\frac{1}{2}$ since
$$ 
e_c < \dfrac{1}{g} \ \ \Longleftrightarrow \ \ 
\dfrac{1}{g} \left[-(g+1)+ \sqrt{g (g+4)} \right] < \dfrac{1}{g}
\ \ \Longleftrightarrow \ \ \sqrt{g (g+4)}<2+g 
\ \ \Longleftrightarrow \ \ 0<4. 
$$  
Then, the sign of the trace is determined as follows:
$$ 
{\rm tr}(J({\rm E_{sn}})) \left\{
\begin{array}{ll}
< 0, & \textrm{if} \ \left\{
 \begin{array}{ll} 
  g \le \dfrac{1}{2}, & e < \dfrac{1}{g}, \\[2.0ex]  
  g > \dfrac{1}{2},   & e_c<e< \dfrac{1}{g}, 
 \end{array} 
\right. 
\\[5.0ex] 
> 0, & \textrm{if} \ \ g>\dfrac{1}{2}, \ \, e<e_c, 
\\[2.0ex] 
= 0, & \textrm{if} \ \ g>\dfrac{1}{2}, \ \, e=e_c, 
\end{array} 
\right. 
$$ 

When ${\rm tr}(J({\rm E_{sn}}))>0$, the degenerate node is unstable.
When $e=e_c$, the system undergoes a BT bifurcation. 
When ${\rm tr}(J({\rm E_{sn}}))<0$, for which 
$g>\frac{1}{2}, \, e<e_c$, we proceed to calculate the center manifold 
and derive the reduced dynamics.

To analyze the SN bifurcation, let  
$$ 
n = n_c + \mu, 
$$ 
and apply the affine transform,
\begin{equation}\label{Eqn13}
\left(\begin{array}{c} X \\ Y \end{array} \right) 
= \left(\begin{array}{c} \dfrac{1+eg}{2e} \\[2.0ex] 
 \dfrac{1-eg}{2} \end{array} \right) 
\left[
\begin{array}{cc}
-\dfrac{1}{2} (1 + eg) & -\dfrac{1 + eg}{2e} \\[2.0ex] 
\dfrac{eg(1-e g)}{1+e g} & \dfrac{g(1-e g)}{1+e g} 
\end{array} 
\right] 
\left(\begin{array}{c} u \\ v \end{array} \right). 
\end{equation}
Substituting this into \eqref{Eqn7}, we obtain the transformed system, 
\begin{equation}\label{Eqn14}
\begin{array}{rl} 
\dfrac{du}{d \tau} = \!\!\!\!& 
- \dfrac{4 e g^2}{C_g} \Big(\mu-\dfrac{4 e g}{1-e^2 g^2} \, \mu^2 \Big)  
-\dfrac{4 e g^2}{C_g}\, \mu\, u  
+\dfrac{4 e g^2 [C_g - 2 g(1-e g)]}{C_g (1+eg)^2}\, \mu \, v   
\\[2.5ex] 
& -\dfrac{g (1+e g)^2}{C_g}  u^2 
+ \dfrac{g [(eg+4g-3)C_g-8g(eg^2+eg-g+1)]}{C_g (1+e g)} uv 
\\[2.5ex] 
& - \dfrac{g(1-eg)\{ C_g[C_g-2g(eg+2g-3)]+8g^2(eg^2+eg-g+1)\}}{C_g (1+e g)^3} 
\,v^2 
\\[2.5ex] 
\dfrac{dv}{d \tau} = \!\!\!\!& 
-\dfrac{C_g}{2(1 \!+\! e g)}\, v 
+\dfrac{4 e g^2}{C_g} \Big( \mu + \dfrac{4 e g}{1 \!-\! e^2 g^2}\, \mu^2 \Big)
\!+\! \dfrac{4 e g^2}{C_g} \, \mu \, u 
\!-\! \dfrac{4 e g^2 [C_g \!-\! 2g(1 \!-\! eg)]}{C_g (1+eg)^2} \, \mu \, v,  
\\[2.5ex]
& +\dfrac{g (1+e g)^2}{C_g}\, u^2 
+ \dfrac{[C_g+4g(1-eg)][(1+eg)^2+4eg^2]}{C_g (1+eg)}\, u\,v 
\\[2.5ex] 
& + \dfrac{C_g^2+4g(1-eg)(1+eg)^2 [C_g + g (1+eg)]}
   {2 C_g (1+eg)^3} \, v^2 
\end{array} 
\end{equation} 
where 
$$ 
C_g = g^2 e^2 + 2g(g+1) e + 1 - 2g >0, \quad \textrm{for} \ \ 
 g>\dfrac{1}{2} \ \ \textrm{and} \ \ e<e_c.  
$$ 

Using center manifold theory, let 
$$ 
v = h(u,\mu) = h_{01} \, \mu + h_{20}\, u^2 + h_{11}\, u \mu + h_{02}\, \mu^2 
+ O(|u,v\mu|^3),
$$  
defining the center manifold, 
$$ 
W^c = \big\{(u,v)| v = h(u,\mu) \big\}. 
$$ 
Substituting the equation 
$\frac{dv}{d \tau} = (2 h_{20} u + h_{11} \mu ) \frac{du}{d\tau}$
and matching coefficients yield the following solutions:
$$ 
\begin{array}{rl}
h_{01} = \!\!\!\! & \dfrac{8 e g^2 (1 \!+\! e g)}{C_g^2}, \quad 
h_{20} = \dfrac{2 g (1 \!+\! e g)^3}{C_g^2}, \quad 
h_{11} = \dfrac{16 e g^2(1 \!+\! eg)^2 
[C_g (eg \!+\! 3g \!+\! 1) \!+\! 8g^2(1 \!-\! eg)]}{C_g^4}, \\[2.5ex] 
h_{02} = \!\!\!\! &  - \dfrac{32 e^2 g^3 (1 \!+\! e g)^3
      [ C_g^2 (eg \!+\! 8g \!+\! 1)-4gC_g (eg^2\!+\! 4eg \!+\! 10g^2
\!-\! g \!+\! 4)+80g^4(eg \!+\! 2 e \!-\! 1) 
      ]}{(1-e g) C_g^6}. 
\end{array} 
$$ 

Then, submitting the center manifold expression into 
the $u$-equation in \eqref{Eqn14}, we obtain 
$$ 
\dfrac{du}{d \tau} = -\dfrac{4 e g^2}{C_g}\, \mu 
- \dfrac{g (1+e g)^2}{C_g} \, u^2 + O(|u,\mu|^3),  
$$  
which confirms that system \eqref{Eqn7} exhibits a saddle-node bifurcation. 
Since the coefficient of $u^2$ is negative, a cusp bifurcation does 
not occur.
\end{proof}

\subsection{BT bifurcation}

In this section, we consider the Bogdanov-Takens (BT) bifurcation. 
As shown in the previous section, this bifurcation occurs 
at the critical point, defined by 
\begin{equation}\label{Eqn15}
n=n_c = \dfrac{(1-eg)^2}{4eg}, \quad 
e=e_c = \dfrac{1}{g} \left[-(g+1)+ \sqrt{g (g+4)} \right], \quad 
\textrm{for} \ \  g> \dfrac{1}{2}, \\[1.0ex] 
\end{equation} 
at which the equilibrium becomes 
\begin{equation}\label{Eqn16}
(X_{\rm 2 bt_2},\,Y_{\rm 2 bt_2}) = 
\left(\dfrac{g\, [\,3 g+\sqrt{g (g+4)}\,]}{2 (2 g-1)},\, 
\dfrac{2+g-\sqrt{g (g+4)}}{2} \right). 
\end{equation} 
We first derive the standard normal form to determine the codimension 
of the BT bifurcation, followed by the derivation of the parametric 
normal form (PSNF) to analyze the bifurcation structure in detail. 
In \cite{Kuznetsov2005},Kuznetsov discussed that the system exhibits 
a codimension-three BT bifurcation and derived the corresponding 
normal form without considering its unfolding. In this section, 
we rigorously establish the existence of the codimension-three 
BT bifurcation and present a systematic and detailed bifurcation 
analysis based on PSNF theory.

\subsubsection{Codimension of BT bifurcation} 

\begin{theorem}\label{Thm2}
System \eqref{Eqn7} undergoes a Bogdanov-Takens bifurcation from the 
equilibrium $E_{\rm bt}=(X_{\rm 2 bt_2},\,Y_{\rm 2 bt_2})$ at the critical 
point $(n,e)=(n_c,e_c)$. More precisely, 
\begin{enumerate}
\item[{$(1)$}] 
The bifurcation has codimension two for 
$g \in (\frac{1}{2},\frac{4}{3}) \bigcup (\frac{4}{3},\infty)$. 

\item[{$(2)$}] 
The bifurcation has codimension three at 
$g = \frac{4}{3}$, indicating a degenerate $($cusp$)$ BT bifurcation. 
\end{enumerate}  
\end{theorem}

\begin{proof} 
At the critical point defined in \eqref{Eqn15}, we apply the 
affine transformation \eqref{Eqn13} with $e=e_c$, yielding the system 
\eqref{Eqn7} expanded to second order: 
$$
\begin{array}{rl}
\dfrac{du}{d \tau} = \!\!\!\! & 
\dfrac{(4g+1)\sqrt{g (g+4)}+g (4g+7)} 
{[\,3 g+ \sqrt{g (g+4)}\,] [\,g+1+\sqrt{g (g+4)}\,]} \, (v - uv), \\[2.5ex] 
\dfrac{dv}{d \tau} = \!\!\!\! & 
\dfrac{4g}{[\,3 g+\sqrt{g (g+4)}\,] [\,g+1+\sqrt{g (g+4)}\,] 
           [(g+1) \sqrt{g (g+4)}-g (g-5)]^2 } 
\\[3.0ex] 
& \times 
\big\{ 
g^2 \big[(g+4)(g^2-10g-2) \sqrt{g (g+4)}-g^4 +4g^3 -72g^2 -52g-2 \big]\, u^2  
\\[1.0ex]  
& \quad - g \big[3(g^3 +15g^2-18g-5) \sqrt{g(g+4)} 
         -3g^4 +77g^3 -30g^2 -87g-4 \big] \, uv 
\\[1.0ex]
& \quad +\big[ (g^3 +57g^2 +30g +1) \sqrt{g(g+4)} 
      -g (g^3 -69g^2 -114g -17) \big]\, v^2 
\big\}. 
\end{array} 
$$ 

Next, we perform the transformation,
$$ 
u = y_1, \quad v = (y_1 + 1)\, y_2, 
$$ 
which brings the system to its SNF up to second order:
\begin{equation}\label{Eqn17}
\dfrac{d y_1}{d \tau} = y_2, \quad \dfrac{d y_2}{d \tau} 
= c_{20}\, y_1^2 + c_{11}\, y_1 y_2, 
\end{equation} 
where 
$$
c_{20} = -\dfrac{g \big[ \sqrt{g(g+4)}-g \big]}{2} < 0, \quad 
c_{11} = \dfrac{g}{2 g+\sqrt{g(g+4)}} \, (4-3 g), \quad 
\textrm{for} \ \ g > \dfrac{1}{2},  
$$
which implies that the only possibility for \( c_{11} = 0 \) is 
when \( g = \tfrac{4}{3} \). This corresponds to a codimension-three 
(degenerate cusp) BT bifurcation.
\end{proof}

\subsubsection{The PSNF for the codimension-two BT bifurcation} 

To analyze the codimension-two BT bifurcation (Case (1) in 
Theorem~\ref{Thm2}), we derive the parametric normal form (PSNF) 
corresponding to this case. To begin, let
\begin{equation}\label{Eqn18}
n=n_c + \mu_1, \quad e=e_c + \mu_2, 
\end{equation} 
and substitute this into \eqref{Eqn7}. This yields a perturbed system 
of the form of \eqref{Eqn14} with $e=e_c$, plus perturbed terms 
expanded up to second order:
$$ 
\begin{array}{ll} 
\dfrac{d u}{d \tau} = v + f_1 (u,v,\mu_1,\mu_2) 
+ O(|u,v,\mu_1,\mu_2|^3), \\[1.5ex] 
\dfrac{d v}{d \tau} = f_2 (u,v,\mu_1,\mu_2)
+ O(|u,v,\mu_1,\mu_2|^3). 
\end{array}
$$ 

Applying the nonlinear transformation, 
$$ 
\begin{array}{rl}
u = \!\!\!\!\! & \dfrac{1}{g^2 (g+4) (2 g-1) (\sqrt{g (g+4)}-g)}
\big\{ 2 g (g+4)(2g-1)\, y_1 -2 \big[ (g+4) (2 g+1) 
\\[2.5ex]
&\hspace*{1.6in} 
+\, (2 g+5) \sqrt{g (g+4)}\, \big]\, \beta_1 
- g^2 \big[ g+4+3 \sqrt{g(g+4)} \big]\, \beta_2 \big\}, 
\\[2.5ex] 
v = \!\!\!\!\! & \dfrac{2}{g^3 (g+4) (2g-1)^2 (\sqrt{g(g+4)}-g)^2}
\big\{g^2 \big[ (g\!+\! 4) (2g \!-\! 1)^2 (\sqrt{g(g \!+\!4)}-g) \big]\, y_2
\\[2.0ex]
& +\, 2 g (g \!+\! 4) (2 g \!-\! 1)^2\, y_1 y_2 
-\,4 g^2 (2 g \!-\! 1) \big[ (g \!-\! 2) \sqrt{g(g \!+\! 4)} 
\!-\! g (g \!+\! 4) \big] \beta_1
\\[1.0ex]
&-\, 2 g^4 (2 g \!-\! 1) \big[(g^2 \!+\! 2g \!-\! 2)
\sqrt{g(g \!+\! 4)} \!-\! g^2 (g \!+\! 4) \big] \beta_2
\\[1.0ex] 
& +\,16 \big[2 (g+1) \sqrt{g(g+4)}+2 g^2+6 g+1 \big]\, \beta_1^2
+8 g^2 \big[ 3 \sqrt{g (g+4)}+5 g+2 \big]\, \beta_1 \beta_2 
\\[1.0ex]
&-\,2 g^4 \big[(g-2) \sqrt{g(g+4)}-g^2-2 \big]\, \beta_2^2
-4 g (2 g-1) \big[ 3 \sqrt{g(g+4)}+g+4 \big] \, \beta_1 \, y_1 
\\[1.0ex] 
& -\,2 g^3 (2 g-1) \big[ (g+1) \sqrt{g(g+4)}-(g-1) (g+4) \big]\, \beta_2 y_1
\\[1.0ex]
&-\,2 (2 g-1) \big[(2g+5) \sqrt{g(g+4)}+(g+4)(2g+1)\big]\, \beta_1 y_2
\\[1.0ex]
& -\,g^2 (2 g-1) \big[ 3 \sqrt{g(g+4)}+g+4 \big]\, \beta_2 y_2 \big\}, 
\end{array} 
$$ 
along with the {\it forward} nonlinear parametrization, 
$$ 
\begin{array}{rl}
\mu_1 = \!\!\!\! & \dfrac{1}{2 g^5 (2 g-1)} \beta_1
\big\{2 g^3 \big[\sqrt{g(g+4)} + g+1 \big] 
\\[1.5ex] 
& \hspace{0.8in} 
+\big[g (g+1) (2 g+7)+(2 g^2+5 g+1) \sqrt{g(g+4)}\big] \beta_1 \big\},
\\[2.0ex] 
\mu_2 = \!\!\!\! &
-\dfrac{2 [g+4+\sqrt{g(g+4)}]}{g^3 (g+4)}\, \beta_1
         -\dfrac{g+4-\sqrt{g(g+4)}}{g (g+4)}\, \beta_2 
\end{array} 
$$ 
we obtain the PSNF:
\begin{equation}\label{Eqn19} 
\begin{array}{ll} 
\dfrac{d y_1}{d \tau} = y_2 + O(|y_1,y_2,\mu_1,\mu_2|^3), \\[2.0ex]
\dfrac{d y_2}{d \tau} = \beta_1+ \beta_2 y_2-y_1^2+ C_{11}\, y_1 y_2
+ O(|y_1,y_2,\mu_1,\mu_2|^3), 
\end{array}
\end{equation} 
where 
$$
C_{11} = \dfrac{g-4+ \sqrt{g(g+4)}}{2g} = 
\dfrac{2(3g-4)}{g(4-g+ \sqrt{g(g+4)}}
\left\{
\begin{array}{ll}
<0 & \textrm{if} \ \ g < \dfrac{4}{3}, \\[2.0ex] 
>0 & \textrm{if} \ \ g > \dfrac{4}{3}. 
\end{array} 
\right. 
$$ 
To eliminate the higher-order terms in the first equation of
\eqref{Eqn19}, we apply the additional transformation: 
$$
y_1 \rightarrow y_1, \quad 
y_2 + O(|y_1,y_2,\mu_1,\mu_2|^3) \rightarrow y_2.
$$ 
The resulting PSNF is in the standard form associated with the 
typical codimension-two BT bifurcation. Such a bifurcation may give 
rise to saddle-node (SN), Hopf, and homoclinic loop (HL) bifurcations. 
Further details can be found in, for example, 
\cite{GuckenheimerHolmes1993,HLY2018,Kuznetsov1998,HanYu2012,
YuZhang2019,ZengYuHan2024}, 
and are therefore omitted here.

Note that 
$$ 
\det \left[ \dfrac{\partial (\mu_1,\mu_2)}{\partial (\beta_1,\beta_2)} 
\right] 
= -\,\dfrac{\big[\,g+1+\sqrt{g(g+4)}\,\big] \big[\,g+4-\sqrt{g(g+4}\,\big]}
{g^3(2g-1)(g+4)} <0 \quad \textrm{for} \ \ g > \dfrac{1}{2}, 
$$ 
which indicates that system \eqref{Eqn19} with $(\beta_1,\beta_2)
\sim (0,0)$ for $(y_1,y_2)$ near $(0,0)$ is equivalent 
(in the sense of normal form theory) to system 
\eqref{Eqn7} with $(n,e)$ near $(n_c,e_c)$ for $(X,Y)$ 
near $(X_{\rm 2 bt_2}, Y_{\rm 2 bt_2})$.

\subsubsection{The PSNF for the codimension-three BT bifurcation} 

We now consider the codimension-three BT bifurcation 
(Case (2) in Theorem~\ref{Thm2}), for which $g=\frac{4}{3}$, yielding 
$$
n_c = \frac{1}{3}, \quad e_c = \frac{1}{4}, \quad 
X_{\rm 2 bt_3}=\dfrac{8}{3}, \quad Y_{\rm 2 bt_3}= \dfrac{1}{3}.
$$ 
Let 
\begin{equation}\label{Eqn20} 
n = \dfrac{1}{3} + \mu_1, \quad e=\dfrac{1}{4} + \mu_2, \quad 
g = \dfrac{4}{3} + \mu_3. 
\end{equation}
Introducing the affine transformation,
$$
\left(\begin{array}{c} X \\ Y \end{array} \right) 
= \left(\begin{array}{c} \dfrac{8}{3} \\[2.0ex] 
 \dfrac{1}{3} \end{array} \right) 
\left[
\begin{array}{rr}
-\dfrac{3}{2} & -\dfrac{8}{3} \\[2.0ex] 
\dfrac{1}{6} & \dfrac{2}{3} 
\end{array} 
\right] 
\left(\begin{array}{c} u \\ v \end{array} \right)  
$$ 
into system \eqref{Eqn7} and expanding the resulting system up to 
fourth-order terms, we obtain 
$$
\begin{array}{rl}
\hspace*{-0.80in} 
\dfrac{d u}{d \tau} = \!\!\!\! & 
\tfrac{1}{6144} \mu_2 (16384+6144 \mu_3+8352 \mu_3^2-8019 \mu_3^3)
\\[1.0ex]
& -\tfrac{1}{12288} \big[64 \mu_2 (1024+192 \mu_3+243 \mu_3^2)
-2592 \mu_3^2+2673 \mu_3^3 \big]\, u 
\\[1.0ex]
& +\tfrac{1}{32768} (32768+6144 \mu_3+7776 \mu_3^2-9477 \mu_3^3)\, v
+\tfrac{1}{384} (1024 \mu_2-81 \mu_3^2)\, u^2-u v, \hspace*{0.30in} 
\end{array}
$$ 
\begin{equation}\label{Eqn21} 
\begin{array}{rl}
\dfrac{d v}{d \tau} = \!\!\!\! &
-\tfrac{1}{49152} \mu_1 \big[ 32768+256 \mu_2 (32+15 \mu_3) (32+27 \mu_3)
                     +32768 \mu_2^2 (12-32 \mu_2+9 \mu_3) 
\\[1.0ex]
& +3 \mu_3 (224 \!+\! 45 \mu_3) (64 \!-\! 9 \mu_3) \big] 
+\tfrac{1}{256} \mu_1^2 
\big[ 256 \!+\! 768 \mu_2 (4 \!+\! 9 \mu_3) \!+\! 10752 \mu_2^2 
\!+\! 9 \mu_3 (64 \!+\! 33 \mu_3) \big]
\\[1.0ex] 
& -\tfrac{3}{32} \mu_1^3 (16+256 \mu_2+51 \mu_3) + \tfrac{9}{4}\, \mu_1^4
-\tfrac{9}{512} \mu_2 \mu_3^2 (32-27 \mu_3)
\\[1.0ex] 
& -\tfrac{1}{1152} \mu_2^2 (1024+2112 \mu_3+1413 \mu_3^2) 
-\tfrac{2}{9} \mu_2^3 (16+33 \mu_3)+ \tfrac{16}{9} \mu_2^4 
\\[1.0ex] 
& -\tfrac{1}{2304} \big\{3 \mu_1 
\big[512 \!-\! 1536 \mu_2 (16 \mu_2 \!+\! 3 \mu_3) 
+9 \mu_3 (64 \!+\! 15 \mu_3) \big]
+864 \mu_1^2 (32 \mu_2 \!+\! 3 \mu_3) \!-\! 3456 \mu_1^3
\\[1.0ex]
& -12 \mu_2 (1024+192 \mu_3+351 \mu_3^2)
+2304 \mu_2^2 (8+11 \mu_3)-8192 \mu_2^3 \big\} u
\\[1.0ex]
& -\tfrac{1}{12288}\big\{48 \mu_1 \big[ 512+1024 \mu_2 (2-9 \mu_2)
+9\mu_3(64-3\mu_3) \big] -1152 \mu_1^2 (16+15 \mu_3)
\\[1.0ex]
& -64 \mu_2 (1024+192 \mu_3+243 \mu_3^2)
+4096 \mu_2^2 (40+57 \mu_3)+81 \mu_3^2 (32-33 \mu_3)\big\}\,v
\\[1.0ex]
&+\tfrac{1}{1152}\big[144 \mu_1 (16 \!+\! 128 \mu_2 \!+\! 27 \mu_3)
\!-\! 512 \mu_2 (16 \!+\! 9 \mu_3) \!-\! 9 \mu_3 (64 \!-\! 75 \mu_3)
\!-\! 3456 \mu_1^2 \!+\! 3072 \mu_2^2 \big] u^2
\\[1.0ex]
& +\tfrac{1}{768} \big[192 \mu_1 (16 \!+\! 160 \mu_2 \!+\! 33 \mu_3)
\!-\! 1024 \mu_2 (8 \!+\! 9 \mu_3) \!-\! 9 \mu_3 (64 \!-\! 57 \mu_3) 
\!-\! 5760 \mu_1^2 \!-\! 4096 \mu_2^2 \big] u v
\\[1.0ex]
&+\tfrac{1}{32} \big[ 3 \mu_1 (16+256 \mu_2+51 \mu_3)-32 \mu_2 (4+9 \mu_3)
 -144 \mu_1^2-448 \mu_2^2 \big] v^2
\\[1.0ex] 
& -\tfrac{1}{9} (6 \mu_1-16 \mu_2+3 \mu_3) u^3
-\tfrac{1}{24} (48 \mu_1-128 \mu_2+39 \mu_3) u^2 v
 -\tfrac{1}{8} (12 \mu_1-32 \mu_2+21 \mu_3) u v^2
\\[1.0ex]
&-\tfrac{45}{32} \mu_3 v^3
+\tfrac{1}{72} (-64 u^2+72 v^2-64 u^3-240 u^2 v-288 u v^2-108 v^3
 +64 u^4+336 u^3 v
\\[1.0ex] 
& +648 u^2 v^2+540 u v^3+162 v^4).
\end{array} 
\end{equation}
Next, we apply a nonlinear transformation of the state variables,
\begin{equation}\label{Eqn22} 
\begin{array}{rl}
u = \!\!\!\! & -\tfrac{9}{8} y_1+\tfrac{243}{256} \beta_1 -\tfrac{9}{8}\beta_2
-\tfrac{459}{256} \beta_1 \beta_3
+\tfrac{405}{256} y_1 y_2 +\tfrac{729}{640} y_2^2
+\tfrac{2187}{640} y_1 \beta_1
-\tfrac{8991}{2048} y_2 \beta_1  \\[1.0ex] 
& + \cdots + O(|y_1,y_2,\beta_1,\beta_2|^5), \\[1.0ex] 
v = \!\!\!\! & -\tfrac{9}{8} y_2
+\tfrac{81}{128} y_1 \beta_1+ \tfrac{9}{8} y_1 \beta_2
+\tfrac{5103}{1024} y_2 \beta_1 + \tfrac{9}{32} \beta_2 y_2
-\tfrac{45}{16} \beta_3 y_1 y_2 
\\[1.0ex] 
& + \cdots + O(|y_1,y_2,\beta_1,\beta_2|^5), \\[1.0ex] 
\end{array} 
\end{equation}
where some terms are omitted for brevity, and the {\it forward} parametrization,
\begin{equation}\label{Eqn23} 
\begin{array}{rl}
\mu_1 = \!\!\!\! & 
\beta_1 \big( \tfrac{27}{16}-\tfrac{2187}{256} \beta_1-\tfrac{9}{2} \beta_3
-\tfrac{31246586811}{167772160} \beta_1^2
-\tfrac{81}{128} \beta_2^2+\tfrac{9}{2} \beta_3^2 
+\tfrac{22166703}{573440} \beta_1 \beta_2
+\tfrac{755805}{114688} \beta_1 \beta_3
\\[1.0ex]
& -\tfrac{9}{8} \beta_2 \beta_3
-\tfrac{8253243713943}{26843545600} \beta_1^3
+\tfrac{81}{256} \beta_2^3+\tfrac{16}{27} \beta_3^3
+\tfrac{35525453769}{234881024} \beta_1^2 \beta_2
+\tfrac{25882444593}{73400320} \beta_1^2 \beta_3
\\[1.0ex]
& -\tfrac{54271863}{2293760} \beta_1 \beta_2^2
+\tfrac{33}{8} \beta_2^2 \beta_3
+\tfrac{773123}{53760} \beta_1 \beta_3^2
+\tfrac{17}{12} \beta_2 \beta_3^2
-\tfrac{52105569}{573440} \beta_1 \beta_2 \beta_3 \big),
\\[1.0ex] 
\mu_2 = \!\!\!\! & 
\tfrac{81}{256} \beta_1 -\tfrac{3}{8} \beta_2 
-\tfrac{282123}{262144} \beta_1^2
+\tfrac{93}{256} \beta_2^2
+\tfrac{81}{4096} \beta_1 \beta_2
-\tfrac{153}{256} \beta_1 \beta_3
+\tfrac{11}{24} \beta_2 \beta_3
+\tfrac{664276823397}{2684354560} \beta_1^3
\\[1.0ex]
& -\tfrac{681}{4096} \beta_2^3
-\tfrac{8762499}{2293760} \beta_1^2 \beta_2
-\tfrac{101169}{28672} \beta_1^2 \beta_3
-\tfrac{23029353}{4587520} \beta_1 \beta_2^2
-\tfrac{147}{256} \beta_2^2 \beta_3
-\tfrac{617}{256} \beta_1 \beta_3^2
\\[1.0ex]
& -\, \tfrac{5}{216} \beta_2 \beta_3^2
+ \tfrac{12498617}{2150400} \beta_1 \beta_2 \beta_3
\\[1.0ex] 
\mu_3 = \!\!\!\! & 
\tfrac{27}{16} \beta_1 +\tfrac{2}{3} \beta_2+\tfrac{32}{27} \beta_3 
\tfrac{53055}{4096} \beta_1^2
+\tfrac{1}{4} \beta_2^2 +\tfrac{272}{243} \beta_3^2
-\tfrac{1161}{160} \beta_1 \beta_2
+\tfrac{51}{40} \beta_1 \beta_3
+\tfrac{8}{27} \beta_2 \beta_3, 
\\[1.0ex] 
\end{array} 
\end{equation}
to system \eqref{Eqn21}, we obtain the PSNF up to fourth-order terms:
\begin{equation}\label{Eqn24} 
\begin{array}{ll} 
\dfrac{d y_1}{d \tau} = y_2, \\[1.5ex]
\dfrac{d y_2}{d \tau} = \beta_1+\beta_2\,y_2+\beta_3\, y_1 y_2
+y_1^2- \tfrac{243}{256}\, y_1^3 y_2.
\end{array}
\end{equation} 
Moreover, a simple calculation yields
$$ 
\det\left[
\dfrac{\partial(\mu_1,\mu_2,\mu_3)}
{\partial(\beta_1,\beta_2,\beta_3)}
\right]_{\beta_1=\beta_2=\beta_3=0}
= -\, \dfrac{3}{4},
$$
which implies that system \eqref{Eqn24} with $(\beta_1,\beta_2,\beta_3) 
\sim (0,0,0)$ for $(y_1,y_2)$ near $(0,0)$ is topologically equivalent 
to system \eqref{Eqn7} with $(n,e,g)$ close to 
$(\frac{1}{3},\frac{1}{4},\frac{4}{3})$ for 
$(X,Y)$ near $(\frac{8}{3},\frac{1}{3})$.

Finally, by applying an additional transformation together with 
a time rescaling and reparametrization,
$$ 
y_1 \rightarrow \dfrac{8 \sqrt[5]{2}}{9}\, y_1, \quad 
y_2 \rightarrow \dfrac{16 \sqrt[5]{16}}{27}\, y_2, \quad 
\tau \rightarrow \dfrac{3 \sqrt[5]{4}}{4}\, \tau, \quad 
\beta_1 \rightarrow \dfrac{64 \sqrt[5]{4}}{81}\, \beta_1, \quad 
\beta_2 \rightarrow \dfrac{2 \sqrt[5]{8}}{3}\, \beta_2, \quad 
\beta_3 \rightarrow \dfrac{3 \sqrt[5]{4}}{4}\, \beta_3, 
$$
the coefficient of the term $y_1^3 y_2$ in \eqref{Eqn24} can be normalized, 
yielding the simplified form,  
\begin{equation}\label{Eqn25} 
\begin{array}{ll} 
\dfrac{d y_1}{d \tau}=y_2, \\[1.5ex]
\dfrac{d y_2}{d \tau}=\beta_1+\beta_2\,y_2+\beta_3\,y_1 y_2+y_1^2-y_1^3 y_2.
\end{array}
\end{equation}

\begin{figure}[!t] 
\vspace*{-1.45in}
\begin{center}
\hspace*{-3.70in}
\begin{overpic}[width=0.60\textwidth,height=0.60\textheight]{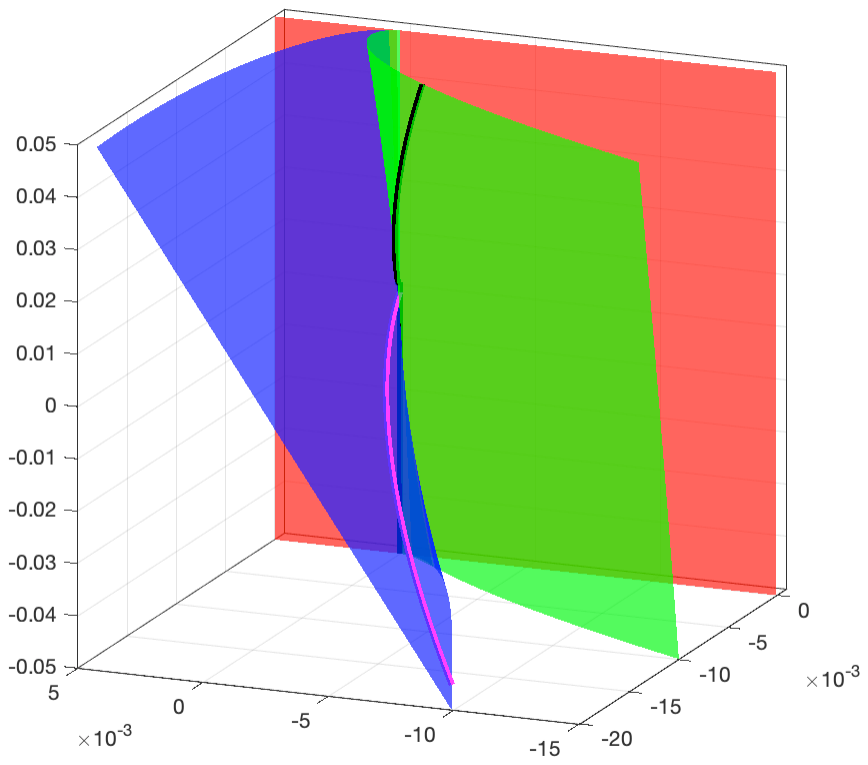}
\end{overpic}

\vspace*{-5.41in} 
\hspace*{2.20in}
\begin{overpic}[width=0.72\textwidth,height=0.61\textheight]{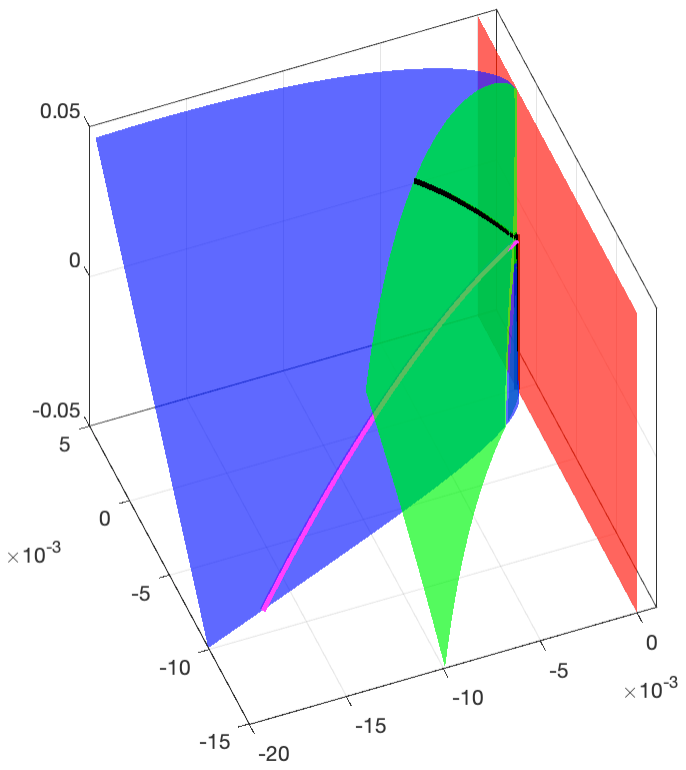}
\end{overpic} 

\vspace*{-1.85in} 
\hspace*{-1.50in} {$\beta_1$} 

\vspace*{-0.05in} 
\hspace*{-4.50in} {$\beta_2$} 

\vspace*{-1.35in} 
\hspace*{-6.30in} {$\beta_3$} 

\vspace*{-0.70in} 
\hspace*{ 0.30in} {$\beta_3$} 

\vspace*{ 0.90in} 
\hspace*{ 0.80in} {$\beta_2$} 

\vspace*{ 0.20in} 
\hspace*{ 3.50in} {$\beta_1$} 

\vspace*{0.05in} 
\hspace*{-0.32in}(a)\hspace*{2.95in}(b) 

\vspace*{ 0.7in} 
\hspace*{-3.60in}
\begin{overpic}[width=0.40\textwidth,height=0.27\textheight]{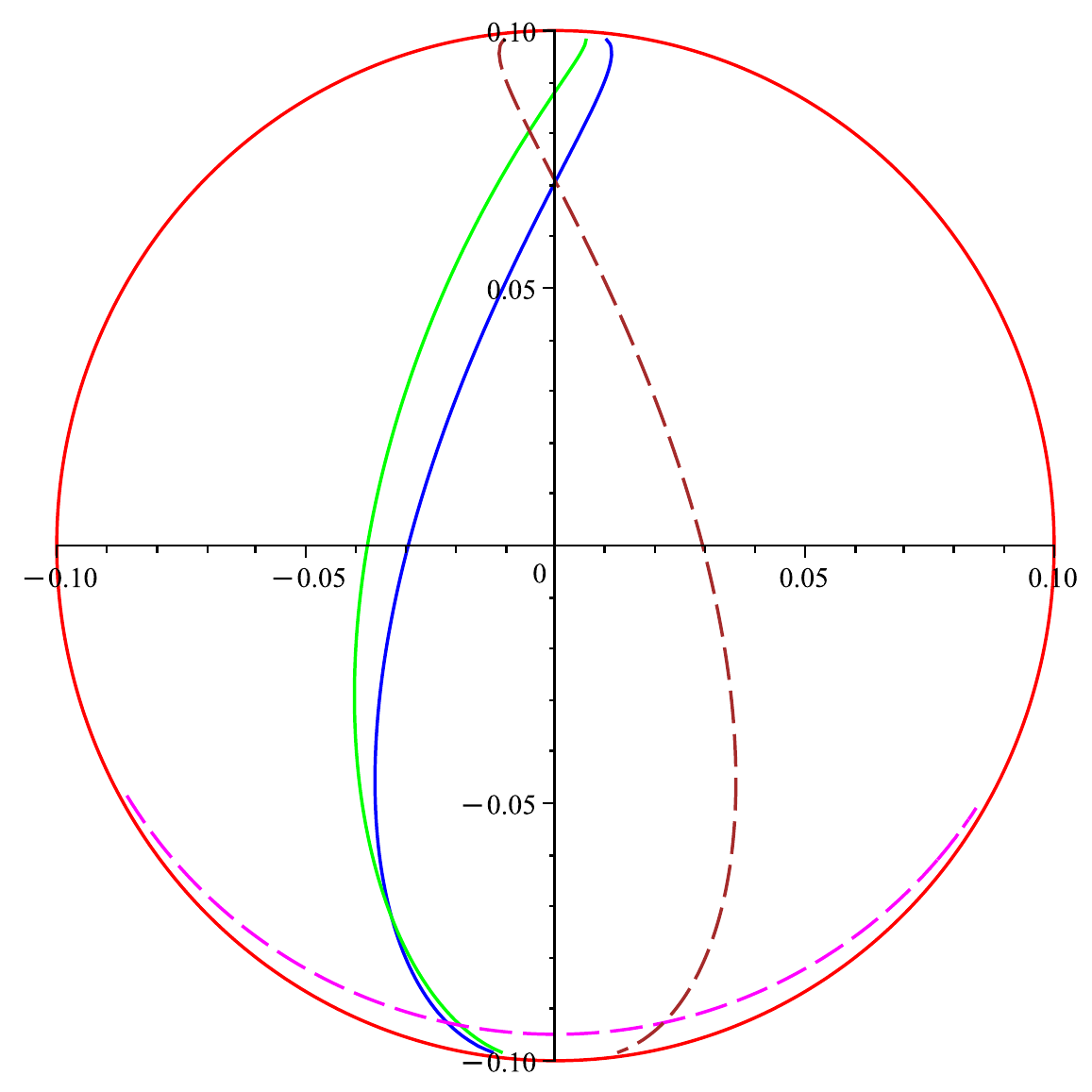}
\put(50.8,92.5){\vector(0,1){13.5}}
\put(96.8,47.7){\vector(1,0){13.5}}
\put(34.6,13.6){$\bullet$} 
\put(29.4,13.6){\scriptsize{C}}
\put(38.8, 4.7){$\bullet$} 
\put(40.0, 8.5){\scriptsize{GH}}
\put(47.0,82.5){$\bullet$} 
\put(35.0,82.5){\scriptsize{DHL}}
\put(49.3,91.5){$\bullet$} 
\put(44.0,96.0){\footnotesize{${\rm S_1}$}}
\put(49.2, 1.2){$\bullet$} 
\put(48.6,-5.0){\footnotesize{${\rm S_2}$}}
\put(110.0,42.0){\small{$\beta_2$}}
\put(54.0,106.0){\small{$\beta_3$}}
\put( 4.0,72.0){\scriptsize{SN}}
\put(92.0,72.0){\scriptsize{SN}}
\put(43.0,55.0){\scriptsize{H}}
\put(30.0,65.0){\scriptsize{HL}}
\end{overpic}

\vspace*{-2.37in} 
\hspace*{2.70in}
\begin{overpic}[width=0.369\textwidth,height=0.262\textheight]{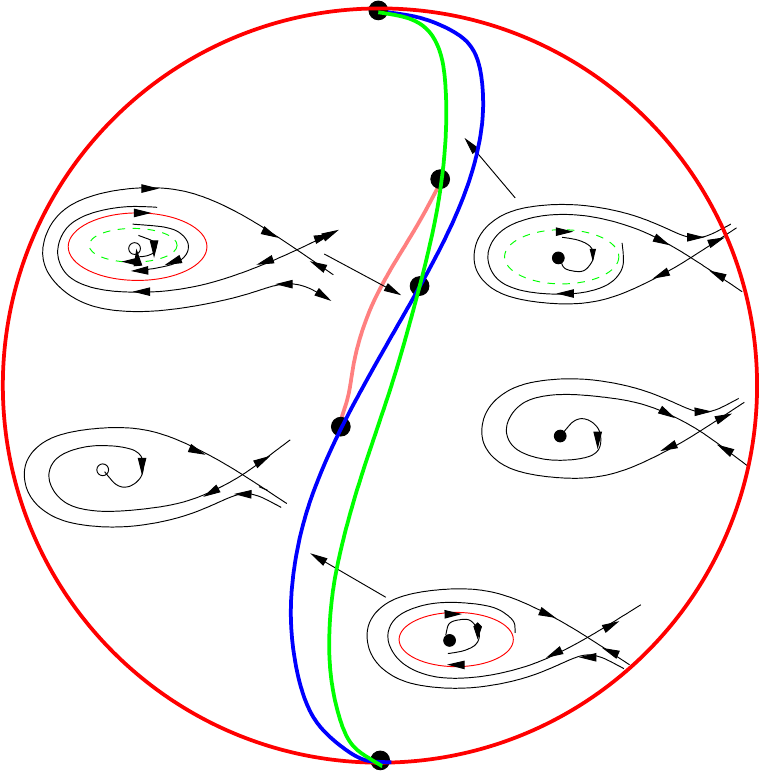}
\put(49.5,98.0){\vector(0,1){15.0}}
\put(99.5,49.8){\vector(1,0){15.0}}
\put(42.0,102.0){\footnotesize{${\rm S_1}$}}
\put(48.6,-7.0){\footnotesize{${\rm S_2}$}}
\put(114.0,44.0){\small{$\beta_2$}}
\put(53.0,112.5){\small{$\beta_3$}}
\put(44.0,76.0){\scriptsize{DHL}}
\put(58.0,59.0){\scriptsize{C}}
\put(35.0,47.0){\scriptsize{GH}}
\put(-1.0,76.0){\scriptsize{SN}}
\put(95.0,76.0){\scriptsize{SN}}
\put(33.0,27.0){\scriptsize{H}}
\put(47.0,27.0){\scriptsize{HL}}
\put(15.0,78.0){\scriptsize{DLC}}
\end{overpic}

\vspace*{0.20in} 
\hspace*{-0.32in}(c)\hspace*{2.95in}(d) 

\caption{Codimension-3 BT bifurcation diagrams based on the
normal form \eqref{Eqn25}. (a)–(b) Bifurcation surfaces: red indicates
saddle-node (SN) bifurcations, blue indicates Hopf bifurcations, and
green indicates homoclinic loop (HL) bifurcations. The red curve on the
blue surface marks the generalized Hopf (GH) bifurcation, while the blue
curve on the green surface corresponds to the degenerate homoclinic loop
(DHL) bifurcation. (c)–(d) Projections of the three-dimensional
bifurcation diagrams onto the cone intersecting the $2$-sphere
$\beta_1^2+\beta_2^2+\beta_3^2=\sigma^2$ with $\sigma=0.1$. In (c),
curves are obtained from the bifurcation formulas in
Theorem~\ref{Thm3}: the intersection of the pink and blue curves
indicates the GH bifurcation, while the intersection of the brown and
green curves marks the DHL bifurcation. (d) A schematic representation
of the bifurcation diagram with the corresponding phase portraits.}
\label{Fig1}
\end{center}
\vspace*{-0.20in} 
\end{figure}

\begin{remark}\label{Rem2.4}
The classical derivation of the PSNF~\eqref{Eqn24} for a codimension-three 
(degenerate cusp) BT bifurcation follows the six-step transformation 
method introduced by Dumortier \textit{et al.}~\cite{Dumortier1987} 
and widely used in the literature. This approach applies a sequence 
of {\it backward} transformations, at each step retaining dominant terms 
and eliminating one or two terms from the Taylor expansion of the 
vector field, possibly expressed in differential rather than algebraic form. 
However, it does not yield an explicit relation between the PSNF parameters 
and the original system’s parameters, making it difficult to match 
the dynamics directly. In contrast, our one-step {\it forward} 
transformation establishes a direct mapping between the parameters 
of the normal form and those of the original system.
\end{remark}

Since detailed analyses of the codimension-three BT bifurcation are 
available in~\cite{Dumortier1987,YuZhang2019,ZengYuHan2024}, 
we omit the derivation and instead summarize the results in the 
following theorem, accompanied by the bifurcation diagram 
in Figure~\ref{Fig1}, which closely resembles those 
in~\cite{Dumortier1987,YuZhang2019,ZengYuHan2024}.

\begin{theorem}\label{Thm3}
For system \eqref{Eqn7}, a codimension-three (degenerate cusp) BT 
bifurcation occurs at the equilibrium $(X_{\rm 2 bt_3},Y_{\rm 2 bt_3})
= \big(\frac{8}{3},\frac{1}{3}\big) $ when $(n,e,g)=\big(\frac{1}{3},
\frac{1}{4},\frac{4}{3}\big)$. From the PSNF \eqref{Eqn25}, 
the following six local bifurcations are obtained, together 
with their bifurcation surfaces/curves: 
\begin{enumerate}
\item[{$(1)$}] 
Saddle-node $({\rm SN})\!:$ \\[1.0ex]
${\rm SN} = \big\{(\beta_1,\beta_2,\beta_3) \mid \beta_1 =0 \big\}. $

\vspace{0.05in} 
\item[{$(2)$}] 
Hopf $({\rm H})\!:$  \\[1.0ex] 
${\rm H} = \Big\{(\beta_1,\beta_2,\beta_3) \mid 
\beta_2=\big(\beta_3+ \beta_1 \big) \sqrt{-\beta_1} \Big\}. $

\vspace{0.05in} 
\item[{$(3)$}] 
Homoclinic loop $({\rm HL})\!:$ \\[1.0ex] 
${\rm HL} = \Big\{(\beta_1,\beta_2,\beta_3) \mid 
\beta_2=\frac{5}{7} \big(\beta_3+\frac{179}{11} \beta_1 \big) 
\sqrt{-\beta_1} \Big\}. $

\vspace{0.05in} 
\item[{$(4)$}] 
Generalized Hopf $({\rm GH})\!:$  \\[1.0ex] 
${\rm GH} = \Big\{(\beta_1,\beta_2,\beta_3) \mid  
\beta_2 = 4 \beta_1 \sqrt{-\beta_1}, \ \beta_3= 3 \beta_1  \Big\}. $

\vspace{0.05in} 
\item[{$(5)$}] 
Degenerate homoclinic $({\rm DHL})\!:$ \\[1.0ex]
${\rm DHL} = \Big\{(\beta_1,\beta_2,\beta_3) \mid  
\beta_2 = \frac{4}{11} \beta_1 \sqrt{-\beta_1}, \ 
\beta_3 = - \frac{15}{11} \beta_1 \Big\}. $

\vspace{0.05in} 
\item[{$(6)$}] 
Double limit cycle $({\rm DLC})\!:$  

Occurs on a critical surface tangent to the Hopf surface {\rm H}
along the {\rm GH} curve, and tangent to the {\rm HL} surface along 
the {\rm DHL} curve.   
\end{enumerate} 
\end{theorem}

\subsection{Hopf bifurcation}

In this subsection, we analyze the Hopf bifurcation and determine 
the maximal number of small-amplitude limit cycles that can bifurcate from a 
Hopf critical point. Our analysis uses the method of normal 
forms together with the associated \textsc{Maple} programs.  
We obtain the following result.

\begin{theorem}\label{Thm4}
For system \eqref{Eqn7}, the codimension of Hopf bifurcation is two. 
That is, at most two small-amplitude limit cycles can bifurcate from 
a Hopf critical point: an outer stable cycle and an inner unstable one, 
both enclosing the stable equilibrium ${\rm E_2}$. 
\end{theorem} 

We present two proofs of the theorem: a purely algebraic approach, 
which yields additional insight, and a combined algebraic–graphical 
approach, which is more straightforward.

\begin{proof} \ {\bf Method 1.} \ 
We first solve $F_1(X_2)=0$ for $X_2$, obtaining two solutions 
\begin{equation}\label{Eqn26}
X_{2\pm} = \dfrac{1}{2e} \big[ 1 + e g \pm \sqrt{\Delta_1} \big], 
\end{equation} 
which requires $\Delta_1=(1-eg)^2-4egn>0$ (with $\Delta_1=0$ giving 
the condition for a BT bifurcation).
This yields two equilibrium solutions:
$$
{\rm E_{2 \pm}} = \big(X_{2 \pm}, 1 - e X_{2 \pm} \big).
$$
Moreover, the condition $1 - e X_{2 \pm}>0$ implies 
$X_{2\pm}<\frac{1}{e}$, which in term requires $eg<1$. 
Further, a direct computation of the Jacobian matrix evaluated at 
${\rm E_{2 \pm}}$ shows that 
$$
\begin{array}{rl}  
\det (J({\rm E_{2+}})) = \!\!\!\! & 
-\,\dfrac{g \Delta_1 \big[ 1+e g+ \sqrt{\Delta_1} \,\big]  
\big[ 1-e g+2 n+ \sqrt{\Delta_1} \,\big]}
{2(1+n)\big[ 1-e g+ \sqrt{\Delta_1}\, \big]} <0, \\[2.5ex]
\det (J({\rm E_{2-}})) = \!\!\!\! & 
\dfrac{2 g (1+n) \Delta_1 \big[ 1-e g+ \sqrt{\Delta_1} \,\big]}
{\big[ 1+e g+ \sqrt{\Delta_1}\, \big]
\big[1-eg + 2n +  \sqrt{\Delta_1}\, \big]} >0, 
\end{array} 
$$ 
implying that ${\rm E_{2+}}$ is a saddle, while ${\rm E_{2-}}$ is 
either a stable focus when ${\rm tr}(J({\rm E_{2-}}))<0$ 
or an unstable focus when ${\rm tr}(J({\rm E_{2-}}))>0$. 
Consequently, a Hopf bifurcation can occur from the 
equilibrium ${\rm E_{2-}}$ only if ${\rm tr}(J({\rm E_{2-}}))=0$. 
 
Next, to determine the condition under which the system undergoes a 
Hopf bifurcation from the equilibrium ${\rm E_{2-}}$, 
we present a parametric analysis. 
As discussed in Section 2, solving the polynomial equation $F_1(X_2)=0$ 
for $n$ gives
$$ 
n = \dfrac{(1-e X_2)(X_2-g)}{g}, \quad 
\textrm{with} \ \ g<X_2<\dfrac{1}{e}. 
$$  
The trace of the Jacobian matrix evaluated at ${\rm E_2}$ is 
then obtained as 
$$ 
{\rm tr}(J({\rm E_2})) = -\,\dfrac{1}{X_2} 
\big(e X_2^2 - g \, X_2 + g^2 \big).  
$$ 
Solving ${\rm tr}(J({\rm E_2})) =0$ yields two Hopf critical points:
$$ 
X_{\rm 2 H_{\pm}} = \dfrac{g}{2e} \big( 1 \pm \sqrt{1-4 e}\, \big), 
\quad e<\dfrac{1}{4}, 
$$
where the restriction $e<\frac{1}{4}$ follows the requirement $\Delta_1>0$.  
Verifying the condition $g<X_2 <\frac{1}{e}$ yields 
\begin{equation}\label{Eqn27}
0<e<\dfrac{1}{4}, \quad g < X_{\rm 2 H_\pm} < \dfrac{1}{e}, \quad 
\left\{
\begin{array}{lll}
g<\dfrac{1}{2e} \big( 1 - \sqrt{1-4 e} \big), 
& \textrm{for} \ \ X_{\rm 2 H_+}, \\[2.0ex] 
g<\dfrac{1}{2e} \big( 1 + \sqrt{1-4 e} \big), 
& \textrm{for} \ \ X_{\rm 2 H_-}. 
\end{array} 
\right. 
\end{equation} 
We then define the two Hopf critical points as 
\begin{equation}\label{Eqn28}
{\rm C_{H_\pm}} = \left(X_{\rm 2 H_\pm},\, Y_{\rm 2 H_\pm} \right)
= \left(\dfrac{1}{2e} \big( 1 \pm \sqrt{1-4 e}\, \big),\,
1-\dfrac{g}{2} \big( 1 \pm \sqrt{1-4 e}\, \big) \right).  
\end{equation}
Under the above conditions, the determinants of the Jacobian 
at these two Hopf critical points are 
$$ 
\det({\rm C_{H_\pm}}) = g \big[1-g (1-e \pm \sqrt{1-4 e} \,) \big],  
$$ 
and one can verify from \eqref{Eqn27} that $\det({\rm C_{H_\pm}})>0$. 

Since the procedure for deriving the normal form (or focus values) of 
the Hopf bifurcation is identical for both Hopf critical points, 
we restrict our analysis to ${\rm C_{H_+}}$. To this end, 
We multiply the right-hand side of system \eqref{Eqn7} 
by $n+Y$, which corresponds to the time rescaling $\tau = (n+ Y) \tau_1$. 
Applying the affine transformation,
$$ 
\left(\begin{array}{c} X \\ Y \end{array} \right)
= \left(\begin{array}{c} X_{\rm 2 H_+} \\ Y_{\rm 2 H_+} \end{array} \right)
+ \left[
\begin{array}{cc}
1 & 0 \\[1.0ex] 
e & \dfrac{\sqrt{2}\, e^2\, \omega_c}
{g \big(1+ \sqrt{1-4 e} \big) \big[g \big(1 -2 e + \sqrt{1-4 e}\, \big) 
-1- \sqrt{1-4 e} \,\big]}
\end{array} 
\right]
\left(\begin{array}{c} x_1 \\ x_2 \end{array} \right), 
$$
where $\omega_c = \sqrt{\det({\rm C_{H_+}})}$, into \eqref{Eqn7} yields 
\begin{equation}\label{Eqn29} 
\begin{array}{rl}
\dfrac{d x_1}{d \tau_1} = \!\!\!\! & x_2 
+ \dfrac{(1-e-\sqrt{1-4 e}\,)
         [1+2 e-\sqrt{1-4 e}-2 g e (2+e)]}{g (e+2) M_1}\, x_1 x_2 
+ \dfrac{\omega_c M_6}{2 g M_2}\, x_2^2 
\\[2.5ex] 
&-\, \dfrac{e M_3}{g M_1}\, x_1^2 x_2 
+\dfrac{\omega_c M_7}{2 g^2 M_2}\, x_1 x_2^2, 
\\[2.5ex] 
\dfrac{d x_2}{d \tau_1} = \!\!\!\! & -\, x_1  
+ \dfrac{e\, [4 \!-\! e \!-\! (4 \!+\! e) \sqrt{1 \!-\! 4 e}\,]\, 
[3 \!-\! 2 e \!-\! 3 \sqrt{1 \!-\! 4 e} \!-\! g e (6+e)]}
{(e+6) M_1 M_4}\, x_1^2 + \dfrac{\omega_c M_9}{2 g M_5}\, x_1 x_2 
\\[2.5ex] 
& -\, \dfrac{2 e^2}{M_1}\, x_2^2 
- \dfrac{e^2\, [2 \!-\! 5 e \!-\! (2\!-\!e) \sqrt{1 \!-\!4 e}]} 
{M_1 M_4}\, x_1^3 
+ \dfrac{e (e \-\! 2) \omega_c M_8}{2 g M_5}\, x_1^2 x_2 
- \dfrac{(1 \!-\! e) M_3}{g M_1}\, x_1 x_2^2, 
\end{array} 
\end{equation} 
where 
$$ 
\begin{array}{ll}
M_1 = \!\!\!\! & 1-\sqrt{1 - 4 e} -2 g e, \hspace*{0.70in}  
M_2 = (1-g)^2-g^2 \big[1-2 e-(1-e g)^2 \big], \\[1.0ex] 
M_3 = \!\!\!\! & 1-3 e-(1-e) \sqrt{1 - 4 e}, \qquad 
M_4 = 1-e- \sqrt{1 - 4 e} -g e (2+e), \\[1.0ex] 
M_5 = \!\!\!\! & e^3 (e+2) g^6-6 e^2 g^5+2 e (e^2+e+3) g^4+2 (e^2-4 e-1) g^3
      +(5+e^2) g^2-2 (2-e) g+1, \\[1.0ex] 
M_6 = \!\!\!\! & e \big[5 (1-e) g^2-4 (2-e) g+3 \big]-(1-g)^2
      +\big[(1-g)^2-e ((3-e) g^2-4 g+1) \big] \sqrt{1 - 4 e}, \\[1.0ex]
M_7 = \!\!\!\! & e (6-9 e+2 e^2) g^2-10 e (1-e) g+2 e (2-e)-(1-g)^2 
\\[1.0ex] 
&+\,\big[1-(1-e) g \big] \big[1-2 e-g (1-3 e) \big]  \sqrt{1 - 4 e}, \\[1.0ex] 
M_8 = \!\!\!\! & \big[g^3 (4 e^2-7 e+2)-g^2 (2 e^2-12 e+5)+(4-5 e) g-1 \big] 
 \sqrt{1 - 4 e} \\[1.0ex]
& +\, g^3 (2 e^3-14 e^2+11 e-2)+g^2 (16 e^2-22 e+5)
      -g (2 e^2-13 e+4)-2 e+1, \\[1.0ex] 
M_9 = \!\!\!\! & \big[ e g^4 (e^3-4 e^2+10 e-4)+g^3 (4 e^3-e^2-6 e+4)
       -g^2 (e^3+9 e^2-21 e+10) 
\\[1.0ex] 
& +\, 4 g (2 \!-\! 3 e) \!+\! e \!-\! 2 \big] \sqrt{1 \!-\! 4 e}  
\!-\! e g^4 (5 e^3 \!-\! 16 e^2 \!+\! 18 e \!-\! 4)
\!+\! g^3 (4 e^4 \!-\! 2 e^3 \!-\! 3 e^2 \!+\! 14 e \!-\! 4)
\\[1.0ex] 
&+\, g^2 (3 e^3+31 e^2-41 e+10)-2 g (7 e^2-14 e+4)-5 e+2.
\end{array} 
$$
Now, applying the Maple program from \cite{Yu1998} for computing 
normal forms of Hopf and generalized Hopf bifurcations, 
we obtain the focus values:
$$ 
\begin{array}{rl}
v_1 = \!\!\!\!& -\, \tfrac{e (e+2)}
{8 [1-8 e-2 e^2-3  \sqrt{1 - 4 e}\, ]\, [1-e+  \sqrt{1 - 4 e}-g e (e+2)]}
\ v_{1a}, 
\\[1.0ex]
v_2 = \!\!\!\! & 
-\,\tfrac{e}
{576 g^2 [1+\sqrt{1 - 4 e}-2 ge ] [1-e+ \sqrt{1 - 4 e}-g e(e+2)]^3
(81 e^7+1800 e^6+11428 e^5+23370 e^4+3657 e^3-4768 e^2-369 e+234)}
\\[1.5ex]
& \times \big[104-580 e+94 e^2+2551 e^3-617 e^4-676 e^5-93 e^6
\\[0.3ex]
&\quad +\,(104-372 e-442 e^2+1303 e^3+459 e^4-8 e^5-9 e^6) 
\sqrt{1-4 e} \, \big]\, v_{2a},
\end{array} 
$$ 
where 
$$
\begin{array}{rl}
v_{1a} = \!\!\!\!& 2 g e (e^3+6 e^2+3 e-1)
+ \big[ 2 e^3-6 e^2-3 e+1+(1+e) (1-4 e)^{3/2}\big],
\\[1.0ex] 
v_{2a} = \!\!\!\!& 
\big[g^3 e^2 (1017 e^6+14094 e^5+46033 e^4+17379 e^3-11998 e^2-1602 e+702)
\\[0.5ex]
& +\,g^2 e (915 e^6+8449 e^5-13250 e^4-29988 e^3+9796 e^2+2349 e-702)
\\[0.5ex]
& -\,g (141 e^6+3892 e^5+16554 e^4-15611 e^3+1501 e^2+1251 e-234)
\\[0.5ex]
& -\,39 e^5-455 e^4+999 e^3-184 e^2-81 e+18 \big]  \sqrt{1 - 4 e}
\\[0.5ex]
&-\,2 g^4 e^3(81 e^7+1800 e^6+11428 e^5+23370 e^4+3657 e^3-4768 e^2-369 e+234)
\\[0.5ex]
&-\,g^3 e^2(108e^7+3543e^6+5830e^5-36567e^4-24783 e^3+12056 e^2+1836 e-702) 
\\[0.5ex]
&+\,g^2e(288 e^7+6517 e^6+58993 e^5+36804 e^4-50110 e^3 +5384 e^2+3753 e-702)
\\[0.5ex]
& +\,g (252 e^7+3647 e^6+6146 e^5-42326 e^4+20621 e^3+533 e^2-1719 e+234)
\\[0.5ex]
&+\,18 e^6+67 e^5-1965 e^4+1431 e^3-58 e^2-117 e+18.
\end{array} 
$$ 
Under the conditions $e<\frac{1}{4}$ and $g<\frac{1}{2e}
(1-\sqrt{1 - 4 e})$, it follows that all the factors in 
$v_2$ are positive, and thus $v_2$ and $v_{2a}$ have 
opposite signs. 
To obtain multiple limit cycle bifurcations, we impose $v_{1a}=0$, which gives
\begin{equation}\label{Eqn30} 
g_c = -\,\dfrac{2 e^3-6 e^2-3 e+1+(1+e) (1-4 e)^{\frac{3}{2}}}
{2 e (e^3+6 e^2+3 e-1)}, \quad e^3+6 e^2+3 e-1 \ne 0.  
\end{equation}
Note that $ e^3+6 e^2+3 e-1 \ne 0$ is equivalent to 
$1 - 8e - 2e^2 - 3 \sqrt{1 - 4 e} \ne 0$, which ensures a nonzero 
denominator in $v_1$. The positivity of $g_c>0$ requires 
$$ 
(e^3+6 e^2+3 e-1) \, \big[ 2 e^3-6 e^2-3 e+1+(1+e) (1-4 e)^{3/2} \big] <0. 
$$  
A detailed examination shows that for $e<\frac{1}{4}$, $g_c>0$ holds when
$$ 
e \in (0,\,0.22668159 \cdots) \bigcup (0.23606797 \cdots,\,0.25). 
$$
Moreover, verifying the condition $g_c < \frac{1}{2e}(1- \sqrt{1 - 4 e}\,)$ 
yields
$$ 
e \in (0.23606797 \cdots,\,0.25) \quad \textrm{under which} \ \ 
e^3+6 e^2+3 e-1 > 0.  
$$  

Substituting $g=g_c$ into $v_{2a}$ gives 
$$ 
v_{2a} \! \mid_{v_1=0} = -\,\dfrac{6 e^3 (1-4 e)}
{e^3+6 e^2+3 e-1)^4}\, \big(v_{2a}^{(1)} \sqrt{1 - 4 e} + v_{2a}^{(2)} \big), 
$$
where
$$
\begin{array}{rl}
v_{2a}^{(1)} = \!\!\!\! & 9 e^{12}+221 e^{11}-711 e^{10}-32860 e^9-78923 e^8
-53364 e^7 +12185 e^6 
\\[0.5ex]
& +\, 23854 e^5-384 e^4-3643 e^3+335 e^2+135 e-18, \\[1.0ex]
v_{2a}^{(2)} = \!\!\!\! & 87 e^{12}+3267 e^{11}+28731 e^{10}+4512 e^9
-102873 e^8 -101184 e^7
\\[0.5ex]
&+\, 9723 e^6+28032 e^5+462 e^4-3099 e^3-3 e^2+171 e-18.
\end{array}
$$ 
It can be verified that $v_{2a}^{(2)} <0 $ for 
$ e \in (0.23606797 \cdots,\,0.25)$, with 
$$
v_{2a}^{(1)} \left\{ \begin{array}{lll} 
\le 0, & \textrm{for} &  0.23608731 \cdots \le e<0.25, \\[0.5ex]
  > 0, & \textrm{for} &  0.23606797 \cdots <e < 0.23608731 \cdots. 
\end{array} 
\right. 
$$ 
Consequently, $v_{2a}>0$ for $ 0.23608731 \cdots \le e<0.25$. 
For $0.23606797 \cdots <e < 0.23608731 \cdots$, 
the inequality $ v_{2a}^{(1)} \sqrt{1 - 4 e} + v_{2a}^{(2)} < 0$ 
is equivalent to 
$ v_{2a1}^2 (1 - 4 e) - v_{2a2}^2 < 0$, which reduces to 
$$
\begin{array}{rl} 
& \ \  -\, 4 e^2 (e^3+6 e^2+3 e-1)^4 (e^2+22 e-4) (e^2+4 e-1) \\[0.5ex]
& \times \, 
(81 e^7+1800 e^6+11428 e^5+23370 e^4+3657 e^3 -4768 e^2-369 e+234) <0, 
\end{array}
$$  
and this condition holds for $0.23606797 \cdots <e < 0.23608731 \cdots$. 

Therefore, $v_{2a}>0$ for $e \in (0.23606797 \cdots,\,0.25)$, when 
$v_1=0$, which implies $v_2<0$ in this range. Thus, 
the maximal number of small-amplitude limit cycles 
bifurcating from the Hopf critical point ${\rm C_{H_+}}$ 
is two: an inner unstable cycle and an outer stable cycle, 
both enclosing the stable equilibrium ${\rm E_2}$.  

Repeating the above procedure for the other Hopf critical point 
${\rm C_{H_-}}$ yields the exact conditions under which 
two small-amplitude limit cycles exist. Hence, for the 
same parameter values, two limit cycles may bifurcate 
from either ${\rm C_{H_+}}$ or ${\rm C_{H_-}}$. 
It should be emphasized, however, that these two Hopf bifurcations occur 
at different points on the equilibrium ${\rm E_{2-}}$, corresponding 
to different values of the parameter $n$. Nevertheless, the bifurcating 
limit cycles share the same stability properties.

\vspace*{0.10in} 
\noindent 
{\bf Method 2.}
Now, we turn to the algebraic-graphical approach. 
Using the parameter $e$, we solve for the trace of the Jacobian matrix 
evaluated at ${\rm E_2}$, which yields 
$$ 
e_{\rm H} = \dfrac{g (X_2-g)}{X_2^2}, 
$$
where the subscript `H` denotes the Hopf critical point. 
Combining $g<X_2<\frac{1}{e} $ with $e=e_{\rm H}$ gives 
$$ 
X_2 < \dfrac{1}{e} = \dfrac{X_2^2}{g (X_2-g)} \ \ 
\Longrightarrow \ \ (g-1) X_2 < g^2,  
$$ 
which leads to the following existence conditions for the Hopf bifurcation:
\begin{itemize}
\item 
$g \le 1, \ \ g<X_2$; 
\item
$g>1, \ \ g<X_2<\dfrac{g^2}{g-1}$.   
\end{itemize} 

At the Hopf critical point, the trace of the Jacobian is zero, and 
the determinant gives a purely imaginary pair of eigenvalues 
$\lambda_{1,2} = \pm i \omega$, where
$$ 
\omega^2 = \det(J({\rm E_2})) 
= \dfrac{g}{X_2} \left[ (1-2 g) X_2^2+3 g^2 X_2 -g^3 \right]. 
$$ 
Requiring $\omega>0$ yields 
$$ 
F_4(X_2) \stackrel{\triangle}{=} (1-2 g) X_2^2+3 g^2 X_2 -g^3 > 0.  
$$ 
The function $F_4$ satisfies 
$$ 
F_4(g) = g^2>0, \quad \dfrac{d F_4}{d X_2} = 2 (1-2 g) X_2+3 g^2. 
$$ 
Let $X_{2\pm}$ be the roots of $F_4(X_2)=0$:
$$ 
X_{2\pm} = \dfrac{g}{2(2 g-1)} \left[ 3 g \pm \sqrt{g(g+4)} \right], 
\quad \textrm{for} \ \ g \ne \dfrac{1}{2}.  
$$  
From these, it follows that $F_4>0$ for 
\begin{itemize}
\item 
$X_2 > g\,$ if $\,0<g \le \frac{1}{2}$; or 
\item 
$g < X_2 < X_{2+}\,$ if $\, g > \frac{1}{2}$.
\end{itemize}  
To compute the focus values for the Hopf bifurcation,
we similarly multiply the right-hand side of system \eqref{Eqn7} by $n+Y$,
which is equivalent to applying the time rescaling $\tau = (n + Y) \tau_1$. 
Applying the affine transformation 
$$ 
\left(\begin{array}{c} X \\ Y \end{array} \right)
= \left(\begin{array}{c} X_2 \\ 1 - e X_2 \end{array} \right)
+ \left[
\begin{array}{cc}
1 & 0 \\[1.0ex] 
\dfrac{g(g-X_2)}{X_2^2} & \dfrac{g \omega}{X_2 [(1-g)X_2 + g^2]} 
\end{array} 
\right]
\left(\begin{array}{c} x_1 \\ x_2 \end{array} \right)
$$ 
into \eqref{Eqn7} yields 
\begin{equation}\label{Eqn31} 
\begin{array}{rl} 
\dfrac{d x_1}{ d \tau_1} = \!\!\!\! & 
x_2 + \dfrac{(g-1) X_2^2-g^3}{X_2^2 [(g-1)X_2-g^2]} \, x_1 x_2 
- \dfrac{g^2 \omega}{X_2 [(g-1) X_2 -g^2]^2}\, x_2^2 
\\[2.5ex]
& + \dfrac{g^2 (X_2-g)}{X_2^3 [(g-1) X_2-g^2]}\, x_1^2 x_2 
- \dfrac{g^2 \omega}{X_2^2 [(g-1) X_2 -g^2]^2} \, x_1 x_2^2, \\[3.0ex] 
\dfrac{d x_2}{ d \tau_1} = \!\!\!\! & 
- x_1 + \dfrac{g (X_2-g) [(3 g-2) X_2^2-4 g^2 X_2 +g^3]}
{X_2 [(g-1) X_2-g^2] [(1-2 g) X_2^2+3 g^2 X_2-g^3] } \, x_1^2 
\\[2.5ex] 
& - \dfrac{(3 g-1) X_2^3-g (5 g+1) X_2^2+3 g^3 X_2-g^4}{X_2^3 \omega}\, x_1 x_2 
+ \dfrac{g (X_2-g)^2}{X_2^2 [(g-1) X_2-g^2]}\, x_2^2 
\\[2.5ex]
& + \dfrac{g^2 (X_2-g)^2}{X_2 [(g-1) X_2-g^2] 
[(1-2 g) X_2^2+3 g^2 X_2 -g^3]}\, x_1^3 
\\[2.5ex]
& - \dfrac{g(X_2-g) (2 X_2^2-g X_2 +g^2)}{X_2^4 \omega}\, x_1^2 x_2 
+ \dfrac{g (X_2^2-g X_2 +g^2)}{X_2^3 [(g-1) X_2-g^2]}\, x_1 x_2^2. 
\end{array} 
\end{equation}  

Using the Maple program from \cite{Yu1998} for 
computing normal forms of Hopf and generalized Hopf bifurcations, we 
obtain the focus values:
$$ 
\begin{array}{ll} 
v_1 = \dfrac{g (X_2-g)\, v_{1a}}{8 X_2^4 [(1-2g) X_2^2+3 g^2 X_2-g^3]}, 
\\[2.0ex] 
v_2 = \dfrac{g (X_2-g)\, v_{2a}}
      {288 X_2^7 [g^2-(g-1)X_2] [(1-2g)X_2^2+3g^2X_2-g^3]^3}, 
\end{array} 
$$ 
where 
$$
\begin{array}{rl}
v_{1a} = \!\!\!\! & 
g (X_2-g) (X_2^3-3 g^2 X_2+g^3)-X_2^2 (X_2^2-g X_2-g^2),
\\[1.0ex] 
v_{2a} = \!\!\!\! & (g-1) (18 g^3-27 g^2+16 g-1) X_2^{10}
-g (g-1) (63 g^3-95 g^2+40 g-1) X_2^9
\\[0.5ex]
&-\,g^2 (148 g^4-105 g^3-14 g^2+6 g+1) X_2^8
+g^4 (1179 g^3-1279 g^2+399 g-45) X_2^7
\\[0.5ex]
&-\,2 g^5 (1329 g^3-1178 g^2+246 g-7) X_2^6
+g^7 (2982 g^2-1902 g+187) X_2^5
\\[0.5ex]
&-\,2 g^8 (855 g^2-344 g+8) X_2^4
+3 g^{10} (116 g-25) X_2^3
+g^{11} (109 g-6) X_2^2
\\[0.5ex]
&-\,66 g^{13} X_2 +9 g^{14} 
\end{array}
$$ 
Note that $v_{1a}$ and $v_1$, as well as $v_{2a}$ and $v_2$, share the 
same sign. Since $v_{1a}$ and $v_{2a}$ depend on two free parameters: 
$g$ and $X_2$, the system can, in principle, exhibit up to three 
small-amplitude limit cycles bifurcating from the Hopf critical point. 
To investigate this possibility, we eliminate
$X_2$ from the equations $v_{1a}=0$ and $v_{2a}=0$, yielding an 
expression for $X_2$ and the corresponding resultant $R(g)$: 
$$
\begin{array}{rl} 
X_2 = \!\!\!\! & 
\dfrac{g^2 (57 g^6-301 g^5+165 g^4+504 g^3-264 g^2+100 g+3)}
{57 g^7-358 g^6+466 g^5+282 g^4-581 g^3+345 g^2-62 g+11}, \\[2.0ex] 
R(g) =  \!\!\!\! & g (3 g-4) (19 g^2-94 g+11). 
\end{array} 
$$ 
$R(g)$ has three positive real roots:
$$ 
g_1 = \dfrac{4}{3}, \quad g_2 = \dfrac{47-20 \sqrt{5}}{19}, \quad 
g_3 = \dfrac{47+20 \sqrt{5}}{19}. 
$$ 
Substituting into the expressions for $X_2$ and $\omega^2$ shows that 
$$ 
\begin{array}{lllll}
g_1 & \Longrightarrow & X_2 = \dfrac{80}{33}, &\Longrightarrow & 
v_1 \, v_2 \ne 0, 
\\[2.0ex] 
g_2 & \Longrightarrow & X_2 = \dfrac{241-107 \sqrt{5}}{76}, & 
\Longrightarrow & v_1=v_2=0, \ \ \textrm{but} \ \ \omega^2 
= -\,\dfrac{4710-1733 \sqrt{5}}{361}<0, 
\\[2.0ex] 
g_3 & \Longrightarrow & X_2 = \dfrac{241+107 \sqrt{5}}{76}, & 
\Longrightarrow & v_1=v_2=0, \ \ \textrm{but} \ \ \omega^2
= -\,\dfrac{4710+1733 \sqrt{5}}{361}<0. 
\end{array} 
$$ 
Thus, no solution satisfies $v_1=v_2=0$ and 
$\omega^2>0$, and three small-amplitude limit cycles are impossible.  
Therefore, the Hopf bifurcation has codimension at most two, 
meaning the maximal number of bifurcating limit cycles is two.

\begin{figure}[!h] 
\vspace*{0.10in}
\begin{center}
\hspace*{-0.20in}
\begin{overpic}[width=0.5\textwidth,height=0.33\textheight]{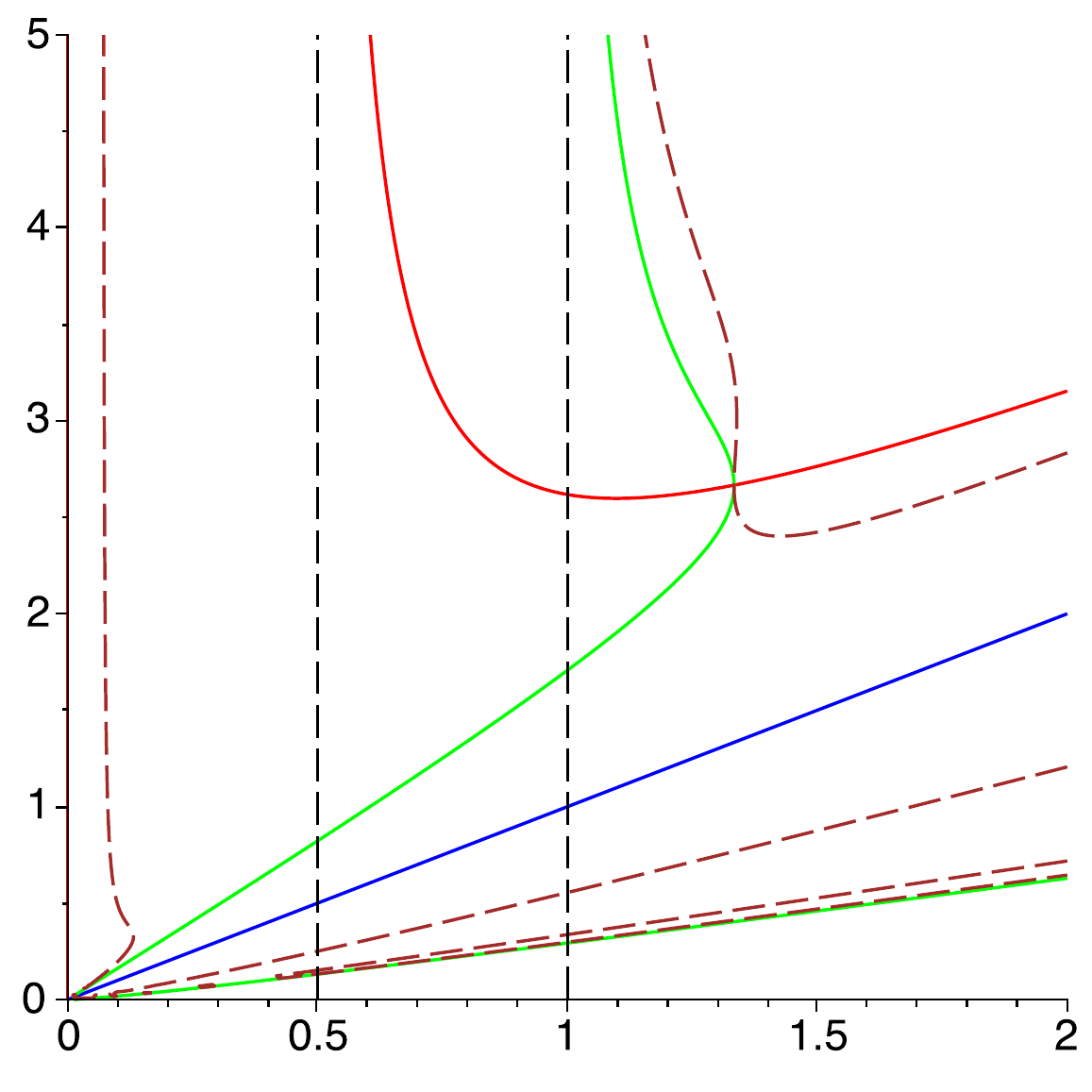}
\put(50.5,-1){$g$}
\put(-8,50){$X_2$}
\put(40,61){\footnotesize{\color{red}$X_{2+}$}}
\put(68,27){\footnotesize{\color{blue}$X_2 \!=\! g$}} 
\put(61,40){\color{green}$v_{1a} \!=\! 0$}
\put(52.7,53){\color{black}$v_{2a} \hspace*{-0.04in} =\! 0$}
\put(50.8,34.8){$\bullet$} 
\end{overpic}

\vspace{0.00in}
\caption{Curve of $v_{1a} = 0$ in the $g$–$X_2$ plane, illustrating
the continuum of parameter pairs $(g, X_2)$ that yield two
small-amplitude limit cycles. Any point on the green segment 
between the blue line and the red curve corresponds to 
values of $g$ and $X_2$ for which this phenomenon occurs. 
A specal case, $(g,X_2)=(1,1+\frac{1}{2} \sqrt{2})$, is marked 
by the black circle.}
\label{Fig2}
\end{center} 
\vspace*{-0.00in} 
\end{figure}
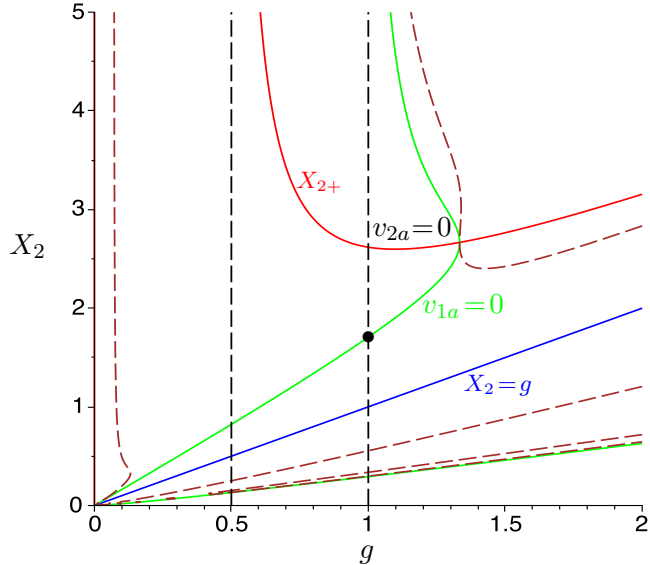

Since Theorem~\ref{Thm3} established that a codimension-three 
BT bifurcation can generate two small-amplitude limit cycles, it 
follows that the codimension of the Hopf bifurcation is two.  
Moreover, the equation $v_{1a} = 0$ involves two free parameters, 
$g$ and $X_2$, thereby allowing infinitely many parameter choices 
that produce two small-amplitude limit cycles. This is illustrated in 
Figure~\ref{Fig2}, where the region in the $g$-$X_2$ plane 
is bounded by the blue line and the red curve. The green
curve represents the condition $v_{1a}=0$, while the 
brown curve $v_{2a}=0$ does not intersect it, 
implying that $v_{2a} \ne 0$ whenever $v_{1a}=0$. 
A direct computation shows that $v_{2a}<0$ (and hence $v_2 < 0$)
whenever $v_{1a}=0$. 
Consequently, the outer limit cycle is stable while the 
inner one is unstable, and both enclose the stable equilibrium ${\rm E_2}$. 

This completes the proof of Theorem~\ref{Thm4}.
\end{proof}

To conclude this section, we present a concrete example. 
Consider $g=1$ and $X_2=1+\frac{1}{2} \sqrt{2}$,  
as indicated by the  black circle in Figure~\ref{Fig2}. 
In this case, we have $v_{1a}=0$, i.e., $v_1=0$, which yields 
$$
v_2 = \dfrac{1}{6} \big(41 - 29 \sqrt{2})<0,
$$
and 
$$
e_{\rm H}= 3 \sqrt{2}-4, \quad n=\sqrt{2}-1, \quad 
\omega_c^2 = \sqrt{2}-1,
$$
all of which satisfy the required positivity conditions.  

Substituting $g=1$ and $e=e_{\rm H} = 3 \sqrt{2}-4<\frac{1}{4}$, we obtain 
$$
\begin{array}{ll} 
X_{\rm 2 H_-} = 1 + \dfrac{\sqrt{2}}{2} 
\quad \Longrightarrow \quad n=\sqrt{2}-1, \quad 
v_0 = -\,\dfrac{\sqrt{2}}{2}, 
\\[2.0ex] 
X_{\rm 2 H_+} = 1 + \sqrt{2} 
\quad \Longrightarrow \quad n=2-\sqrt{2}, \quad 
v_0 = -\,\dfrac{3-\sqrt{2}}{2}, 
\\[2.0ex]  
g_c = -\,\dfrac{2 e^3-6 e^2-3 e+1+(1+e) (1-4 e)^{\frac{3}{2}}}
{2 e (e^3+6 e^2+3 e-1)} 
= 1 <\dfrac{1}{2e}(1+\sqrt{1-4e}) = 1 + \sqrt{2}.
\end{array}
$$ 
This result is in complete agreement with the formulas derived in Method 1. 
Here, $v_0$ represents the transversal condition, which, together 
with the common sign of $v_2$, implies the limit cycles bifurcation from 
the two distinct points possess the same stability property. 

The one-parameter bifurcation diagram is shown in Figure~\ref{Fig3}, 
illustrating the bifurcations of system \eqref{Eqn7}, in particular the 
generalized Hopf bifurcation. This bifurcation leads to two limit cycles 
emerging from two different Hopf critical points, corresponding 
to $n=\sqrt{2}-1$ and $n=2-\sqrt{2}$, respectively. The solution 
$X_{\rm 2 H_-}$ coincides with that obtained in Method 2.

\begin{figure}[!h] 
\vspace*{0.20in}
\begin{center} 
\hspace*{-1.00in}
\begin{overpic}[width=0.38\textwidth,height=0.28\textheight]{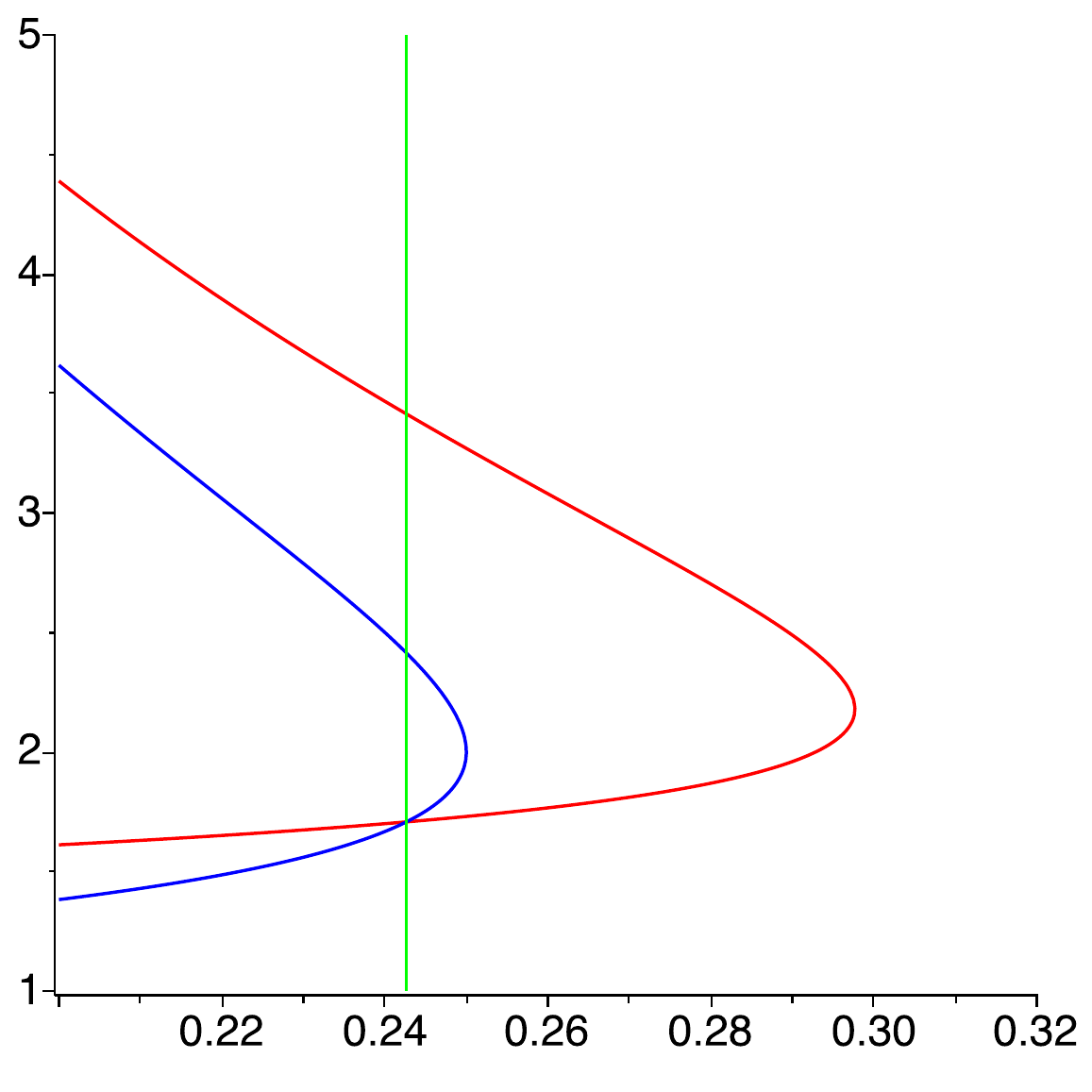}
\put(-8,55){{$X$}} 
\put(48,-3){{$e$}} 
\put(34.1,23.0){$\bullet$}
\put(73.3,33.4){$\bullet$}
\put(40.0,61){\small {\color{red}${\rm X_{2+}}$}} 
\put(47.0,20){\small {\color{red}${\rm X_{2-}}$}}  
\put(39.0,76){{\color{green}${\rm v_{1a}\!=\!0}$}} 
\put(30.0,47){\color{blue}{\small {H}}} 
\put(36, 19){\footnotesize{GH}} 
\put(79, 32.4){\footnotesize{SN}} 
\end{overpic}
\hspace*{0.50in}
\begin{overpic}[width=0.38\textwidth,height=0.28\textheight]{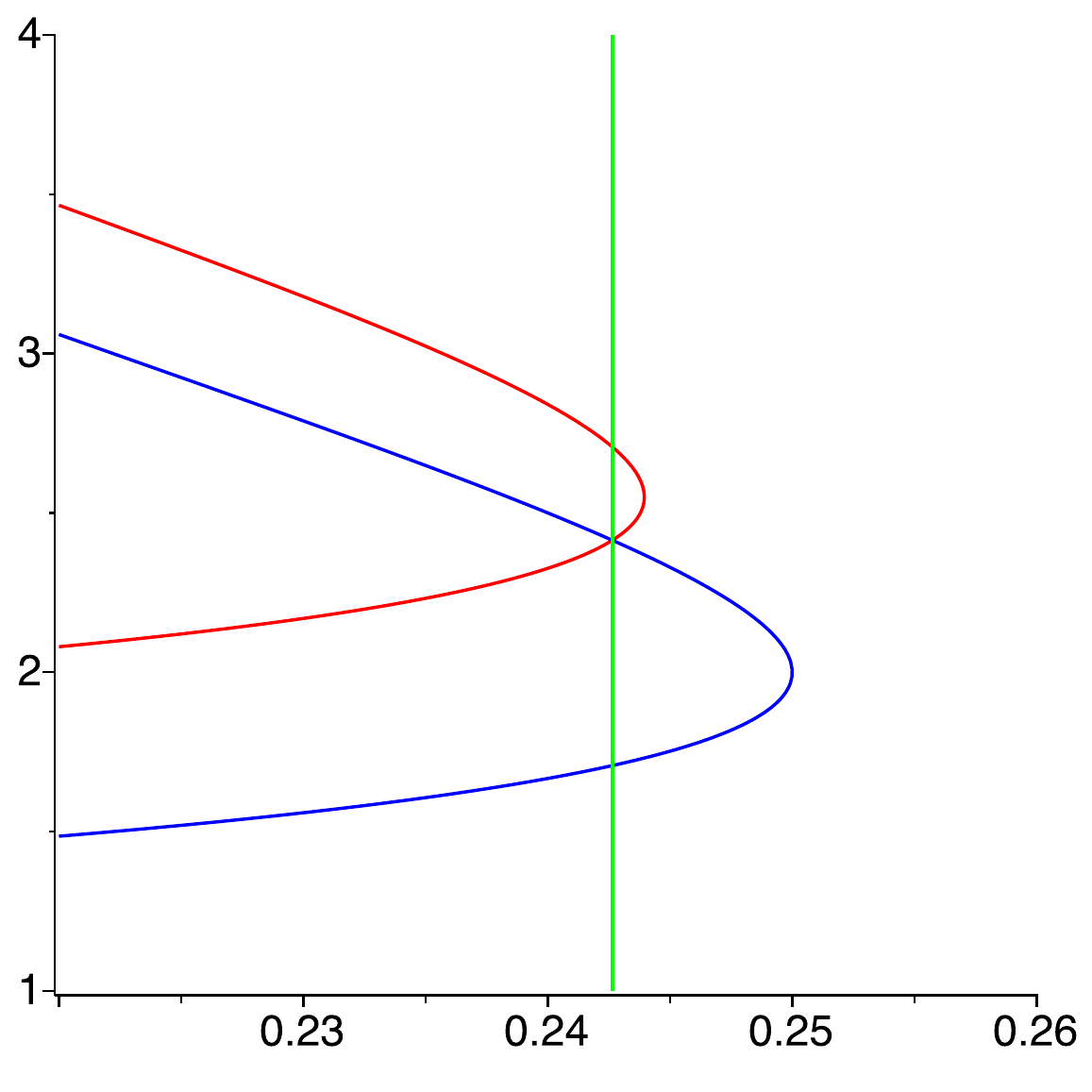}
\put(-8,55){{$X$}} 
\put(48,-3){{$e$}} 
\put(52.4,49.0){$\bullet$}
\put(55.2,53.0){$\bullet$}
\put(35.0,71){\small {\color{red}${\rm X_{2+}}$}} 
\put(35.0,39){\small {\color{red}${\rm X_{2-}}$}}  
\put(55.0,76){{\color{green}${\rm v_{1a}\!=\! 0}$}} 
\put(37.0,21){\color{blue}{\small {H}}} 
\put(50, 43){\footnotesize{GH}} 
\put(61, 52.5){\footnotesize{SN}} 
 
\end{overpic}\hspace*{-0.80in} 

\vspace*{0.10in} 
\hspace*{0.00in}(a)\hspace*{2.90in}(b) 

\caption{Bifurcation diagram of system~\eqref{Eqn7} projected 
onto the $e$–$X$ plane for $e=3\sqrt{2}-4$ and $g=1$. 
Here, SN, H, and GH denote the saddle-node, Hopf, and generalized Hopf 
bifurcations, respectively: (a) $n=\sqrt{2}-1$; (b) $n=2-\sqrt{2}$.}
\label{Fig3}
\end{center}
\vspace*{0.00in} 
\end{figure}
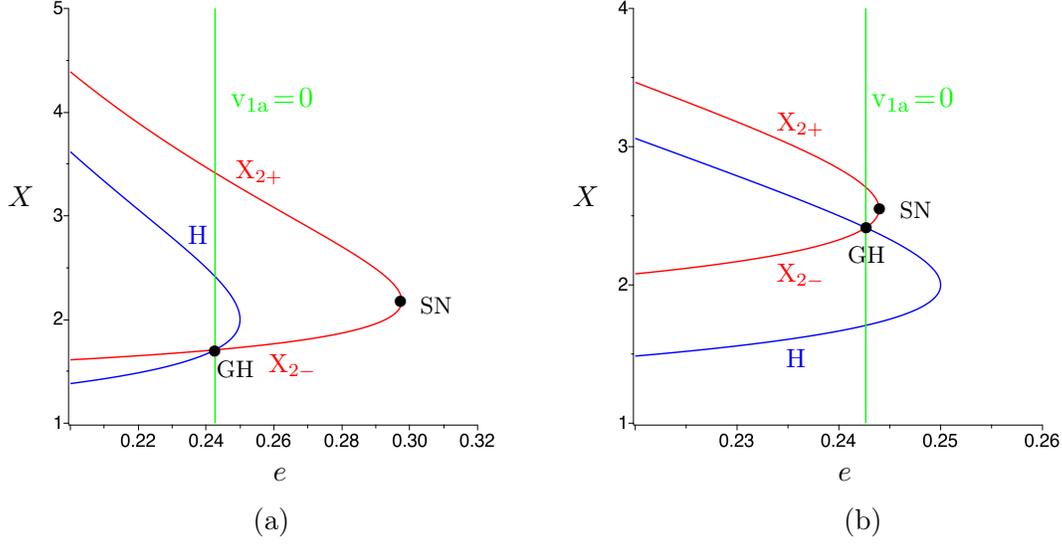

\begin{remark}\label{Rem2.7} 
Comparing the two methods, we observe that the purely algebraic approach 
(Method~1) provides more detailed information about the bifurcation of limit 
cycles, such as their bifurcation points and the directions in which they 
emerge. In contrast, the algebraic–graphical approach (Method~2) 
is less involved, offers intuitive insight directly from the graphs, 
and leads to simpler derivations.
\end{remark}

\subsection{Simulation}

In this section, we present simulations for system \eqref{Eqn7}, 
focusing on the Hopf and BT bifurcations. 
We first examine the Hopf bifurcation, followed by the BT bifurcation.  

\subsubsection{Simulatiion for two limit cycles bifurcating from 
a Hopf bifurcation} 

Suppose the normal form of a dynamical system near 
an equilibrium, truncated at order $2k+1$, is given by 
$$ 
\dfrac{d r}{d \tau} = r \big[v_0({\bf a}) \, \mu + v_1 ({\bf a})\, r^2 
+ \cdots + v_k({\bf a})\, r^{2k} \big], 
$$ 
where ${\bf a}=(a_1,a_2, \ldots,a_k)^T$ is a $k$-dimensional parameter vector
and $\mu$ is a perturbation parameter.  
The coefficient $v_k$ is called the $k$th-order focus value, 
while $v_0$ is the transversal condition, which can be determined via 
linear analysis. The higher-order coefficients $v_k$ ($k\ge 1$ 
can be computed using normal form methods or other 
approaches~\cite{HanYu2012}.  

Generalized Hopf bifurcation theory states that if, 
at a critical parameter point ${\bf a_c}$, the following 
conditions hold: 
$$ 
\begin{array}{ll} 
& v_0({\bf a_c})=v_1({\bf a_c}) = \cdots = v_{k-1}({\bf a_c})=0, \quad 
v_k({\bf a_c}) \ne 0, \\[1.5ex] 
\textrm{and} \quad & \det\left[
\dfrac{\partial(v_0,v_1, \cdots,v_{k-1})}
{\partial(a_1,a_2, \cdots,a_k)}
\right]_{\bf a=a_c} \ne 0, 
\end{array} 
$$
then $k$ small-amplitude limit cycles can bifurcate from the equilibrium 
near ${\bf a=a_c}$. 

For system \eqref{Eqn7}, an explicit example was given at the end of the 
proof of Theorem~\ref{Thm4}, with $g=1$ and $X_2=1+\frac{1}{2} \sqrt{2}$, 
yielding 
$$
v_0 = -\, \dfrac{X_21 [g^2+(1-g) X_2]}{2 g} 
= -\, \dfrac{125252+62501 \sqrt{2}}{251000}, \quad 
v_1=0, \quad v_2=\frac{1}{6}(41-29\sqrt{2})<0.
$$
Next, we perturb the parameters $g$ and $e$ as 
$$ 
g = 1 + \varepsilon_1, \quad e=e_{\rm H}+\varepsilon_2.  
$$ 
Since $v_0<0$, ${\rm E_2}$ is asymptotically stable whenever $\varepsilon_2>0$.
Taking $\varepsilon_1 = 0.004$ and $\varepsilon_2=-0.00000352$ gives
$$ 
v_1= 0.00018652 \cdots, \quad v_2 = -0.00201568 \cdots,  
$$
and that the normal up to $5$th order becomes 
$$ 
\dfrac{d r}{d \tau} \approx r \, \big(\!\! -\! 0.000003 + 0.00018652\, r^2 
-0.00201568 \, r^4 \big).  
$$ 
Solving $\frac{d r}{d \tau} = 0$ yields the approximations for 
the amplitudes of the two small-amplitude limit cycles:
$$ 
r_1 \approx 0.143969, \quad r_2 \approx 0.267967. 
$$

\begin{figure}[!h] 
\vspace*{0.50in}
\begin{center}
\hspace*{-3.33in}
\begin{overpic}[width=0.40\textwidth,height=0.22\textheight]{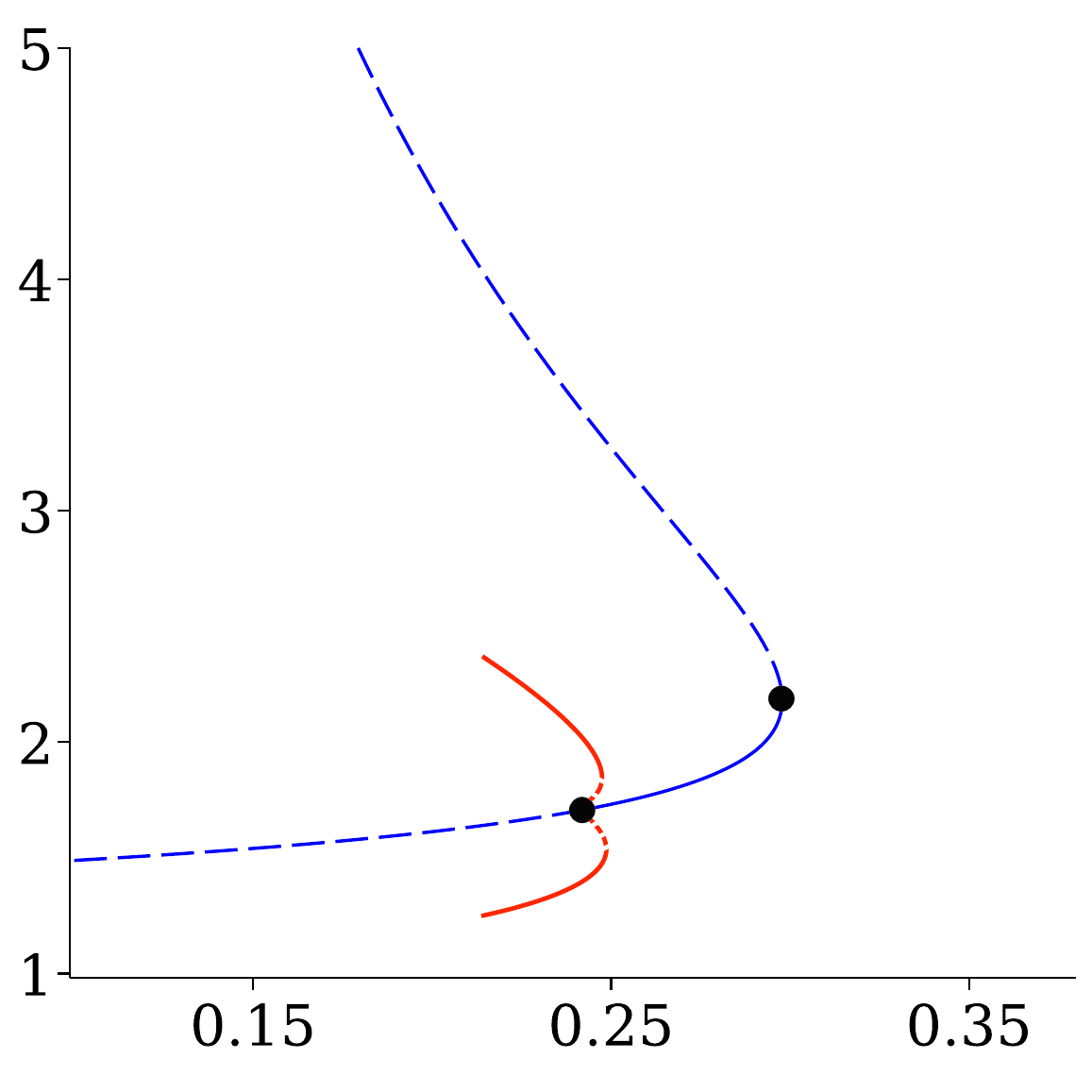}
\put(-8,48){{$X$}} 
\put(51,-3){\footnotesize{\large $e$}} 
\put(51,53){{$X_{2+}$}} 
\put(30,21){{$X_{2-}$}} 
\put(75,26){{SN}} 
\put(45,18){{H}} 
\put(57, 8){\line(0,1){15}} 
\put(54, 8){{\color{green}\line(0,1){20}}} 
\put(52.5,8){\line(0,1){23}} 
\end{overpic}

\vspace*{-1.912in} 
\hspace*{2.16in}
\begin{overpic}[width=0.30\textwidth,height=0.30\textheight,angle=-90]{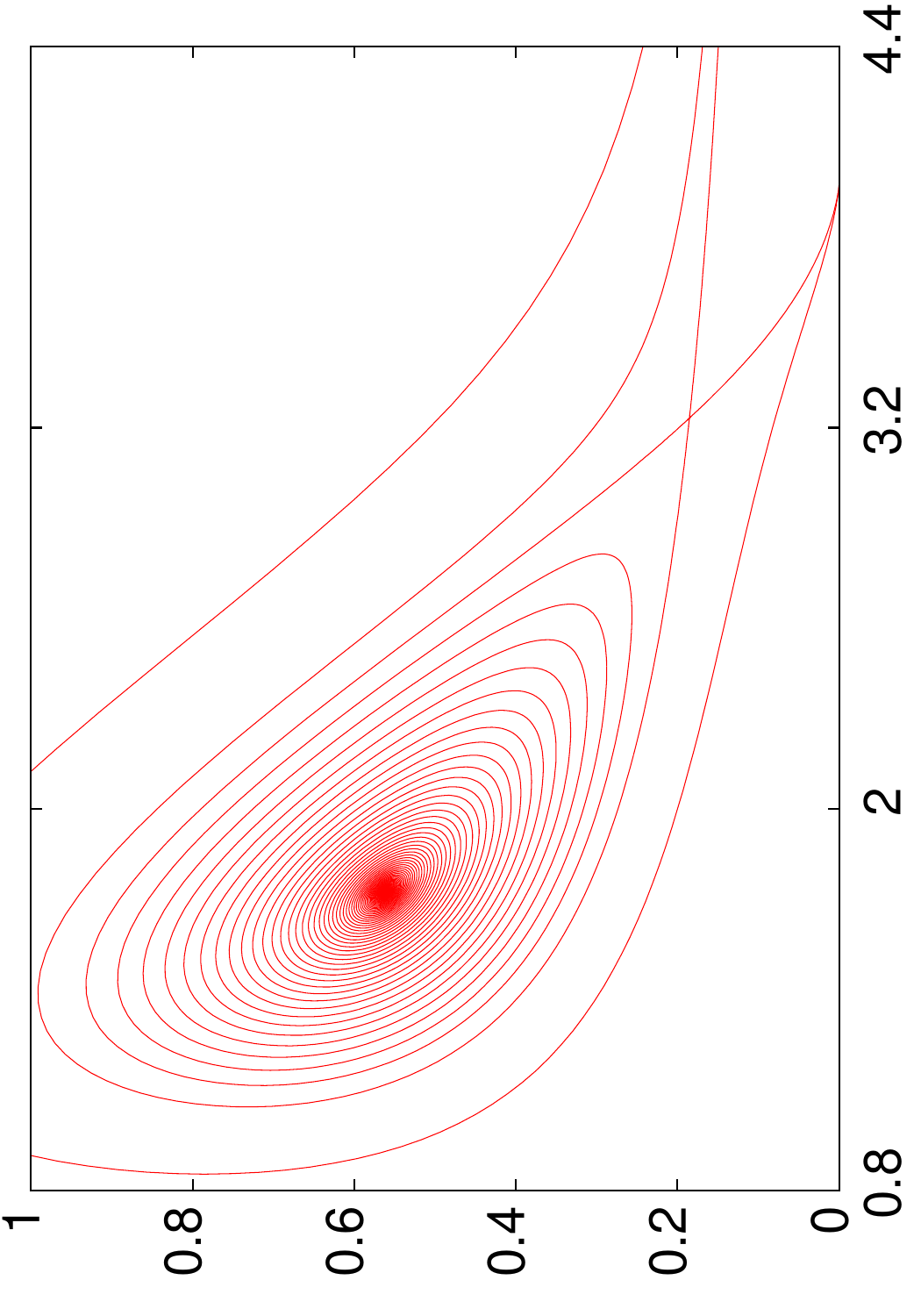}
\put(-4,35){{$Y$}} 
\put(51,-3){\footnotesize{$X$}} 
\put(31,40){$\bullet$}
\put(66.7,16.5){$\bullet$}
\put(84,5){$\bullet$}
\put(62,25){\vector(-3,4){1}}
\put(60,18.8){\vector(4,-1){1}}
\put(78,16.8){\vector(-4,1){1}}
\put(62,43.5){\vector(-3,4){1}}
\put(53.8,40.0){\vector(-3,4){1}}
\put(66,11.9){\vector( 4,-1){1}}
\put(74,12.4){\vector( 1,-1){1}}
\put(37,40){\small {${\rm E_2}$}} 
\put(84, 9){\small {${\rm E_1}$}} 
 
\end{overpic}\hspace*{-0.80in} 

\vspace*{0.10in} 
\hspace*{0.05in}(a)\hspace*{2.97in}(b) 

\vspace*{0.20in} 
\hspace*{-1.00in}
\begin{overpic}[width=0.30\textwidth,height=0.30\textheight,angle=-90]{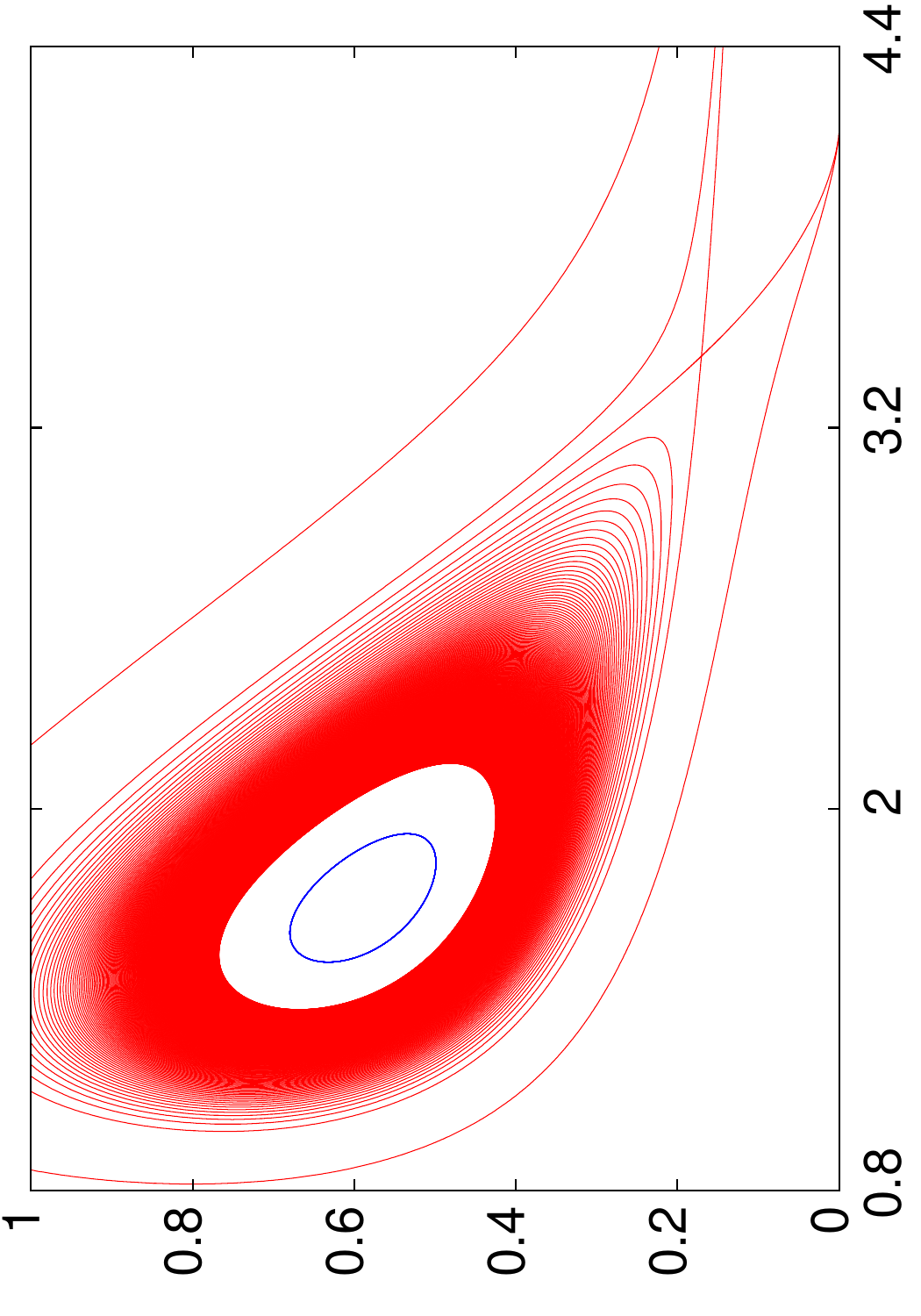}
\put(-4,35){{$Y$}} 
\put(51,-3){\footnotesize{$X$}} 
\put(30.5,40.8){$\bullet$}
\put(71.4,15.6){$\bullet$}
\put(88,5){$\bullet$}
\put(68,22.6){\vector(-3,4){1}}
\put(66,17.5){\vector(1,0){1}}
\put(80,16.3){\vector(-1,0){1}}
\put(63,43.2){\vector(-3,4){1}}
\put(63.8,30.0){\vector(-3,4){1}}
\put(66,12.7){\vector( 4,-1){1}}
\put(77.5,12.4){\vector( 4,-3){1}}
\put(34.0,41){\small {${\rm E_2}$}} 
\put(87, 9){\small {${\rm E_1}$}} 
 
\end{overpic}
\hspace*{0.40in}
\begin{overpic}[width=0.30\textwidth,height=0.30\textheight,angle=-90]{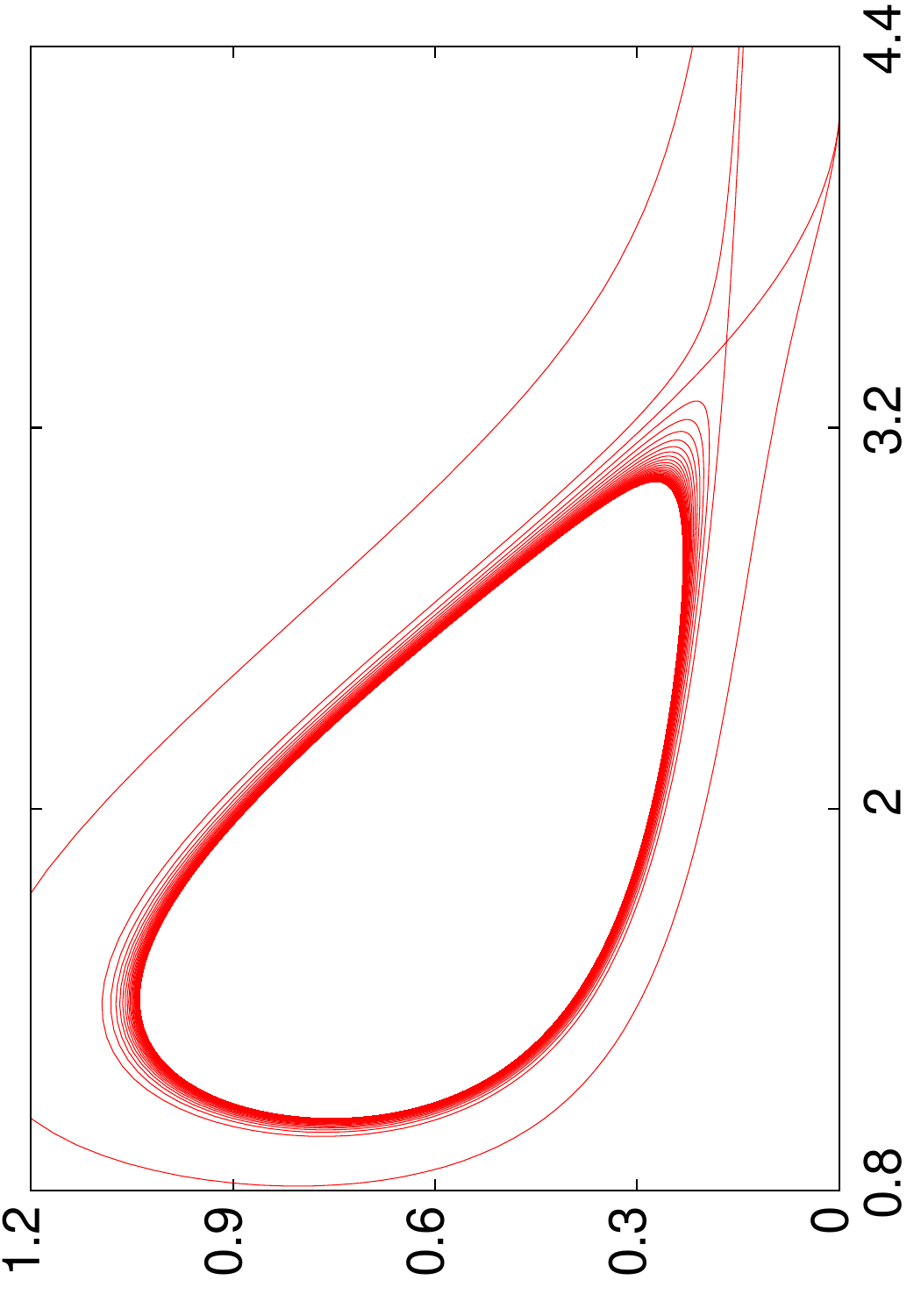}
\put(-4,42){{$Y$}} 
\put(51,-3){\footnotesize{$X$}} 
\put(30.0,35.7){$\bullet$}
\put(72.5,13.4){$\bullet$}
\put(88,5){$\bullet$}
\put(71,17.5){\vector(-3,4){1}}
\put(69,15.2){\vector(1,0){1}}
\put(80,14.5){\vector(-1,0){1}}
\put(63.8,36.3){\vector(-3,4){1}}
\put(70,10.9){\vector( 4,-1){1}}
\put(80.0,10.4){\vector( 4,-3){1}}
\put(35,35.7){\small {${\rm E_2}$}} 
\put(87, 9){\small {${\rm E_1}$}} 
 
\end{overpic}\hspace*{-0.80in} 

\vspace*{0.10in} 
\hspace*{0.05in}(c)\hspace*{2.97in}(d) 

\caption{Simulation of system \eqref{Eqn7} for $g=1.004$
and $n=0.41071668$: (a) bifurcation diagram in the $e$–$X_2$ plane;
the three vertical lines mark the values of $e$ used in the 
subsequent simulations. (b) $e=0.25223318$, trajectories converging to 
the stable equilibrium ${\rm E_2}$;
(c) $e=0.24223719$, two limit cycles surrounding stable ${\rm E_2}$ 
(outer stable, inner unstable); 
(d) $e=0.24$, one stable limit cycle enclosing the unstable ${\rm E_2}$.}
\label{Fig4}
\end{center}
\vspace*{0.00in} 
\end{figure}

\begin{figure}[!t] 
\vspace*{0.00in}
\begin{center} 
\begin{overpic}[width=0.50\textwidth,height=0.28\textheight]{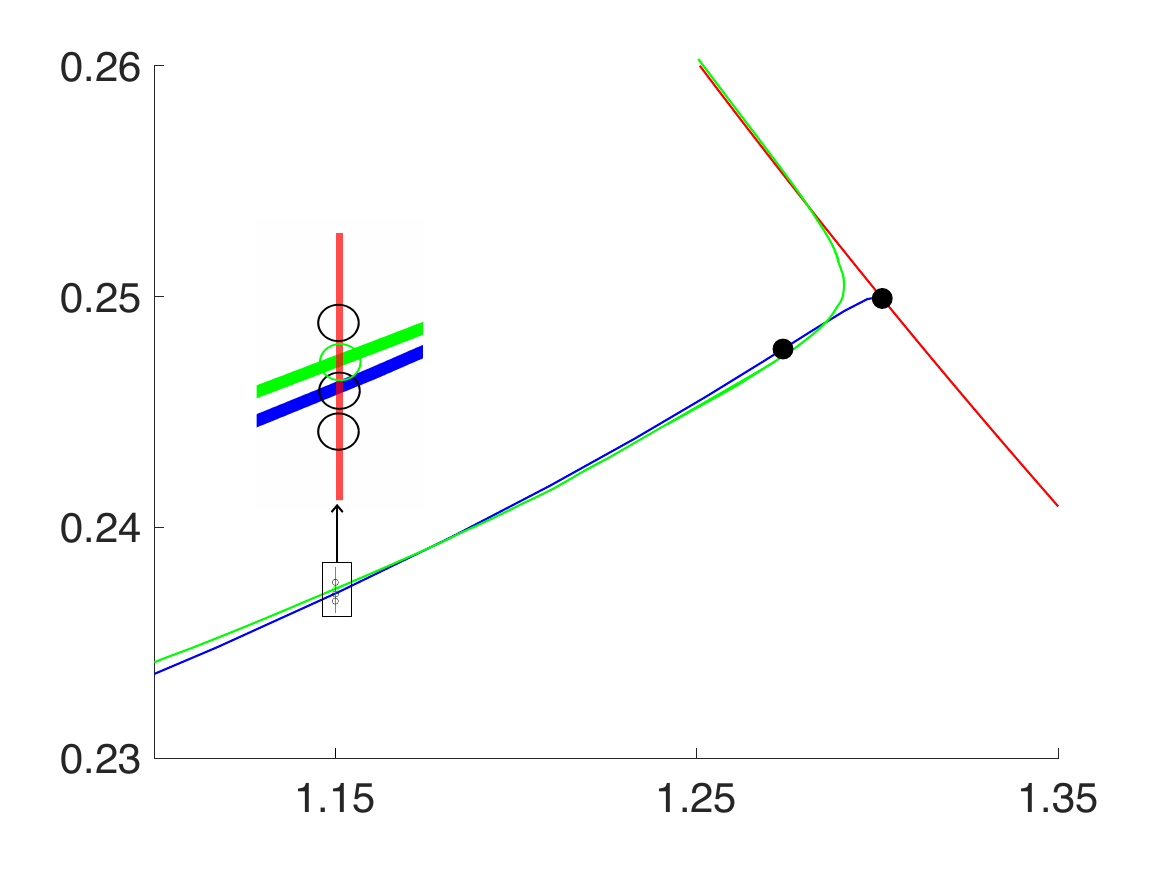}
\put(0,40){{$e$}} 
\put(48,1){{$g$}}
\put(36.2,29.0){{$\circ$}}  
\put(45,38){\color{blue}{\footnotesize {H}}} 
\put(45,30){\color{green}{\footnotesize {HL}}} 
\put(62,50){\footnotesize{GH}} 
\put(85,42){\color{red}\footnotesize{SN}} 
\put(77,54){\footnotesize{BT}} 
\put(20,50){\footnotesize{\color{black}(a)}} 
\put(20,46){\footnotesize{\color{black}(b)}} 
\put(33,43){\footnotesize{\color{black}(c)}} 
\put(33,39){\footnotesize{\color{black}(d)}} 
\end{overpic}

\caption{BT bifurcation diagram of the original system \eqref{Eqn7}, 
projected on the $g$-$e$ parameter plane, showing four points
(a), (b), (c) and (d)  for simulations taken along the line: $g=1.15$. 
The corresponding simulated phase portraits are shown in Figure \ref{Fig6}.
SN, H, GH, HL and BT denote saddle-node, Hopf, generalized Hopf, homoclinic 
loop and Bogdanov-Takens bifurcations, respectively.}
\label{Fig5}

\vspace{-0.60in} 
\hspace*{-0.50in}
\begin{overpic}[width=0.50\textwidth,height=0.43\textheight]{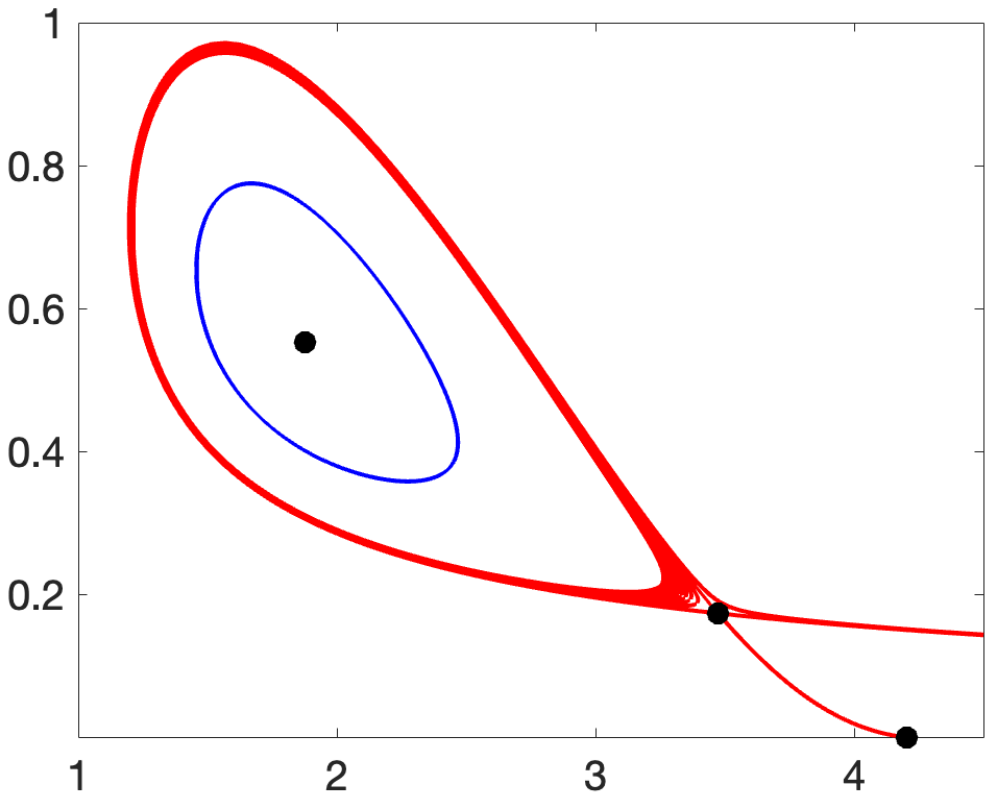}
\put(29,52){\color{black}\footnotesize{${\rm E_{2-}}$}} 
\put(55,40){\color{black}\footnotesize{${\rm E_{2+}}$}} 
\put(66,32.5){\color{black}\footnotesize{${\rm E_1}$}} 
\put(64,65){\color{black}\footnotesize{(a)}} 
\end{overpic}
\hspace*{0.00in}
\begin{overpic}[width=0.50\textwidth,height=0.43\textheight]{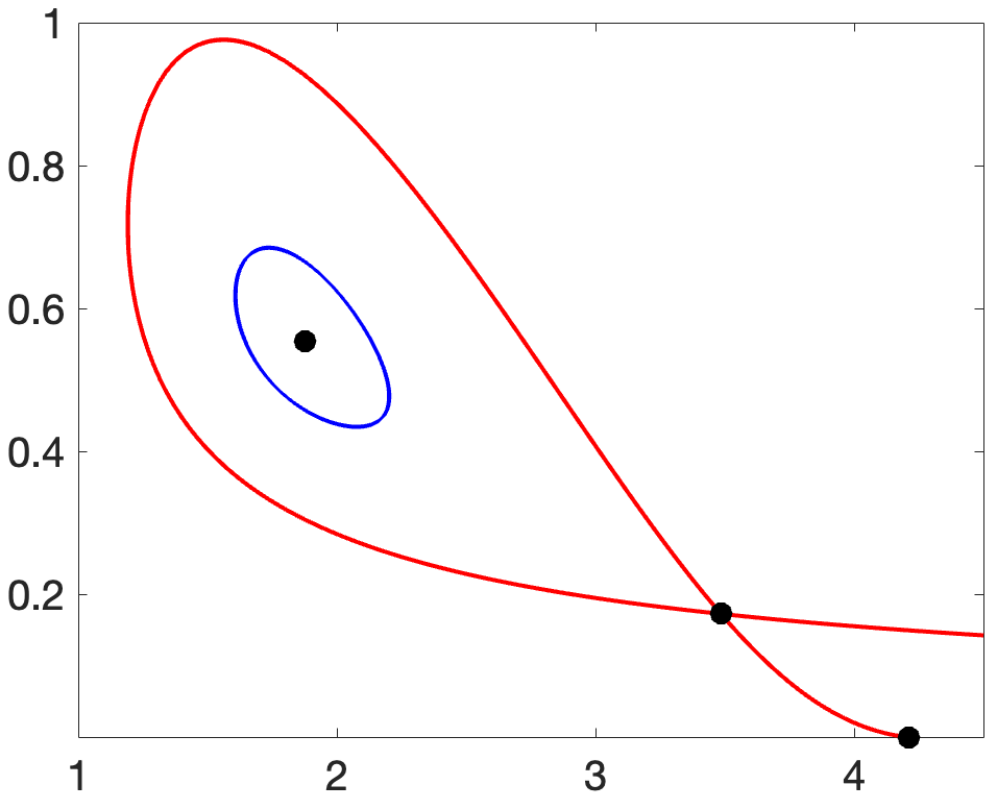}
\put(24,55){\color{black}\footnotesize{${\rm E_{2-}}$}} 
\put(55,39){\color{black}\footnotesize{${\rm E_{2+}}$}} 
\put(66,32.5){\color{black}\footnotesize{${\rm E_1}$}} 
\put(64,65){\color{black}\footnotesize{(b)}} 
\end{overpic}\hspace*{-0.50in} 

\vspace{-1.80in} 
\hspace*{-0.50in}
\begin{overpic}[width=0.50\textwidth,height=0.43\textheight]{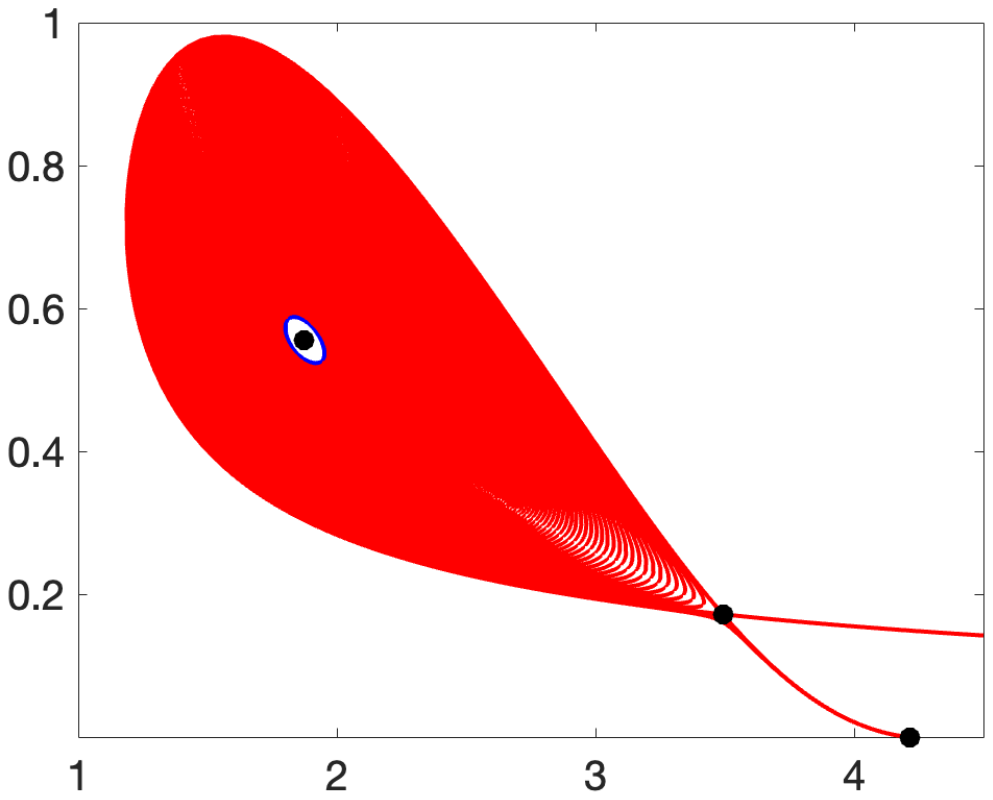}
\put(31,52){\color{black}\footnotesize{${\rm E_{2-}}$}} 
\put(55,39){\color{black}\footnotesize{${\rm E_{2+}}$}} 
\put(66,32.5){\color{black}\footnotesize{${\rm E_1}$}} 
\put(64,65){\color{black}\footnotesize{(c)}} 
\end{overpic}
\hspace*{0.00in}
\begin{overpic}[width=0.50\textwidth,height=0.43\textheight]{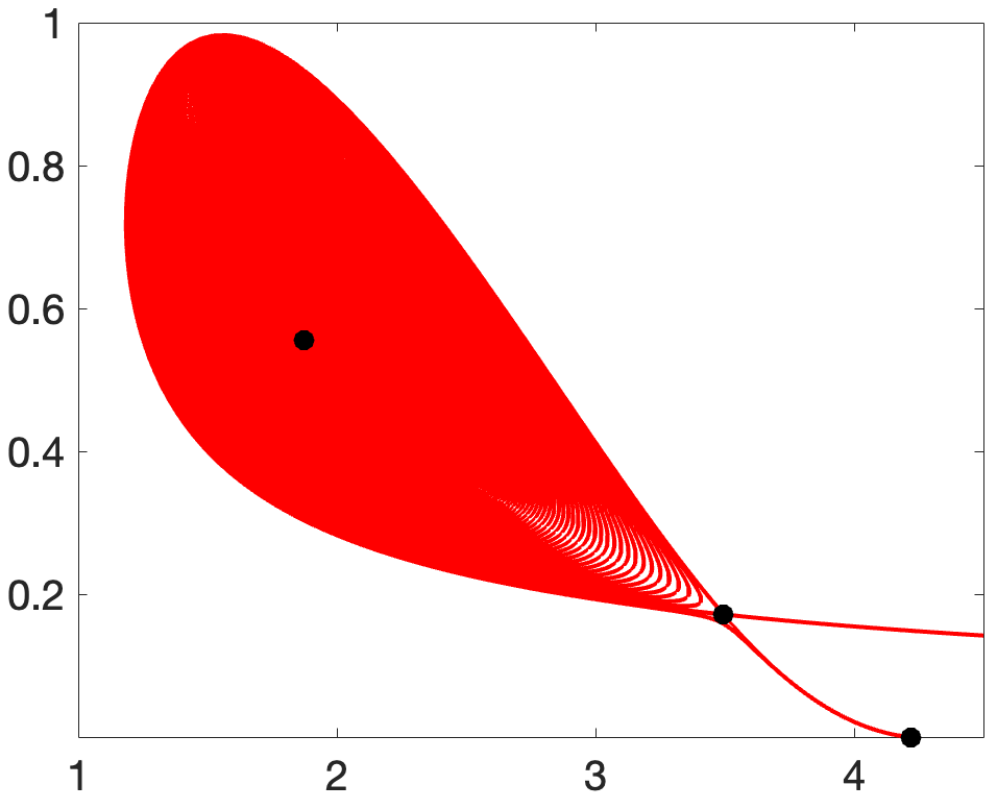}
\put(30,52){\color{black}\footnotesize{${\rm E_{2-}}$}} 
\put(55,39){\color{black}\footnotesize{${\rm E_{2+}}$}} 
\put(66,32.5){\color{black}\footnotesize{${\rm E_1}$}} 
\put(64,65){\color{black}\footnotesize{(d)}} 
\end{overpic}\hspace*{-0.50in} 

\vspace{-0.95in} 

\caption{Simulated phase portraits for the BT bifurcation of the original 
system \eqref{Eqn7} with $g=1.15$: (a) two limit cycles for $e=0.238$,
(b) 1 unstable limit cycle and stable homocilinic loop for $e=0.23748$,
(c) 1 unstable limit cycle for $e=0.23715$, and 
(d) unstable ${\rm E_2}$ for $e=0.237$.}
\label{Fig6}
\end{center}
\vspace*{-0.20in} 
\end{figure}

The simulated phase portraits are shown in Figure~\ref{Fig4} for 
$g=1.004$ and $n=0.41071668$. 
Figure~\ref{Fig6}(a) presents the bifurcation diagram in the $e$-$X_2$ plane,
around the Hopf critical point, 
$$
e=e_{\rm H}=\dfrac{94627 \sqrt{2}-126253}{31250} \approx 0.242233175.
$$
Figure~\ref{Fig4}(b), with $e=0.25223318$, illustrates the convergence 
of trajectories to the stable equilibrium ${\rm E_2}$.
Figure~\ref{Fig4}(c), obtained with $e=0.24223719$, shows two limit cycles 
surrounding the stable ${\rm E_2}$: the outer cycle stable, while  
the inner one unstable. 
Figure~\ref{Fig4}(d), for $e=0.24$, exhibits a single stable limit cycle 
enclosing the unstable equilibrium ${\rm E_2}$. 

The unstable limit cycle in Figure~\ref{Fig4}(c) is obtained using 
a backward-time integration scheme, under which the unstable limit cycle 
appears stable in the reversed-time simulation. 
It should be noted, however, that this backward-time approach 
is only applicable to two-dimensional systems and does not extend 
to higher dimensions.

\subsubsection{Simulation for the BT bifurcation of system~\eqref{Eqn7}}

In this section, we present simulations of the BT bifurcation as
summarized in Theorem~\ref{Thm3}.
Instead of using the normal form~\eqref{Eqn20}, we perform simulations
directly on the original system~\eqref{Eqn7} for easier comparison.
To clearly illustrate the bifurcation phenomena in the two-parameter
space, we fix $n=\tfrac{7}{20}$, which is close to
$n_c=\tfrac{1}{3}$, and choose $g$ and $e$ as the bifurcation parameters.
This yields the bifurcation diagram in Figure~\ref{Fig5}, where the
red, blue, and green curves correspond to the saddle-node (SN),
Hopf (H), and homoclinic loop (HL) bifurcations, respectively.
The points labeled GH and BT denote the generalized Hopf and
Bogdanov–Takens bifurcations.
The intersection point of the Hopf and homoclinic loop bifurcation
curves is found at $(g,e)=(1.176737,0.239162)$, indicated by the small
circle in Figure~\ref{Fig5}.

To simulate different bifurcation scenarios, we fix $g=1.15$ and
consider four values of $e$:
$$
e=0.238,\quad 0.23748,\quad 0.23715, \quad 0.237,
$$
corresponding to the points (a), (b), (c), and (d) in
Figure~\ref{Fig5}.  
When $e>0.238$, trajectories converge to the stable equilibrium
${\rm E_{2-}}$.  
At $e=0.238$ (point (a)), the system exhibits two limit cycles,
as shown in Figure~\ref{Fig6}(a).  
Note that point (a) lies very close to the intersection of the Hopf
and homoclinic loop curves, but does not lie exactly on either
bifurcation curve.  
This agrees with the scenario in Figure~\ref{Fig1}(d), where varying
perturbation parameters across the pink curve leads to the emergence of
two limit cycles.  
Thus, even though $e$ is varied only slightly, the system already
enters the regime of double limit cycles.

As $e$ decreases further to point (b), the trajectory touches the
homoclinic loop bifurcation curve.  
At this stage, the outer stable limit cycle collapses into a stable
homoclinic loop, while the inner unstable limit cycle persists,
as illustrated in Figure~\ref{Fig6}(b).  
Decreasing $e$ further breaks the homoclinic loop, and the trajectory
approaches the Hopf bifurcation curve, producing progressively smaller
limit cycles.  
This corresponds to point (c) in Figure~\ref{Fig5} and the simulation
in Figure~\ref{Fig6}(c).  
Finally, as $e$ decreases past the Hopf bifurcation curve, the limit
cycles disappear, and the equilibrium ${\rm E_{2-}}$ loses stability.
This situation is represented by point (d) in Figure~\ref{Fig5} and
the corresponding simulation in Figure~\ref{Fig6}(d).

These simulation results are fully consistent with the bifurcation
structure in Figure~\ref{Fig1}(d), where traversing from left to right
across the diagram passes through the double limit cycle (DLC) region.
This exactly matches the bifurcation phenomena observed in the
simulations.

\section{Bifurcation Analysis of System \eqref{Eqn8}}

The analysis of system \eqref{Eqn8} is considerably more involved than 
that of system \eqref{Eqn7}. With the aid of graphical approaches, 
we now present a complete characterization of its equilibria 
and their stability.

The equilibrium solutions of system \eqref{Eqn8} are given by 
\begin{equation}\label{Eqn32}
\begin{array}{ll} 
\textrm{Trivial equilibrium:} & {\rm E_0}= (0,\,0), \\[1.0ex] 
\textrm{Boundary equilibrium:} & {\rm E_1}= \Big( \dfrac{1}{e},\ 0 \Big),
\\[2.0ex]
\textrm{Positive equilibrium:} & {\rm E_2}
= \big(X_2,\, (1+a X_2)(1- e X_2) \big),
\end{array} 
\end{equation} 
where $X_2$ satisfies the cubic equation,
\begin{equation}\label{Eqn33}
L_1(X_2) = a^2 d e X_2^3+a d (2 e-a) X_2^2+[1+d e-a (2 d+g)] X_2-(d+g).
\end{equation}

Following the method of {\it hierarchical parametric analysis} 
\cite{ZengYu2023,ZengYuHan2024}, we solve $L_1=0$ for $g$ to obtain 
\begin{equation}\label{Eqn34}
g = \dfrac{X_2-d (1-e X_2) (1+a X_2)^2}{1+a X_2}.
\end{equation}
The reqirement $Y_2>0$ and $g>0$ yields the existence conditions for 
${\rm E_2}$:
\begin{equation}\label{Eqn35}
0<X_2<\dfrac{1}{e}, \quad 0<d< \dfrac{X_2}{(1-e X_2) (1+a X_2)^2}.  
\end{equation}
We summarize the existence and stability properties of the equilibria 
in the following lemma.

\begin{lemma}\label{Lem2}
The equilibria ${\rm E_0}$ and ${\rm E_1}$ exist for all 
positive parameter values. The positive equilibrium 
${\rm E_2}$ exits if $0<X_2<\frac{1}{e}$ and 
$0<d<\frac{X_2}{(1-e X_2)(1+a X_2)^2}$.
\begin{itemize}
\item 
${\rm E_0}$ is always a saddle point.
\item 
${\rm E_1}$ is a stable node if $g>\frac{1}{a+e}$ 
and a saddle point if $g<\frac{1}{a+e}$.
\item 
${\rm E_2}$ is locally asymptotically stable $({\rm LAS})$ for 
$\,e\!>\!0\,$ and one of the following conditions holds: 
\begin{enumerate}
\item[{\rm (a)}]
$\frac{1}{2e}<X_2 \le \frac{1}{e}$, $a>0$ and 
$d<\frac{X_2}{(1-e X_2)(1+aX_2)^2}$.

\vspace{0.10in}
\item[{\rm (b)}]
$0<X_2 < \frac{1}{2e}$, $0<a\le \frac{e}{1-2 e X_2}$ and 
$d<\frac{X_2}{(1-e X_2)(1+aX_2)^2}$.  

\vspace{0.05in}
\item[{\rm (c)}]
$0<X_2 < \frac{1}{2e}$, $\frac{e}{1-2 e X_2}<a<\min\Big\{
\frac{e+1}{1-2 e X_2},\,\frac{e+ \sqrt{(1-e X_2)/X_2}}{1-2 e X_2} \Big\}$
\\[1.0ex] 
and $\ \frac{X_2 [a (1-2 e X_2)-e]}{(1-e X_2) (1+a X_2)^2} 
<d<\min\Big\{ \frac{X_2}{(1-e X_2)(1+a X_2)^2},\,
\frac{1}{(1+a X_2)^2[a (1-2 e X_2)-e]} \Big\}$. 
\end{enumerate} 
\end{itemize}
\end{lemma}

\begin{proof}
Existence follows directly from \eqref{Eqn32} and \eqref{Eqn35}. 

For stability, Jacobian of \eqref{Eqn8} is 
\begin{equation}\label{Eqn36} 
J(X,Y) = \left[
\begin{array}{cc}
1 - 2 e X - \dfrac{Y}{(1+aX)^2} & -\,\dfrac{X}{1+aX} \\[2.0ex] 
\dfrac{Y}{(1+aX)^2} & -g + \dfrac{X}{1+aX} - 2 d Y \end{array} \right]. 
\end{equation} 

\noindent 
\text{At ${\rm E_0}$}: Substituting $(0,0)$ yields eigenvalues 
$1$ and $-g$, hence ${\rm E_0}$ is a saddle point.

\noindent 
\text{At ${\rm E_1}$}: Substituting $\left(\frac{1}{e}, 0\right)$ 
gives eigenvalues $-1$ and $\frac{1}{a+e}-g$, so 
${\rm E_1}$ is a stable node if $g>\frac{1}{a+e}$ \\
\hspace*{0.45in} and a saddle point if $g<\frac{1}{a+e}$.

\noindent 
\text{At ${\rm E_2}$}: The trace and determinant of the Jacobian 
evaluated at ${\rm E_2}$ are 
\begin{equation}\label{Eqn37} 
\begin{array}{rl} 
{\rm tr}(J({\rm E_2})) = \!\!\!\! & \dfrac{1}{1+a X_2} 
\Big\{ X_2 \big[a (1-2 e X_2)-e \big] - d (1-e X_2) (1+a X_2)^2 \Big\}, 
\\[2.0ex]  
\det(J({\rm E_2})) = \!\!\!\! & 
\dfrac{X_2 (1-e X_3)}{(1+a X_2)^2} 
\Big\{ 1 - d \big[ a(1-2 e X_2)-e \big] (1+a X_2)^2 \Big\}.   
\end{array}
\end{equation}
${\rm E_2}$ is LAS if ${\rm tr}(J({\rm E_2})) < 0$ 
and $ \det(J({\rm E_2})) > 0$. 
Combining these with the existence conditions yields three parameter regimes: 
\begin{enumerate}
\item[{(1)}]
$1-2 e X_2 \le 0$: $\frac{1}{2e} \le X_2 < \frac{1}{e}$, $a>0$,
$d< \frac{X_2}{(1-e X_2)(1+a X_2)^2}$.  

\item[{(2)}]
$1-2 e X_2 > 0$ and $a (1-2 e X_1)-e \le 0$: 
$0<X_2<\frac{1}{2e} $, $0<a\le \frac{e}{1-2e X_2}$ and 
$d< \frac{X_2}{(1-e X_2)(1+a X_2)^2}$. 

\vspace{0.05in} 
\item[{(3)}]
$1-2 e X_2 > 0$ and $a (1-2 e X_1)-e > 0$: 
$0<X_2<\frac{1}{2e} $, $a> \frac{e}{1-2e X_2}$, and 
$$
\frac{X_2 [a (1-2 e X_2)-e]}{(1-e X_2) (1+a X_2)^2}
< d < \frac{X_2}{(1-e X_2)(1+a X_2)^2}.
$$
The additional requirement $ \det(J({\rm E_2})) > 0$ 
in this case gives 
$ d < \frac{X_2}{(1-e X_2)(1+a X_2)^2}$, 
which leads to the bound on $a$ in case (c).
\end{enumerate}
\vspace*{-0.20in} 
\end{proof}

We now turn to bifurcations of system \eqref{Eqn8} from ${\rm E_2}$, 
considering in turn the saddle-node (SN) bifurcation, 
the Bogdanov–Takens (BT) bifurcation, and the Hopf bifurcation.

\subsection{SN bifurcation}

In this case, the approach used in Section~2.1, determining the SN critical 
point from the equation of equilibrium solution, cannot be applied directly. 
Instead, we compute the eigenvalues of the Jacobian evaluated at ${\rm E_2}$.

We first state the result.

\begin{theorem}\label{Thm5}
System \eqref{Eqn8} undergoes a saddle-node bifurcation at the 
critical point $d=d_{\rm sn}$ if $a (1-3eX_2)- 2e \ne 0$. 
\end{theorem} 

\begin{proof}
The determinant 
of the Jacobian evaluated at ${\rm E_2}$, with $g$ given by \eqref{Eqn34}, 
is 
$$ 
\det(J({\rm E_2})) =  \dfrac{X_2(1-e X_2)}{(1+a X_2)^2}
\big\{1 - d (1+a X_2)^2 \big[a (1-2 e X_2)-e \big] \big\}. 
$$ 
Solving $\det(J({\rm E_2}))=0$ yields 
$$ 
d_{\rm sn} = \dfrac{1}{(1+a X_2)^2 [a (1-2 e X_2)-e]},
$$ 
which requires $ X_2<\frac{1}{2e}$ and $a>\frac{e}{(1-2 e X_2)}$ to 
ensure $d_{\rm sn}>0$. 

The existence condition for $d$ further imposes 
$$ 
d_{\rm sn} < \dfrac{X_2}{(1-e X_2) (1+a X_2)^2} \quad 
\Longrightarrow \quad a>\dfrac{1}{X_2 (1-2 e X_2)}.  
$$ 
Summarizing, the existence conditions for an SN bifurcation are:
\begin{equation}\label{Eqn38}
0<X_2 < \dfrac{1}{2e}, \quad a > \dfrac{1}{X_2 (1-2 e X_2)}, \quad 
(X_{\rm 2 sn}, Y_{\rm 2 sn})= \big(X_2,\,(1-eX_2)(1+a X_2) \big), 
\end{equation} 
at which the Jacobian matrix has a zero eigenvalue. 
Under these conditions, the trace of $J({\rm E_2})$ is
$$ 
{\rm tr_{sn}}(J({\rm E_2})) = 
\dfrac{{\rm tr_{sn}^*}}{(1+a X_2) [a (1-2 e X_2)-e]},
$$ 
where 
$$  
{\rm tr_{sn^*}} = 
 X_2(1-2 e X_2)^2 a^2 - 2 e X_2(1-2 e X_2) a+e (1+e) X_2-1, 
$$
and ${\rm tr_{sn^*}} \ne 0$ is required. 

To analyze the SN bifurcation, set $d=d_{\rm sn}+\mu$, 
and apply the affine transform,
\begin{equation}\label{Eqn39}
\left(\!\! \begin{array}{c} X \\ Y \end{array} \!\! \right) 
\!=\! \left(\!\! \begin{array}{c} X_2 \\[2.0ex] 
 (1 \!-\! eX_2)(1 \!+\! a X_2) \end{array} \!\! \right) 
\!\! \left[\!\!
\begin{array}{cc}
\dfrac{X_2 [a(1 \!-\! 2e X_2) \!-\! e]}{1+a X_2} & \dfrac{e X_2-1}
{ (1 \!+\! a X_2)[a (1 \!-\! 2 e X_2) \!-\! e] }
\\[2.5ex] 
-\dfrac{X_2}{1+a X_2} & \dfrac{1 - e X_2}{1 + a X_2}
\end{array} 
\!\! \right] \!\! 
\left(\!\! \begin{array}{c} u \\ v \end{array} \!\! \right). 
\end{equation}
Substituting this into \eqref{Eqn8} and expanding to second order yields 
$$
\begin{array}{rl} 
\dfrac{du}{d \tau} = \!\!\!\!& S_{01} \, \mu 
+ S_{120}\, x_1^2 + \cdots + O(|u,v,\mu|^3), \\[1.0ex] 
\dfrac{dv}{d \tau} = \!\!\!\!& S_{01} \, \mu 
+ S_{220}\, x_1^2 + \cdots + O(|u,v,\mu|^3), \\[1.0ex] 
\end{array} 
$$
where 
$$
\begin{array}{rl} 
S_{01} = \!\!\!\! & \dfrac{(1-e X_2)^2 (1+a X_2)^3 [a (1-2 e X_2)-e]}
         {{\rm tr_{sn}^*}}, \quad 
S_{120} = \dfrac{a X_2^2 (1-e X_2) [a (1-3 e X_2)-2 e]}
         {(1+a X_2)^3\, {\rm tr_{sn}^*}}, \\[3.0ex]  
S_{220} = \!\!\!\! & \dfrac{a X_2^2[ a(1-2 e X_2)-e]
                   [e X_2^2 (1-2 e X_2) a^2+a X_2 (1-3 e X_2)a+e(1+e) X_2-1]}
                    {(1+a X_2)^3 \, {\rm tr_{sn}^*}}.
\end{array}
$$  

Using center manifold theory, we write 
$$ 
v = h(u,\mu) = h_{01} \, \mu + h_{20}\, u^2 + h_{11}\, u \mu + h_{02}\, \mu^2 
+ O(|u,v\mu|^3),
$$  
defining the center manifold, 
$$ 
W^c = \big\{(u,v)| v = h(u,\mu) \big\}. 
$$ 
Substituting the equation 
$\frac{dv}{d \tau} = (2 h_{20} u + h_{11} \mu ) \frac{du}{d\tau}$
and matching coefficients gives the $h_{ij}$ solutions. Inserting 
the center manifold expression into the $u$-equation, we obtain 
$$ 
\dfrac{du}{d \tau} = S_{01} \, \mu 
+ S_{20}\, x_1^2 + O(|u,\mu|^3),  \quad (S_{20}=S_{120}),  
$$  
showing that system \eqref{Eqn8} exhibits a saddle-node bifurcation 
provided $a (1-3 e X_2)-2 e \ne 0$. 
\end{proof}

\subsection{Cusp bifurcation}

From the analysis in the previous section, the critical values
\begin{equation}\label{Eqn40}
a_{\rm cusp} = \dfrac{2 e}{1-3 e X_2}, \quad 
d_{\rm cusp}= \dfrac{(1- 3 e X_2)^3}{e ( 1- e X_2)^3}, 
\quad  \textrm{for} \ \ X_2 < \dfrac{1}{3e},  
\end{equation}
are identified as candidates for a cusp bifurcation of system \eqref{Eqn8}. 
Indeed, we have the following result. 

\begin{theorem}\label{Thm6}
System \eqref{Eqn8} undergoes a cusp bifurcation at the equilibrium 
$(X_{\rm cusp},Y_{\rm cusp})= \big(X_2,\,\frac{(1-e X_2)^2}{1-3e X_2} \big)$ 
near the critical point $(d,a)=(d_{\rm cusp},a_{\rm cusp})$ 
provided that $ X_2 < \frac{1}{3e}$. 
\end{theorem}

\begin{proof}
At the critical point $(d_{\rm cusp},a_{\rm cusp})$, the determinant 
of the Jacobian vanishes, and its trace is given by  
$$ 
\begin{array}{rl}
{\rm tr_{cusp}}(J({\rm E_2})) = \!\!\!\! & 
\dfrac{{\rm tr_{cusp}^*}}{e(1-e X_2)} 
\left\{
\begin{array}{ll}
>0, & \textrm{if} \ \ 0<X_2 < \dfrac{e+6-\sqrt{e(e+8)}}{2 e (e+9)} \\[2.0ex] 
<0, & \textrm{if} \ \ \dfrac{e+6-\sqrt{e(e+8)}}{2 e (e+9)} < X_2 
< \dfrac{1}{3e}, 
\end{array} 
\right. 
\end{array}
$$
where 
$$
{\rm tr_{cusp}^*} = e^2 X_2 (1-e X_2)-(1-3 e X_2)^2 \ne 0.  
$$

To analyze the cusp bifurcation, let $d=d_{\rm cusp}+\mu$, 
and perform the affine transform,
\begin{equation}\label{Eqn41}
\left(\! \begin{array}{c} X \\ Y \end{array} \! \right) 
= \left(\! \begin{array}{c} X_2 \\[2.0ex] 
 \dfrac{(1 - e X_2)^2}{1 - 3e X_2} \end{array} \! \right) 
\! \left[\!
\begin{array}{cc}
e X_2 & -\, \dfrac{(1- 3 e X_2)^2}{e (1-e X_2) }
\\[2.5ex] 
\dfrac{X_2 (1- 3 e X_2)}{e (1-e X_2) } & 1 - 3 e X_2 
\end{array} 
\! \right] \! 
\left(\! \begin{array}{c} u \\ v \end{array} \! \right). 
\end{equation}
Substituting this into \eqref{Eqn8} and expanding up to 
third order gives 
$$
\begin{array}{rl} 
\dfrac{du}{d \tau} = \!\!\!\!& C_{01} \, \mu 
+ C_{130}\, x_1^3 + \cdots + O(|u,v,\mu|^4), \\[1.0ex] 
\dfrac{dv}{d \tau} = \!\!\!\!& -\,C_{01} \, \mu 
+ C_{230}\, x_1^3 + \cdots + O(|u,v,\mu|^4), \\[1.0ex] 
\end{array} 
$$
where 
$$
\begin{array}{rl} 
C_{01} = \!\!\!\! & \dfrac{e(1-e X_2)^5}{(1- 3 e X_2)^2 {\rm tr_{cusp}^*}}, 
\qquad 
C_{130} = \dfrac{2 e^2 X_2^2 (1-2 e X_2) (1-3 e X_2)^4}
         {(1-e X_2)^4 {\rm tr_{cusp}^*}}, \\[3.0ex]  
C_{230} = \!\!\!\! & 
-\,\dfrac{2 e^3 X_2^3 (1-3 e X_2 )^3 [e (1-e X_2)+(1-3 e X_2)]}
         {(1-e X_2)^4 {\rm tr_{cusp}^*}}. 
\end{array}
$$  

Applying center manifold theory, we set 
$$ 
v = h(u,\mu) = h_{01} \mu + h_{20} u^2 + h_{11} u \mu + h_{02} \mu^2 
+ h_{30} x_1^3 + h_{21} x_1^2 \mu + h_{12} x_1 \mu^2 + h_{03} \mu^3 
+ O(|u,v\mu|^4),
$$  
which defines the center manifold, 
$$ 
W^c = \big\{(u,v)| v = h(u,\mu) \big\}. 
$$ 
Substituting
$$
\dfrac{dv}{d \tau} = \big(2 h_{20} u + h_{11} \mu 
3 h_{30} u^2 + 2 h_{21} u \mu + h_{12} \mu^2 \big) \dfrac{du}{d\tau}
$$
and matching coefficients gives explicit expressions for 
$h_{ij}$. Inserting $v=h(u,\mu)$ into the $u$-equation yields 
$$ 
\dfrac{du}{d \tau} = C_{01} \, \mu + C_{02}\, \mu^2 + C_{03}\, \mu^3  
+ C_{30}\, u^3  + C_{11} \, u\, \mu + C_{21} \, u^2 \mu 
+ C_{12}\, u \mu^2 + O(|u,\mu|^4),  
$$ 
where $C_{30} = C_{130}$, and other $C_{ij}$ are rational functions 
in $e$ and $X_2$.
 
Finally, using the nonlinear transformation,
$$ 
u = y - \tfrac{ C_{01} C_{11} -(C_{01} C_{12} + C_{02} C_{11}) \mu }
              {2 C_{01}^2}\, y^2 
      + \tfrac{C_{02}}{C_{01}}\, y\, \mu 
      + \tfrac{(C_{11}^2 + 2 C_{01} C_{21}}{6 C_{01}^2} \, y^3 
      + \tfrac{C_{01} C_{12}+C_{02} C_{11}}{2 C_{01}^2}\, y^2 \mu  
      + \tfrac{C_{03}}{C_{01}}\, y\, \mu^2, 
$$ 
we obtain the reduced equation on the center manifold: 
\begin{equation}\label{Eqn42} 
\dfrac{d y}{d \tau} = C_{01}\, \mu + C_{30}\, y^3 + O(|y,\mu|^4), 
\quad (C_{30} = C_{130} \ne 0), 
\end{equation} 
which confirms that system \eqref{Eqn8} exhibits a cusp bifurcation. 
\end{proof}

\subsection{BT bifurcation}

In this section, we analyze the Bogdanov-Takens (BT) bifurcation 
which occurs at the equilibrium ${\rm E_2}$ of system \eqref{Eqn8}. 
The BT bifurcation condition requires 
${\rm tr}(J({\rm E_2}))=\det(J({\rm E_2}))=0$. From 
\eqref{Eqn37}, setting ${\rm tr}(J({\rm E_2}))=0$ yields 
$$
d_c = \dfrac{X_2 [a (1-2 e X_2)-e]}{(1-e X_2) (1+a X_2)^2}.
$$
The condition $d_c>0$ requires 
$$
0<X_2<\frac{1}{2e} \quad \textrm{and} \quad a>\frac{e}{1-2e X_2}.
$$ 
The trace then becomes 
$$ 
{\rm tr}(J({\rm E_2}))= -\,\dfrac{1}{1-e X_2} 
\left\{ X_2 \big[ a(1 - 2 e X_2) -e \big]^2 - \dfrac{1-eX_2}{X_2} \right\}. 
$$ 
Solving ${\rm tr}(J({\rm E_2}))=0$ for $a$ gives 
$$ 
a_- = \dfrac{e X_2 - \sqrt{X_2 (1-e X_2)}}{X_2(1-2 e X_2)}, \quad 
a_+ = \dfrac{e X_2 + \sqrt{X_2 (1-e X_2)}}{X_2(1-2 e X_2)}
\stackrel{\triangle}{=} a_c.  
$$ 
It can be verified that $d_c(a=a_-)<0$. 
Therefore, the BT bifurcation point is defined by 
$$ 
(a,d)=(a_c,d_c) \quad \Longrightarrow \quad 
g = \dfrac{X_2 (1-2 e X_2) [(1+e) X_2-1]}{(X_2+ g_{c1}) g_{c2}}, 
$$  
where 
$$ 
g_{c1} = \sqrt{X_2(1-e X_2)}, \quad g_{c2}= 1-e X_2+\sqrt{X_2(1-e X_2)}. 
$$ 
The positivity condition $g>0$ requires 
\begin{equation}\label{Eqn43} 
\dfrac{1}{1+e} < X_2 < \dfrac{1}{2e} \quad \textrm{and} \quad 0<e<1. \\[1.5ex]
\end{equation}

As in Section 2, we first determine the codimension of the BT 
bifurcation using the normal form without unfolding. 
We then derive the parametric normal form (PSNF) to examine the 
bifurcation structure in detail. This analysis rigorously establish 
the existence of the codimension-three (parabolic type) 
BT bifurcation, and present a systematic and detailed bifurcation 
study based on PSNF theory.

\subsubsection{Codimension of BT bifurcation}

\begin{theorem}\label{Thm7}
System \eqref{Eqn8} undergoes a Bogdanov-Takens bifurcation from the 
equilibrium 
$$
{\rm E_{bt}} = \big(X_2,\,(1-eX_2)(1+a_c X_2)\big)
$$ 
at the critical point $(a,d)=(a_c,d_c)$. More precisely, define 
\begin{equation}\label{Eqn44}  
X_{2c} = \dfrac{e+6-\sqrt{e (e+8)}}{2 e (e+9)}. 
\end{equation} 
Then, 
\begin{enumerate}
\item[{$(1)$}] 
The BT bifurcation has codimension two if 
$$
X_2 \in \Big(\frac{1}{1+e},X_{2c} \Big) \bigcup \Big(X_{2c},\frac{1}{2e}\Big), 
\quad 0<e<1. 
$$
\item[{$(2)$}] 
The BT bifurcation has codimension three if 
$$
X_2 = X_{2c}, \quad 0<e<1.
$$  
\end{enumerate}  
\end{theorem}

\begin{proof} 
At the critical point $(a,d)=(a_c,d_c)$, we introduce the 
affine transformation, 
$$
\begin{array}{rl}  
\left(\begin{array}{c} X \\ Y \end{array} \right) 
= \!\!\!\!\! & \left(\begin{array}{c} X_2 \\[2.0ex] 
 (1-eX_2)(1+a_c X_2) \end{array} \right) 
\\[5.0ex]
& + \left[
\begin{array}{cc}
\dfrac{(1-2 e X_2) g_{c1}}{g_{c2}}
& \dfrac{-X_2 (1-2 e X_2)}{g_{c2}} \\[2.0ex]
 \dfrac{(1-e X_2) (1-2 e X_2)}{g_{c2}}
& \dfrac{-(1-e X_2)(1-2 e X_2) (X_2+g_{c1})}{g_{c2}^2} 
\end{array} 
\right] 
\left(\begin{array}{c} u \\ v \end{array} \right). 
\end{array}
$$ 
into system \eqref{Eqn8}. Expanding to third order gives 
\begin{equation}\label{Eqn45}
\begin{array}{rl}
\dfrac{du}{d \tau} = \!\!\!\! & 
v + \tfrac{e X_2 (1-2 e X_2) \{(1-3 e X_2+X_2) g_{c1}
-X_2 [(2-e) (1-e X_2)-e X_2]\}}{(1-e X_2) [(1+e) X_2-1]^2}\, u^2  \\[1.5ex]
& +\, \tfrac{(1-2e X_2) [2 (1-2 e X_2) g_{c1}-(1-2 e X_2) (1-e X_2+X_2)]}
  {(1-e X_2) [(1+e) X_2-1]^2}\, uv 
\\[1.5ex]
& -\,\tfrac{ X_2 [e (e^2-10 e+5) X_2^2-2 (e^2-7 e+2) X_21+e-4]
          +[(5 e^2-10 e+1) X_2^2+6 (1-e) X_2+1] g_{c1}
         }{(1-e X_2)^2 [(1+e) X_2-1]^4} 
\\[1.0ex]
& \quad \times \, (1-2 e X_2)^3 (e X_2 u^3 + u^2 v), 
\\[1.5ex] 
\dfrac{du}{d \tau} = \!\!\!\! & 
\Big( \tfrac{
-2 [g_{c1}^4 (e^2 X_2 +10 e X_2-7 X_21+2)+2 g_{c1}^2 X_2 (3 X_2-2)-X_2^3]}
{(1- e X_2) [(1+e) X_2-1]^4}
\\[1.5ex]
&\ +\, \tfrac{[g_{c1}^4 (e^2+5 e-25+1/X_2)-g_{c1}^2 (3 X_2^2-26 X_2+2)
        +X_2^2 (2 X_2-7) ] g_{c1} }{(1- e X_2) [(1+e) X_2-1]^4} \Big) 
(1-2 e X_2)^2 u^2 
\\[1.5ex]
& -\, \tfrac{(1-2 e X_2)^3 \{4 X_2(1-e X_2) (1-e X_2+X_2)
          - [(1-e X_2)^2+6 X_2 (1-e X_2)+X_2^2] g_{c1} \}}
         {(1-e X_2)^2 [(1+e) X_2-1]^3}\, u v 
\\[1.5ex] 
& +\, \tfrac{(1-2 e X_2)^2 [2 X_2(1-e X_2)-(1-e X1+X_2) g_{c1}]} 
        {(1-e X_2)^2 [(1+e) X_2-1]^2}\, v^2
\\[1.5ex]
& -\, \tfrac{
(1- e X_2) [(5 e^2-10 e+1) X_2^2+6 (1-e) X_2+1] 
+ [e (e^2-10 e+5) X_2^2-2 (e^2-7 e+2) X_2+e-4] g_{c1} }
{(1-e X_2)^3 [(1+e) X_2-1]^4}  
\\[1.0ex]
& \quad \times \, X_2 (1-2 e X_2)^4 (e X_2 u^3 + u^2 v). 
\end{array} 
\end{equation}  

We now apply the third-order nonlinear transformation,
$$ 
\begin{array}{rl}
u = \!\!\!\! & a_{20}\, y_1^2 + a_{30}\, y_1^3, \\[1.0ex] 
v = \!\!\!\! & b_{20}\, y_1^2 + b_{11}\, y_1 y_2 
+ b_{30}\, y_1^3 + b_{21}\, y_1^2 y_2,
\end{array}
$$ 
and the time rescalnig 
$$ 
\tau = (1+t_{10}\,y_1)\, \tau_1, 
$$
where the coefficients $a_{ij}$, $b_{ij}$ and $t_{10}$ 
chosen appropriately.  
Substituting into \eqref{Eqn45} yields the third-order SNF:
\begin{equation}\label{Eqn46}
\dfrac{d y_1}{d \tau_1} = y_2, \qquad \dfrac{d y_2}{d \tau_1} 
= c_{20}\, y_1^2 + c_{11}\, y_1 y_2 + c_{30}\, y_1^3, 
\end{equation} 
where 
$$
\begin{array}{rlrl} 
c_{20} =\!\!\!\! & \dfrac{X_2 (1-2 e X_2)^2}{(1-e X_2)[(1+e)X_2-1]^4}\ c_{20a},
& \quad c_{20a} = \!\!\!\! &  2 X_2 K_1 - g_{c1} K_2, 
\\[2.0ex] 
c_{11} =\!\!\!\! & \dfrac{1-2 e X_2} {(1-e X_2)^2 [(1+e)X_2-1]^3}\ c_{11a}, 
& \quad 
c_{30} =\!\!\!\! &\dfrac{X_2^2 (1-2 e X_2)^4\,c_{30n}}{(1-e X_2)^2 c_{30d}}. 
\end{array}
$$
Here,  
$$ 
\begin{array}{rl} 
K_1 = \!\!\!\! & 
e^2 (7-10 e-e^2) X_2^3+2 e (e^2+9 e-4) X_2^2 -(e^2+10 e-2) X_2+2, \\[1.0ex] 
K_2 = \!\!\!\! & e (3-e) (e^2+8 e-1) X_2^3+(2e^3+9e^2-24e+1) X_2^2
+(1-e)(e+6) X_2+1, \\[1.0ex] 
c_{11a} = \!\!\!\! & -2 X_2 (1-e X_2)
\big[e^2 (e^2-10 e+5) X_2^3-2 e (e^2-10 e+3) X_2^2 +(e^2-12 e+2) X_2 +2 \big] 
\\[1.0ex]
& +\big[2 e^2 (5 e^2-10 e+1) X_2^4 -2 e (13 e^2-22 e+1) X_2^3
+(23 e^2-30 e+1) X_2^2 
\\[1.0ex]
& \quad -2 (4 e-3) X_2 +1 \big] g_{c1}, 
\end{array}
$$ 
$$
\begin{array}{rl} 
c_{30n} = \!\!\!\! & 
-\,e^2 (1+e) (e^2-26 e+21) X_2^4+e (3 e^3-44 e^2-65 e+30) X_2^3
\\[1.0ex]
& -(3 e^3-12 e^2-70 e+5) X_2^2+(e^2+8 e-10) X_2-1 
\\[1.0ex]
& +\,[\, e (1 \!+\! e) (9 e^2 \!-\! 34 e \!+\! 5) X_2^3
\!-\! (18 e^3 \!+\! 5 e^2 \!-\! 60 e \!+\! 1) X_2^2
\!+\! (9 e^2 \!+\! 35 e \!-\! 10) X_2 \!-\! 5\,]\, g_{c1}, 
\\[1.0ex]
c_{30d} = \!\!\!\! & 
-\,2 X_2 (1\!-\! e X_2) \big[(1 \!-\! e X_2) (1 \!-\! e X_2 
\!+\! 10 X_2) \!+\! 5 X_2^2 \big] 
\big[ 5 (1 \!-\! e X_2) (1 \!-\! e X_2+2 X_2)+X_2^2 \big] 
\\[1.0ex] 
&-\, (1 - e X_2 + X_2) 
\big\{ (1-e X_2)^2 \big[(e^2-44 e+166) X_2^2+2 (22-e) X_2+1 \big]
\\[1.0ex] 
&\hspace*{1.20in} +\, X_2^3 \big[44 (1-e X_2)+X_2 \big] \big\}\, g_{c1}. 
\end{array} 
$$

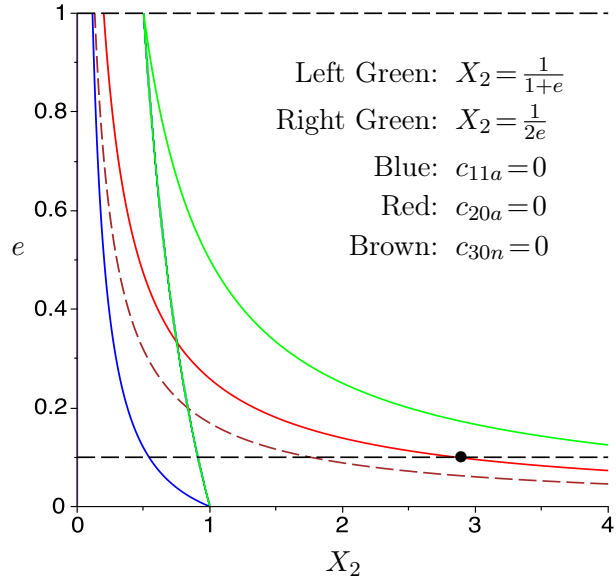
\begin{figure}[!h] 
\vspace*{0.00in}
\begin{center}
\hspace*{-0.20in}
\begin{overpic}[width=0.5\textwidth,height=0.33\textheight]{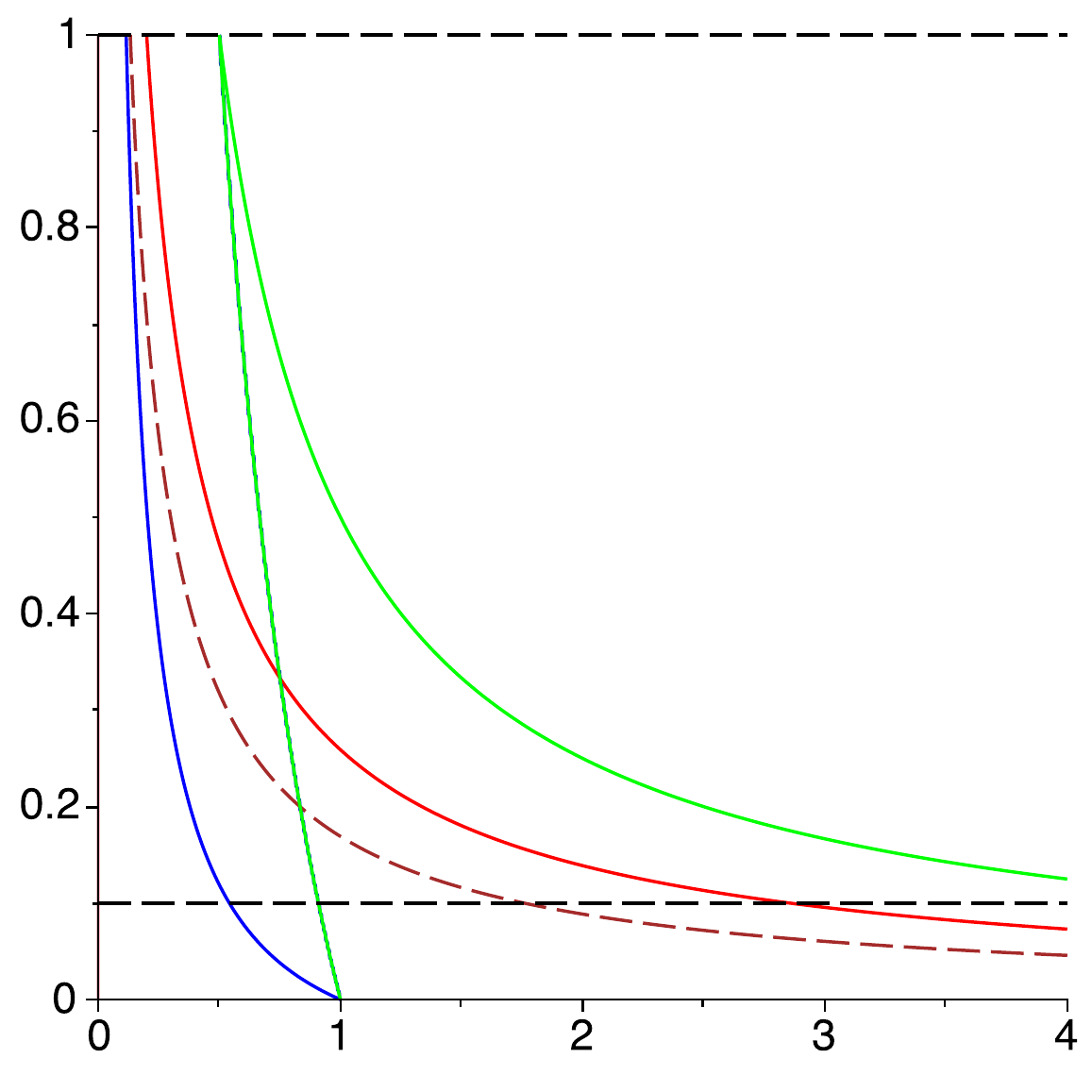}
\put(40,65){$\begin{array}{rl}
              {\color{black}\textrm{Left Green:}}&\!\! 
              {\color{black}X_2 \!=\! \tfrac{1}{1+e}}  \\[0.3ex] 
              {\color{black}\textrm{Right Green:}}&\!\! 
              {\color{black}X_2 \!=\! \tfrac{1}{2e}}   \\[0.3ex] 
              {\color{black}\textrm{Blue:}}&\!\! 
              {\color{black}c_{11a} \!=\! 0}           \\[-0.3ex] 
              {\color{black}\textrm{Red:}}&\!\! 
              {\color{black}c_{20a} \!=\! 0}           \\[-0.3ex] 
              {\color{black}\textrm{Brown:}}&\!\! 
              {\color{black}c_{30n} \!=\! 0}           \\[-0.3ex] 
             \end{array}
           $}

\put(50.5,-3){$X_2$}
\put(-2,50){$e$}
\put(72,14.9){$\bullet$} 
\end{overpic}

\vspace{0.10in}
\caption{Curve of $c_{11a} = 0$ (blue) and $c_{20a}=0$ (red) in 
the $X_2$-$e$ plane. The BT bifurcation region  \eqref{Eqn43} 
lies between the two green curves.
A special case $(X_2,e) = (\tfrac{20}{7}, \tfrac{1}{10})$ is 
marked by the black circle.}
\label{Fig7}
\end{center}
\end{figure}

Using a graphical approach, we plot $c_{20a}=0$ and $c_{11a}=0$ in the 
$(X_2,e)$-plane for $0<e<1$ and $\frac{1}{1+e} < X_2 < \frac{1}{2e}$ 
(Figure~\ref{Fig7}). The BT bifurcation region given by~\eqref{Eqn43} 
lies between the two green curves. In this region, $c_{11} \neq 0$ 
since the blue curve $c_{11a} = 0$ lies entirely outside. 
However, a portion of the red curve $c_{20a} = 0$ lies inside.
Therefore, 
\begin{itemize}
\item 
If $c_{20a}\ne 0 $ (and hence $c_{20} \ne 0$), the BT bifurcation has 
codimension two. 

\item 
If $c_{20a}=0$ (and hence $c_{20} \ne 0$), the BT bifurcation has 
codimension three. 
\end{itemize}
For the codimension-three case, we verify that $c_{30} \neq 0$ 
when $c_{20a} = 0$. Since $c_{30d}<0 $ under the condition~\eqref{Eqn43}, 
it remains to check that $c_{30n} \neq 0$. As shown in Figure~\ref{Fig7}, 
the brown curve $c_{30n} = 0$ lies below the red curve $c_{20a} = 0$, 
implying $c_{30n}>0$ when $c_{20a} = 0$. Consequently, 
$c_{30}<0$ when $c_{20}=0$, which confirms that $c_{30} \neq 0$ 
in this case. Therefore, system~\eqref{Eqn8} exhibits a parabolic 
codimension-three BT bifurcation.

Finally, note that $c_{20a}=0$ can hold only if $K_1 K_2 >0$. 
We can rewrite $c_{20a}=0$ as 
$$
c_{20a} = \dfrac{-X_2 [1-e (1+e) X_2 ] [(1+e) X_2-1]^4 K_3}
{2 X_2 K_1 + g_{c1} K_2}, 
$$
where
\begin{equation}\label{Eqn47}
K_3 = e^2 (e+9) X_2^2-e (e+6) X_2+1.
\end{equation}
Thus, $c_{20a}=0$ is equivalent to $K_3=0$, 
which yields the solution $X_{2c}$ given in \eqref{Eqn44}. 
For example, when $e=\frac{1}{10}$, we obtain $X_{2c}=\frac{20}{7}$ 
and $Y_{2c}=\frac{25}{7}$, as indicated by the black circle in 
Figure~\ref{Fig7}. 
\end{proof}

\subsubsection{The PSNF for the codimension-three BT bifurcation}

In this section, we focus on the codimension-three BT bifurcation, 
since the analysis on the codimension-two BT bifurcation
is similar to that discussed in Section 2.2.2. 
We simply present the PSNF for the codimension-two BT bifurcation, 
and the turn to codimension-three BT bifurcation. 
The PSNF for codimension-two BT bifurcation at the 
equilibrium ${\rm E_{bt}}$ under the perturbation: 
$a=a_c + \mu_1$ and $d=d_c + \mu_2$, 
is obtained up to second order: 
$$
\dfrac{d y_1}{d \tau} = y_2, \qquad 
\dfrac{d y_2}{d \tau} = \beta_1 + \beta_2\, y_2 + y_1^2 + C_{11}\, 
y_1 y_2, 
$$
where 
$$
\begin{array}{rl} 
C_{11} = \!\!\!\! & 
\dfrac{C_{11a}}{X_2 (1-e X_2) (1-2 e X_2) [1-e (1+e) X_2] K_3}, 
\\[2.5ex]
C_{11a} = \!\!\!\! & (1-e X_2)
          \big[\, 6 e^3 (e+1)^2 X_2^4-4 e^2 (3 e^2+9 e+2) X_2^3
           +e (7 e^2+30 e+5) X_2^2 
\\[1.0ex] 
& -\, (e^2 \!+\! 10 e \!+\! 1) X_2 \!+\! 1 \, \big] 
+2 e X_2 \big[\,e^3 (e \!+\! 1)^2 X_2^3 \!-\!2 e^2 (e^2 \!+\! e \!+\! 4) X_2^2 
\!+\! e (e^2 \!+\! 8) X_2 \!-\! 2 \,\big] g_{c1}, 
\end{array} 
$$ 
where $K_3 \ne 0$ (i.e., $c_{20}\ne 0$). 

We now turn to consider the codimension-three BT bifurcation. 
Since we have shown in the previous section that in general 
$c_{30} \ne 0$ when $c_{20}=0$, without loss of generality, 
we can choose a special case to reduce the computational complexity. 
To achieve this, let $e=\frac{1}{10}$, yielding 
$K_3 = 0$ (i.e., $c_{20}=0$, and 
\begin{equation}\label{Eqn48}
a_c = \dfrac{7}{5}, \quad d_c = \dfrac{2}{25}, \quad g_c = \dfrac{2}{7},
\quad X_{2c} = \dfrac{20}{7}, \quad Y_{2c} = \dfrac{25}{7}.  
\end{equation}
Further, introduction the perturbations:
\begin{equation}\label{Eqn49} 
a= \dfrac{7}{5} + \mu_1, \quad d=  \dfrac{2}{25} + \mu_2, \quad 
g= \dfrac{2}{7} + \mu_3, 
\end{equation}  
and the affine transformation,
$$
\left(\begin{array}{c} X \\ Y \end{array} \right) 
= \left(\begin{array}{c} \dfrac{20}{7} \\[2.0ex] 
 \dfrac{25}{7} \end{array} \right) 
\left[
\begin{array}{cc}
\dfrac{2}{7} & -\,\dfrac{4}{7} \\[2.0ex] 
\dfrac{1}{7} & -\,\dfrac{2}{7} 
\end{array} 
\right] 
\left(\begin{array}{c} u \\ v \end{array} \right)  
$$ 
into system \eqref{Eqn8} and expanding the resulting system up to 
third-order terms, we obtain 
\begin{equation}\label{Eqn50} 
\begin{array}{rl}
\dfrac{d u}{d \tau} = \!\!\!\! & v 
-\tfrac{100}{49}\, \mu_1 
+\tfrac{400}{343}\, \mu_1^2 
-\tfrac{1600}{2401}\, \mu_1^3 
+\tfrac{8}{175}\, u^2 
-\tfrac{1}{25}\, uv 
-\tfrac{8}{4375}\, u^3 
-\tfrac{4}{625}\, u^2 v  
\\[1.0ex]
& +\, \big(\tfrac{16}{49}\, u -\tfrac{4}{7}\, v 
+\tfrac{12}{1225}\, u^2 +\tfrac{8}{175}\, uv \big) \mu_1 
-\big(\tfrac{80}{343}\,u -\tfrac{16}{49}\,v \big) \mu_1^2  
\\[1.0ex] 
\dfrac{d v}{d \tau} = \!\!\!\! & 
-\tfrac{600}{343}\, \mu_1 
-\tfrac{625}{49}\,\mu_2 
-\tfrac{25}{7}\, \mu_3 
+\tfrac{2400}{2401}\, \mu_1^2 
-\tfrac{9600}{16807}\, \mu_1^3 
+\tfrac{2}{175}\, uv 
-\tfrac{2}{25}\, v^2 
\\[1.0ex]  
&+\, \big(\tfrac{96}{343}\,\mu_1 +\tfrac{100}{49}\,\mu_2 
+\tfrac{2}{7}\,\mu_3 \big) u 
-\big(\tfrac{24}{49}\, \mu_1 +\tfrac{50}{7}\, \mu_2 + \mu_3 \big) v 
-\tfrac{48}{30625}\, u^3 
-\tfrac{24}{4375}\, u^2 v
\\[1.0ex]
& +\, \big( \tfrac{72}{8575}\, \mu_1 -\tfrac{4}{49}\, \mu_2 \big) u^2 
-\mu_2 v^2 
+\big( \tfrac{48}{1225}\, \mu_1 +\tfrac{4}{7}\, \mu_2 \big) uv 
-\big(\tfrac{480}{2401}\, u -\tfrac{96}{343}\,v \big) \mu_1^2.   
\end{array} 
\end{equation}
Next, applying the third-order nonlinear transformation,
\begin{equation}\label{Eqn51} 
\begin{array}{rl}
u = \!\!\!\! & 
\tfrac{175}{8} y_1- \tfrac{4859575}{69984} \beta_1 
+\tfrac{2975}{108} \beta_2- \tfrac{25}{18} \beta_3
+\big(\tfrac{43157975}{629856} \beta_1- \tfrac{368725}{3888} \beta_2
-\tfrac{28175}{3888} \beta_3 \big) y_1
\\[1.0ex] 
&-\,\big(\tfrac{28175}{1728}-\tfrac{477893419325}{2176782336} \beta_1
+\tfrac{371683375}{3359232} \beta_2-\tfrac{81733225}{839808}\beta_3 \big)y_1^2
-\tfrac{1020425}{279936} y_1^3, 
\\[1.0ex] 
v = \!\!\!\! & 
\tfrac{175}{8} y_2
-\tfrac{35875}{2916} \beta_1+\tfrac{25}{9} \beta_2+\tfrac{100}{21} \beta_3
+\big(\tfrac{11979275}{69984} \beta_1-\tfrac{13475}{216} \beta_2
-\tfrac{175}{18} \beta_3 +\tfrac{117137134831675}{352638738432} \beta_1^2
\\[1.0ex]
&+\, \tfrac{256229575}{839808} \beta_2^2 -\tfrac{179725}{8748} \beta_3^2
-\tfrac{166345661825}{272097792} \beta_1 \beta_2 
+\tfrac{3638346775}{272097792} \beta_1 \beta_3
-\tfrac{22805825}{419904} \beta_2 \beta_3 \big) y_1
\\[1.0ex]
&-\,\big(\tfrac{340590425}{5038848} \beta_1+\tfrac{417725}{7776} \beta_2
-\tfrac{20125}{972} \beta_3 +\tfrac{44247435532825}{88159684608} \beta_1^2
+\tfrac{105515375}{419904} \beta_2^2-\tfrac{250775}{4374} \beta_3^2
\\[1.0ex]
&-\,\tfrac{42709048025}{272097792} \beta_1 \beta_2
-\tfrac{30202633475}{68024448} \beta_1 \beta_3
+\tfrac{8696275}{104976} \beta_2 \beta_3 \big) y_2
-\big(\tfrac{175}{8}+\tfrac{2576367325}{15116544} \beta_1
-\tfrac{10026625}{46656} \beta_2 
\\[1.0ex]
&+\, \tfrac{19075}{1296} \beta_3 \big) y_1^2
-\big(\tfrac{1225}{32}+\tfrac{6186356575}{20155392} \beta_1
+\tfrac{5496575}{31104} \beta_2 -\tfrac{412825}{3888} \beta_3 \big) y_1 y_2
-\tfrac{6950375243525}{14693280768} \beta_1^2 
\\[1.0ex]
&-\, \tfrac{2710925}{34992} \beta_2^2
+\tfrac{2800}{243} \beta_3^2+\tfrac{2905235725}{11337408} \beta_1 \beta_2
+\tfrac{280158725}{3779136} \beta_1 \beta_3-\tfrac{6475}{5832} \beta_2 \beta_3
\\[1.0ex] 
&-\,\tfrac{184631246091766075}{128536820158464} \beta_1^3
+\tfrac{191402575}{944784} \beta_2^3
+\tfrac{18850}{729} \beta_3^3+\tfrac{42529644164525}{22039921152} 
\beta_1^2 \beta_2 
\\[1.0ex]
&-\,\tfrac{377719492833025}{264479053824} \beta_1^2 \beta_3
-\tfrac{123070518025}{204073344} \beta_1 \beta_2^2
-\tfrac{160775125}{629856} \beta_2^2 \beta_3
+\tfrac{18429097225}{68024448} \beta_1 \beta_3^2
\\[1.0ex]
&-\,\tfrac{649775}{104976} \beta_2 \beta_3^2
+\tfrac{105627716075}{204073344} \beta_1 \beta_2 \beta_3
+\tfrac{28175}{864} y_1^3 + \tfrac{625975}{13824} y_1^2 y2,
\end{array} 
\end{equation}
the {\it forward} parametrization,
\begin{equation}\label{Eqn52} 
\begin{array}{rl}
\mu_1 = \!\!\!\! & 
-\tfrac{70315}{11664} \beta_1+\tfrac{49}{36} \beta_2+\tfrac{7}{3} \beta_3
-\tfrac{4322600582237}{58773123072} \beta_1^2
-\tfrac{2304617}{139968} \beta_2^2
+\tfrac{1288}{243} \beta_3^2+\tfrac{390131287}{45349632} \beta_1 \beta_2
\\[1.0ex]
&+\,\tfrac{335704733}{15116544} \beta_1 \beta_3
+\tfrac{121961}{23328} \beta_2 \beta_3
-\tfrac{2287218747969832943}{4113178245070848} \beta_1^3
+\tfrac{927647959}{15116544} \beta_2^3+\tfrac{8645}{729} \beta_3^3
\\[1.0ex]
&+\, \tfrac{259714225331659}{235092492288} \beta_1^2 \beta_2
-\tfrac{752375118606475}{1057916215296} \beta_1^2 \beta_3
-\tfrac{935045965819}{3265173504} \beta_1 \beta_2^2
-\tfrac{197492197}{2519424} \beta_2^2 \beta_3
\\[1.0ex] 
&+\, \tfrac{12502567805}{136048896} \beta_1 \beta_3^2
+\tfrac{2267867}{209952} \beta_2 \beta_3^2
+\tfrac{30773216033}{204073344} \beta_1 \beta_2 \beta_3,
\\[1.0ex] 
\mu_2 = \!\!\!\! & 
-\tfrac{1715}{11664} \beta_1+\tfrac{343}{450} \beta_2-\tfrac{28}{75} \beta_3
-\tfrac{415856256473}{73466403840} \beta_1^2+\tfrac{1248863}{874800} \beta_2^2
+\tfrac{1904}{6075} \beta_3^2-\tfrac{1198523977}{283435200} \beta_1 \beta_2
\\[1.0ex]
&-\, \tfrac{154750967}{94478400} \beta_1 \beta_3
-\tfrac{170471}{145800} \beta_2 \beta_3,
\\[1.0ex] 
\mu_3 = \!\!\!\! & 
-\tfrac{30877}{11664} \beta_1-\tfrac{61}{18} \beta_2+\tfrac{4}{21} \beta_3
+\tfrac{1691428227433}{14693280768} \beta_1^2+\tfrac{311689}{34992} \beta_2^2
-\tfrac{224}{243} \beta_3^2-\tfrac{48424397}{11337408} \beta_1 \beta_2
\\[1.0ex]
&-\,\tfrac{40355665}{3779136} \beta_1 \beta_3
-\tfrac{1057}{5832} \beta_2 \beta_3,
\end{array} 
\end{equation}
and the time rescaling 
$$ 
\tau = \Big(1 - \dfrac{245}{216}\, y_1 \Big)\, \tau_1, 
$$ 
to system \eqref{Eqn51}, we obtain the PSNF up to third-order terms:
\begin{equation}\label{Eqn53} 
\begin{array}{ll} 
\dfrac{d y_1}{d \tau_1} = y_2, \\[2.0ex]
\dfrac{d y_2}{d \tau_1} = \beta_1+\beta_2\,y_1+\beta_3\, y_2
+ \dfrac{9}{4}\, y_1 y_2 - y_1^3.
\end{array}
\end{equation} 
Moreover, a simple calculation yields
$$ 
\det\left[
\dfrac{\partial(\mu_1,\mu_2,\mu_3)}
{\partial(\beta_1,\beta_2,\beta_3)}
\right]_{\beta_1=\beta_2=\beta_3=0} = \dfrac{16807}{1200},
$$
which implies that system \eqref{Eqn53} with $(\beta_1,\beta_2,\beta_3) 
\sim (0,0,0)$ for $(y_1,y_2)$ near $(0,0)$ is topologically equivalent 
to system \eqref{Eqn8} with $(a,d,g)$ close to 
$(\frac{7}{5},\frac{2}{25},\frac{2}{7})$ for 
$(X,Y)$ near $(\frac{20}{7},\frac{25}{7})$.

The system \eqref{Eqn53} is already in the standard form \eqref{Eqn6}.
As noted earlier, the coefficient of $y_1 y_2$, namely $\frac{9}{4}$, 
cannot be normalized to $1$ since the coefficient of $y_1^3$ has been 
normalized to $-1$. 
To the best of our knowledge, a complete bifurcation analysis of 
system~\eqref{Eqn53} is not yet available, although related studies 
can be found in~\cite{DL2001,DL2003a,DL2003b} and in Chapter~3 
of~\cite{LWZH1997}. A detailed analysis of system~\eqref{Eqn53} lies 
beyond the scope of this paper and will be addressed in a forthcoming study.









\subsection{Hopf bifurcation}

In this subsection, we investigate the Hopf bifurcation in the second 
model \eqref{Eqn8} and determine the codimension, which specifies 
the maximal number of small-amplitude limit cycles 
that can bifurcate from the Hopf critical point. 
Although the overall procedure parallels that
in Section~2.3, the analysis here is significantly more difficult.

\begin{theorem}\label{Thm9}
For system \eqref{Eqn8}, the Hopf bifurcation is of codimension two. 
Consequently, at most two small-amplitude limit cycles can bifurcate 
from a Hopf critical point: an outer stable cycle
and an inner unstable cycle, both enclosing the stable equilibrium 
${\rm E_2}$.  
\end{theorem}

\begin{proof}
As stated in Theorem \ref{Thm6}, Hopf bifurcation occurs from 
the equilibrium ${\rm E_2}$ at the critical point 
\begin{equation}\label{Eqn54}
d_{\rm H} = \dfrac{X_2[ a(1 - 2e X_2) - e]}{(1-e X_2)(1+a X_2)^2}, \\[1.0ex]
\end{equation}
where the subscript 'H' denotes the Hopf critical point, together with 
\begin{equation}\label{Eqn55}
0<X_2<\dfrac{1}{2e}, \quad 
\dfrac{e}{1-2e X_2} < a < \min\left\{\dfrac{1+e}{1-2e X_2},\,
\dfrac{e X_2 + g_{c1}}{X_2 (1-2 e X_2)} \right\}. 
\end{equation}  
Multiplying the right-hand side of system \eqref{Eqn8} by $1+a X$,
corresponds to the time rescaling $\tau = (1 + a X_2) \tau_1$. 
At the Hopf critical point, the trace of the Jacobian vanishes,
and its determinant yields a purely imaginary pair of eigenvalues 
$\lambda_{1,2} = \pm i \omega_c$, where
$$ 
\omega_c = \sqrt{\det(J({\rm E_2}))} 
= \sqrt{X_2 \{1- e X_2 -X_2 [a (1-2 e X_2) -e]^2 \}}. 
$$ 
The constraints on $a$  in~\eqref{Eqn55} ensure that $\omega > 0$.

To compute the focus values for the Hopf bifurcation, we apply 
the affine transformation 
$$ 
\left(\begin{array}{c} X \\ Y \end{array} \right)
= \left(\begin{array}{c} X_2 \\ (1-e X_2)(1+a X_2) \end{array} \right)
+ \left[
\begin{array}{cc}
1 & 0 \\[1.0ex] 
\dfrac{e X_2 + g_{c1}}{X_2(1-2 e X_2)} & \dfrac{\omega}{X_2} 
\end{array} 
\right]
\left(\begin{array}{c} x_1 \\ x_2 \end{array} \right)
$$ 
into \eqref{Eqn8}, which yields 
\begin{equation}\label{Eqn56} 
\!\!
\begin{array}{rl} 
\dfrac{d x_1}{ d \tau} = \!\!\!\! & 
x_2 - \dfrac{a e}{\omega_c} (X_2 x_1^2 + x_1^3)
    + \dfrac{1}{X_2}\, x_1 x_2, \\[2.5ex]
\dfrac{d x_2}{ d \tau} = \!\!\!\! & 
- x_1 - \dfrac{[a (1-2 e X_2)-e]
          [\omega_c^2 + a X_2^2 (1-e X_2) (3 a e X_2-a+2 e)]}
           {\omega_c^2 (1-e X_2) (1 + a X_2)}\, x_1^2 
\\[2.5ex]
& +\, \dfrac{2 \omega_c^2 - X_2(1-e X_2) (2 a e X_2-a+e+1)}
           {\omega_c X_2 (1-e X_2) (1 + a X_2)}\, x_1 x_2 
+\dfrac{a (1-2 e X_21)-e}{(1-e X_2) (1 + a X_2)}\, x_2^2 
\\[2.5ex]
&-\, \dfrac{ a [a (1-2 e X_2 )-e]
         \{ \omega_c^2 +X_2(1-e X_2) [ e (1+a X_2)^2 -1] \} }
          {\omega_c^2 (1-e X_2) (1 + a X_2)^2}\, x_1^3
\\[2.5ex]  
&-\,\dfrac{2a X_2[a(1-2 e X_2)-e]^2}{\omega_c(1-e X_2)(1+a X_2)^2}\,x_1^2 x_2 
+\dfrac{a [a (1-2e X_2)-e]}{(1- e X_2)(1 + a X_21)^2}\, x_1 x_2^2. 
\end{array} 
\end{equation}  

Using the Maple program from \cite{Yu1998} for 
computing normal forms of Hopf and generalized Hopf bifurcations, we 
obtain the focus values:
$$ 
\begin{array}{ll} 
v_1 = -\,\dfrac{a X_2}{8 \omega_c^2 (1 + a X_2)^2}\ v_{1a}, 
\\[2.0ex] 
v_2 = -\,\dfrac{a X_2}{288 \omega_c^6 (1-e X_2) (1+a X_2)^4}\ v_{2a},
\\[2.0ex]
v_3 = \dfrac{a X_2}{663552 \omega_c^{10} (1-e X_2)^2 (1+a X_2)^6}\ v_{3a}, 
\end{array} 
$$ 
where 
\begin{equation}\label{Eqn57} 
\begin{array}{rl}
v_{1a} = \!\!\!\! & 
\dfrac{1}{X_2} \big\{ 2 (1+a X_2)^2 (1+a+a^2 X_2) 
-(1-e X_2) \big[a X_2^2 (1+2 a X_2) (3 + 6 a X_2 + 2 a^2 X_2) e^2
\\[1.0ex]
& \qquad +\, a X_2 (2+3X_2 +4 a X_2 +4 a X_2^2 +2 a^2 X_2^2 + 2 a^3 X_2^3) e
\\[1.0ex]
& \qquad +\, (2+2 a +4 a X_2 + 6 a^2 X_2 +a^2 X_2 + 7 a^3 X_2^2 +2 a^4 X_2^3) 
\big]\big\},
\\[0.5ex]
v_{2a} = \!\!\!\! & \cdots, 
\\[0.5ex]
v_{3a} = \!\!\!\! & \cdots, 
\end{array}
\end{equation}
Since $v_{1a}$, $v_{2a}$ and $v_{3a}$ depend on three free parameters: 
$a$, $e$ and $X_2$, the system could in principle have up to four 
small-amplitude limit cycles bifurcating from the Hopf critical point. 
To verify this, we eliminate $a$ from the equations $v_{1a} = 0$, $v_{2a}=0$
and $v_{2a} = 0$ to obtain a solution $a = a(e,X_2)$, along with 
two resultants: 
$$
R_1(e,X_2) = R_0(e,X_2) \, R_{1a}(e,X_2), \quad 
R_2(e,X_2) = R_0(e,X_2) \, R_{2a}(e,X_2), 
$$
where 
$$
\begin{array}{rl} 
R_0(e,X_2) = \!\!\!\! & K_3\, e X_2(1-e X_2) (1-e X_2+X_2)
\big[\,4 e^4 (e+1)^2 X_2^5-8 e^3 (e^2+6 e+1) X_2^4
\\[1.0ex]
& +\, 4 e^2 (e^2+20 e+2) X_2^3-4 e (13 e+1) X_2^2+(13 e+1) X_2-1 \,\big], 
\\[1.0ex] 
R_{1a}(e,X_2) = \!\!\!\! & 2 e^7 (69 e-64) (e+1)^2 X_2^9
-e^6 (543 e^3+574 e^2-43 e-416) X_2^8
\\[1.0ex]
& +\, 3 e^5 (279 e^3+260 e^2+141 e-134) X_2^7 
-e^4 (636 e^3+319 e^2+578 e-71) X_2^6 
\\[1.0ex]
&+\,2 e^3 (123 e^3-96 e^2+131 e+32) X_2^5 
-e^2 (45 e^3-204 e^2+17 e+42) X_2^4 
\\[1.0ex]
&+\,e (3 e^3-40 e^2-25 e+10) X_2^3 
-(17 e^2-12 e+1) X_2^2+2 (6 e-1) X_2-2, \\[1.0ex] 
R_{2a}(e,X_2) = \!\!\!\! & \cdots, 
\end{array} 
$$
with $K_3$ given in \eqref{Eqn47} and $R_{2a}$ being a 
$40$-degree polynomial in $e$. Note that the common factor 
$R_0$ may correspond to center conditions rather than the 
existence of limit cycles.
Thus, eliminating $e$ from $R_{1a}(e,X_2) = 0$ 
and $R_{2a}(e,X_2) = 0$ yields a solution $e = e(X_2)$ and a 
$145$-degree polynomial resultant equation in $X_2$:
$$ 
R_{12}(X_2) = 0. 
$$
Solving the equation numerically in Maple with an accuracy of up to
$1000$ decimal places yields $10$ positive real solutions for $X_2$.
Upon checking these solutions, we find that only six yield $e>0$; however,
all of them satisfy $X_2 > \frac{1}{2e}$, violating the constraint
\eqref{Eqn54}. Hence, there are no solutions satisfying
$v_1 = v_2 = v_3 = 0$ with $v_4 \ne 0$, implying that system~\eqref{Eqn8}
cannot have four small-amplitude limit cycles bifurcating from a Hopf
critical point.

Next, we investigate whether three small-amplitude limit cycles can
bifurcate from the Hopf critical point. That is, we seek parameters
$a$, $e$, and $X_2$ such that $v_{1a} = v_{2a} = 0$ but $v_{3a} \ne 0$.
This problem is more challenging, since the two equations involve
three free parameters, and we must determine the sign of $v_{3a}$ when
$v_{1a} = v_{2a} = 0$. In other words, we must solve
$R_{1a}(e,X_2) = 0$ with $a = a(e,X_2)$ under the constraints
\eqref{Eqn55}, which can be written as
\begin{equation}\label{Eqn58}
0<X_2<\dfrac{1}{2e}, \quad 
\dfrac{e}{1-2e X_2} < a(e,X_2) < \min\left\{\dfrac{1+e}{1-2e X_2},\,
\dfrac{e X_2 + g_{c1}}{X_2 (1-2 e X_2)} \right\}. 
\end{equation}    

Similarly, we apply a graphical approach by plotting 
in the $(X_2,e)$-plane ($0 < X_2 < \frac{1}{2e}$) the curves
$$
R_{1a}=0, \ \  a(e,X_2)-\frac{e}{1-2e X_2}=0, \ \ 
\frac{1+e}{1-2e X_2}-a(e,X_2)=0, \ \ \textrm{and} \ \ 
\frac{e X_2 + g_{c1}}{X_2 (1-2 e X_2)}-a(e,X_2)=0.
$$
The resulting plots are shown in Figure~\ref{Fig8}.
In these figures, the green curve $X_2 = \frac{1}{2e}$ defines the upper
boundary of the admissible region. The red curve $R_{1a} = 0$ has two
branches, and only its lower branch lies within the admissible region
(below the green curve). The blue curve
$a(e,X_2) - \frac{e}{1 - 2e X} = 0$ is used to test the first inequality
in \eqref{Eqn58}. Since this condition is already violated, the other
two inequalities in \eqref{Eqn58} need not be checked; nonetheless, we
find that they are also violated.

\begin{figure}[!h] 
\vspace*{0.00in}
\begin{center}
\hspace*{-0.20in}
\begin{overpic}[width=0.36\textwidth,height=0.23\textheight]{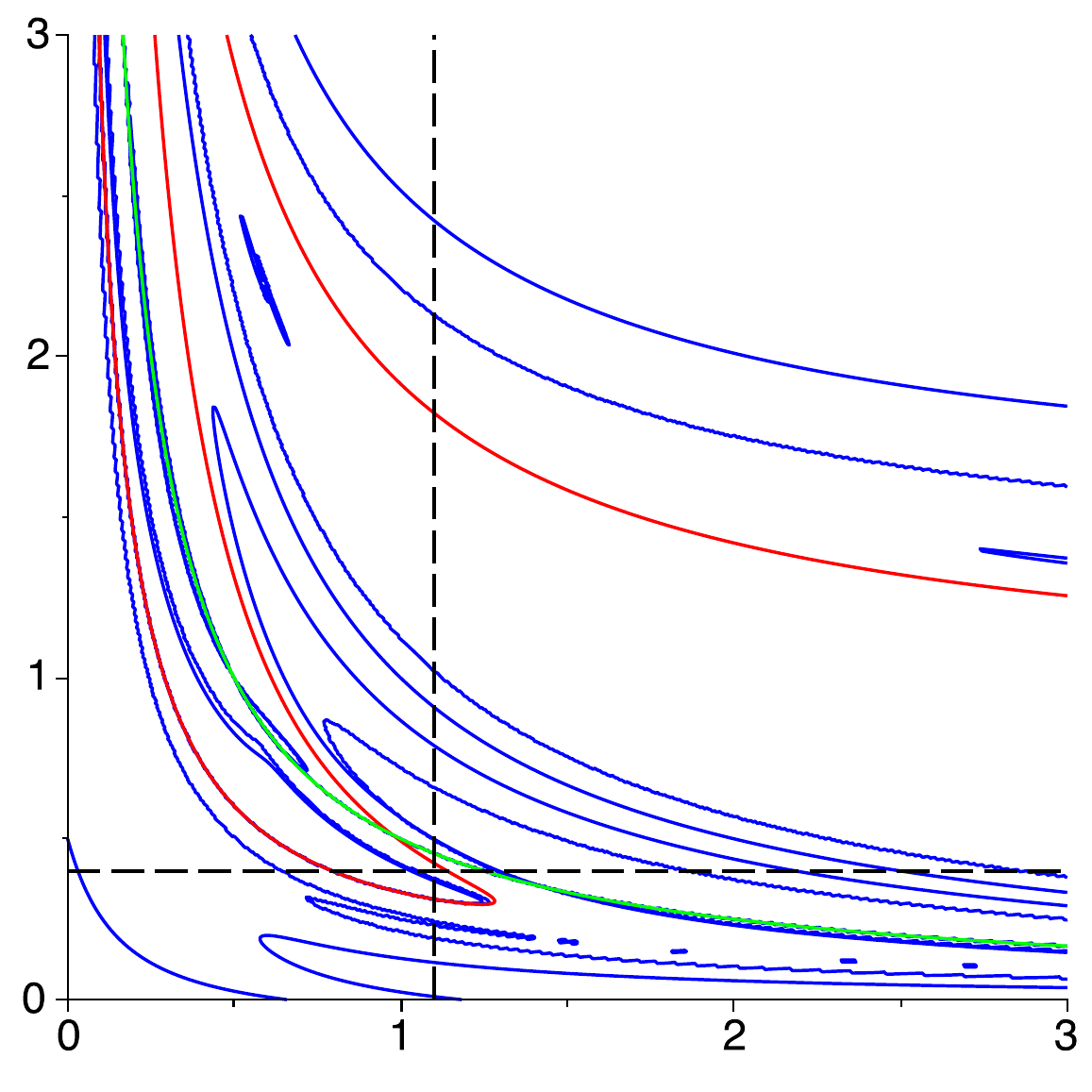}
\put(50.5,-3){\footnotesize{$X_2$}}
\put(-8,47){$e$}
\put(65,82){\color{green} \footnotesize{$X_2$} $\!=\! \frac{1}{2e}$}
\put(65,73){\color{red}\footnotesize{$R_{1a}$} $\!=\! 0$}
\put(65,66){\color{blue}$a \!=\! \frac{e}{1-2e X_2}$}
\end{overpic}\hspace*{0.4in} 
\begin{overpic}[width=0.36\textwidth,height=0.23\textheight]{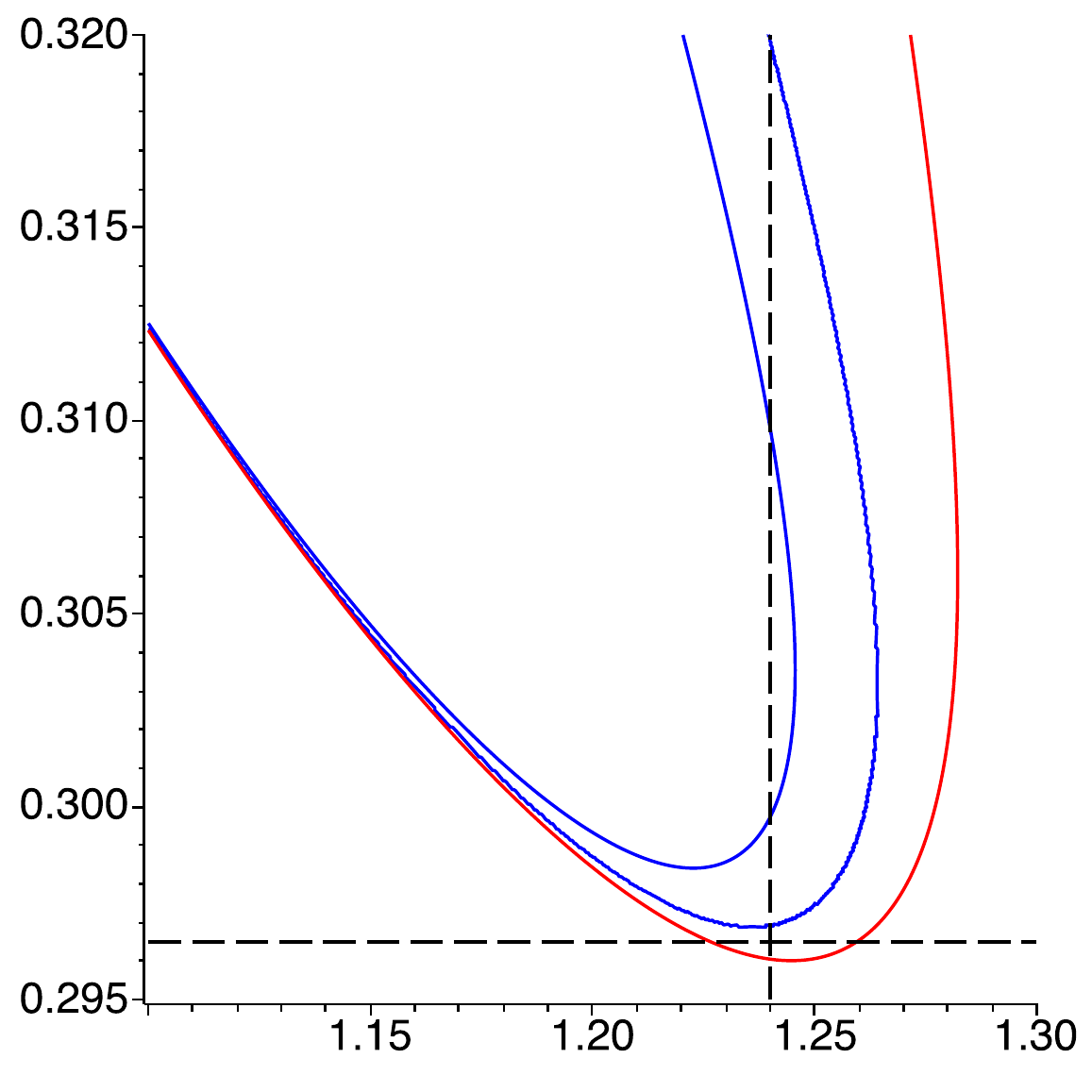}
\put(50.5,-3){\footnotesize{$X_2$}}
\put(-8,47){$e$}
\end{overpic}

\vspace*{0.10in} 
\hspace*{0.03in}(a)\hspace*{2.53in}(b) 

\vspace*{0.20in} 
\hspace*{-0.20in}
\begin{overpic}[width=0.36\textwidth,height=0.23\textheight]{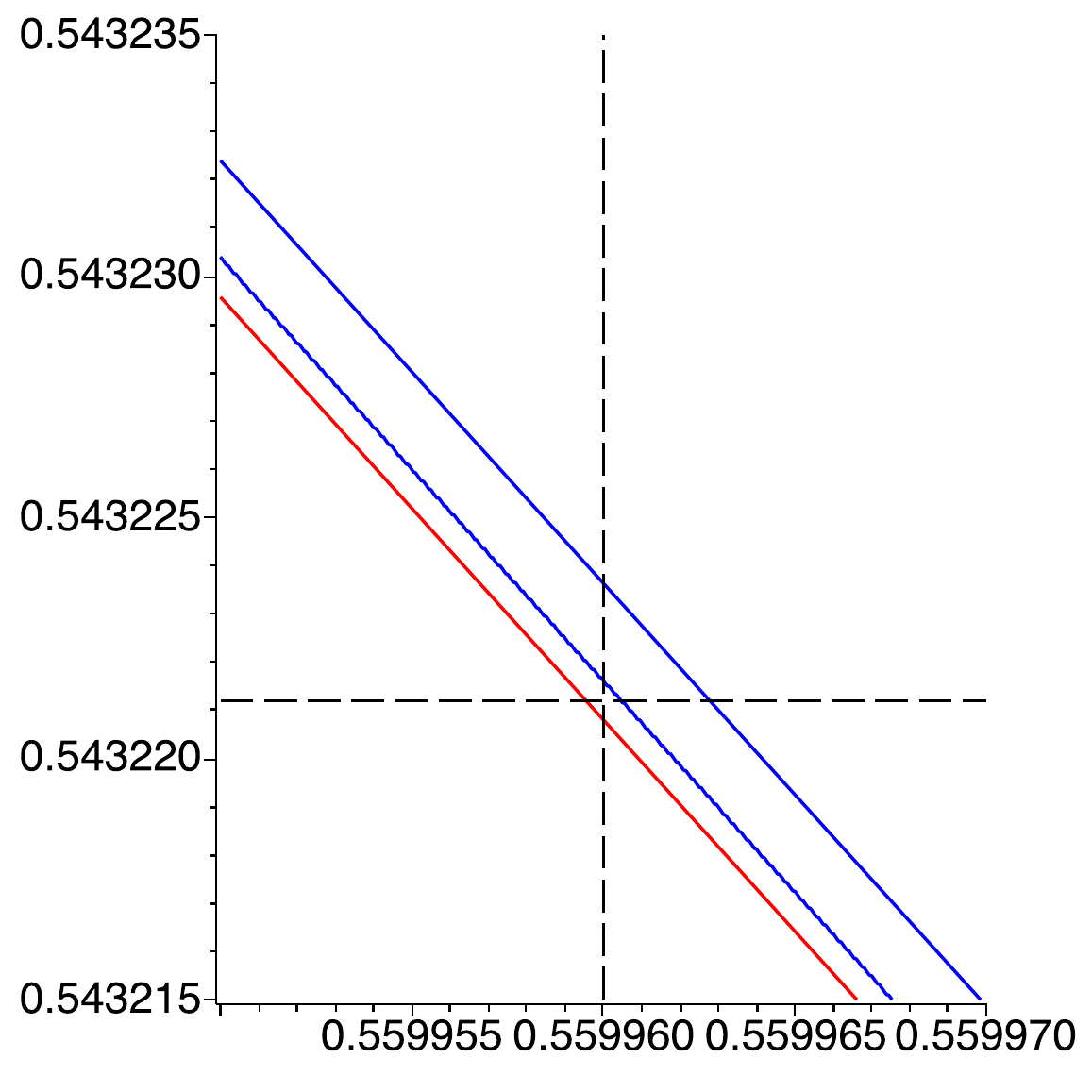}
\put(50.5,-3){\footnotesize{$X_2$}}
\put(-8,47){$e$}
\end{overpic}\hspace*{0.4in} 
\begin{overpic}[width=0.36\textwidth,height=0.23\textheight]{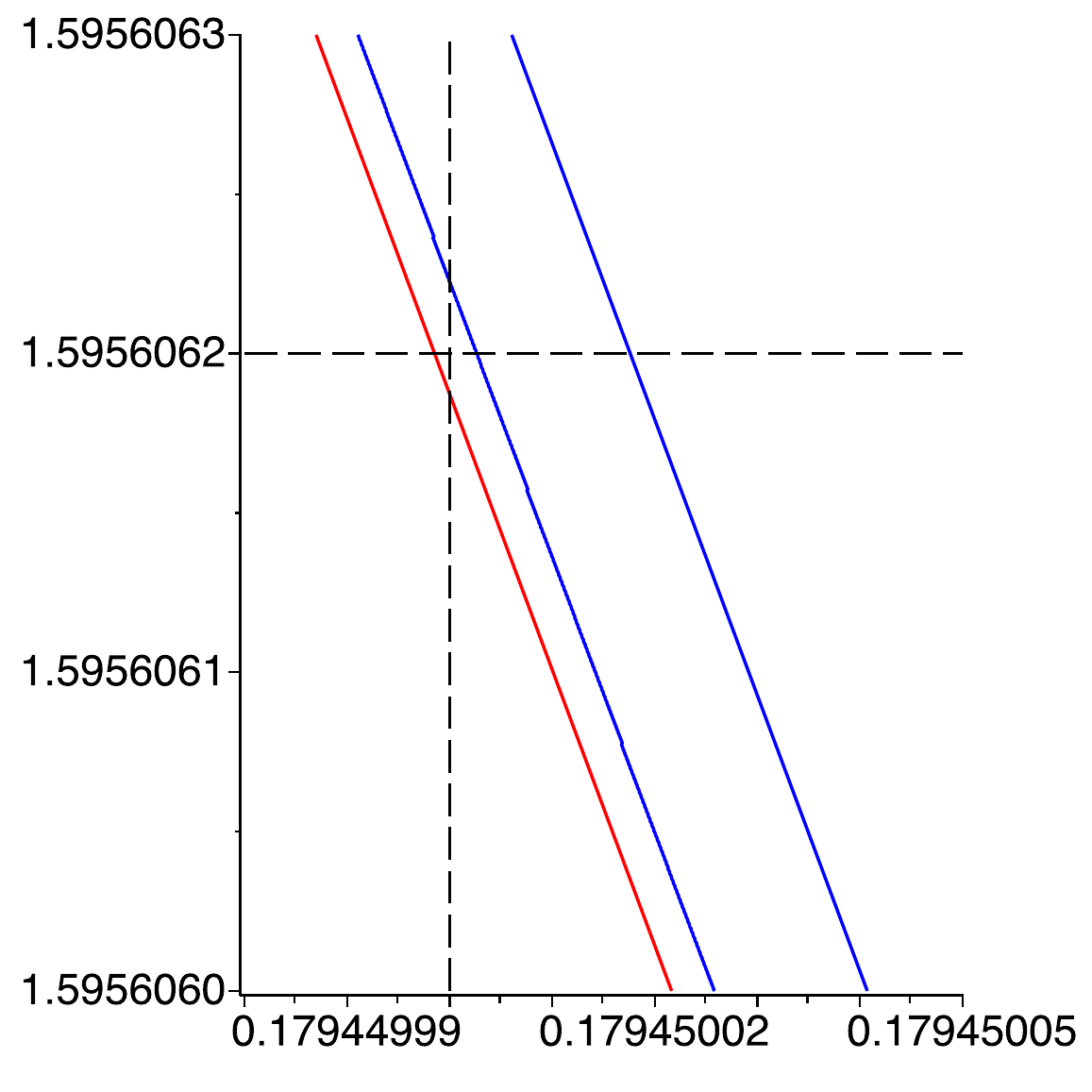}
\put(50.5,-3){\footnotesize{$X_2$}}
\put(-8,47){$e$}
\end{overpic}

\vspace*{0.10in} 
\hspace*{0.03in}(c)\hspace*{2.53in}(d) 

\caption{Curves of $X_2=\frac{1}{2e}$ (green), $R_{1a} = 0$ (red), and
$a(e,X_2)=\frac{e}{1-2e X_2}$ (blue) in the $X_2$–$e$ plane, showing that
the condition $a(e,X_2)>\frac{e}{1-2e X_2}$ does not hold when
$R_{1a}=0$. This demonstrates that three limit cycles cannot bifurcate
from the Hopf critical point.} 
\label{Fig8}
\end{center}
\vspace*{-0.10in} 
\end{figure}
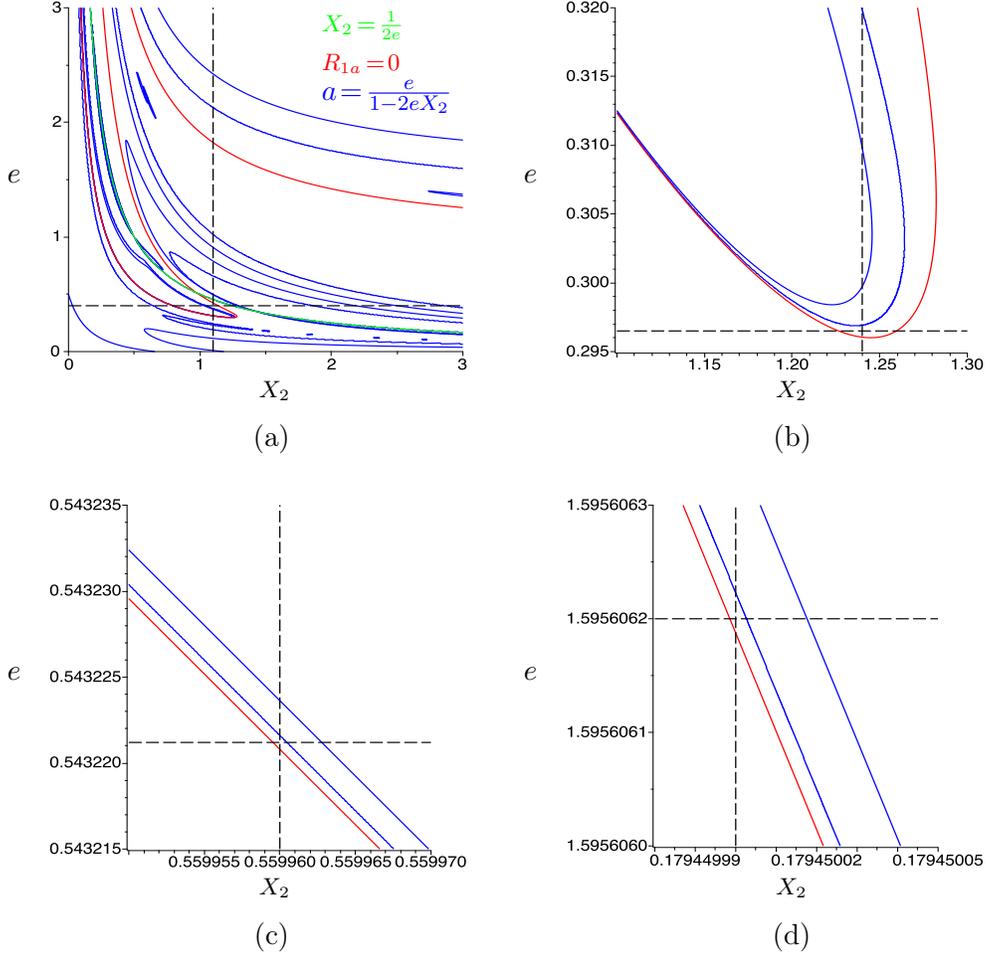

Figure~\ref{Fig8}(a) shows the global view, while
Figures~\ref{Fig8}(b)–(d) provide zoomed-in views along the lower red
curve. In Figure~\ref{Fig8}(a), it appears that the lower branch of the
red curve coincides with the blue curve, but the close-up plots reveal
that the red curve actually encloses the blue branch. At representative
points along this red curve (e.g., the intersection of the two black
dotted lines), one finds that
$a(e,X_2) - \frac{e}{1 - 2e X} < 0$. This confirms that no solutions
exist satisfying $R_{1a} = 0$ under the constraint \eqref{Eqn58}, and
hence three small-amplitude limit cycles cannot bifurcate from the Hopf
critical point.

Therefore, the Hopf bifurcation has codimension at most two,
implying that the maximal number of bifurcating small-amplitude limit
cycles is two. Since Theorem~\ref{Thm5} establishes the existence of a
codimension-three Bogdanov–Takens bifurcation, which can yield two
small-amplitude limit cycles, we conclude that the codimension of the
Hopf bifurcation in system~\eqref{Eqn8} is exactly two. In this case, it
suffices to solve $v_{1a}=0$ under the condition~\eqref{Eqn58}, which is
a polynomial equation in three free parameters. Consequently, $v_{1a}=0$ 
admits infinitely many solutions, each  generating two small-amplitude 
limit cycles at the equilibrium~${\rm E_2}$.

Moreover, from the explicit expression of $v_{1a}$, it is straightforward 
to verify that $v_{1a}=0$ has infinitely many positive solutions for $e$.  
As a concrete example, take $e=\frac{10}{9}$ and 
$$
a=\frac{1}{2}\Big( \frac{e}{1-2e X_2}+\frac{1+e}{1-2e X_2} 
\Big) = \frac{29}{2(9-20X_2)}.
$$
Solving 
$$ 
v_{1a} = \dfrac{385700 X_2^4-1821600 X_2^3+2639250 X_2^2-1705617 X_2+349920}
         {216 (9-20 X_2)^3} =0 
$$
yields $X_2 = 0.3587228155 \cdots \in \big(0,\frac{1}{2e}\big)
=(0,0.45)$, for which 
$$
v_2= -\,6.5738643163 \cdots <0, 
$$
and 
$$
d_{\rm H}= 0.0201276613 \cdots, \quad 
g= 0.0465961211 \cdots, \quad 
\omega_c = 0.4284533067 \cdots, 
$$
all satisfying the necessary positivity conditions.  

Although we have shown that $v_{2a} \ne 0$ when $v_{1a}=0$, 
it remains to prove that $v_{2a}<0$, and hence $v_2<0$, in the case 
$v_{1a}=0$ (i.e., $v_1=0$). Establishing this ensures the existence 
of a stable outer limit cycle and an unstable inner limit cycle. 
However, providing a purely algebraic proof is nearly impossible. 
Instead, we employ a graphical approach combined with computational 
searching. Figure~\ref{Fig9} illustrates the surfaces $v_{1a}=0$ (blue) 
and $v_{2a}=0$ (red) in the three-dimensional $(e,X_2,a)$ space. 
The figure shows that, within the region defined by \eqref{Eqn55}, 
the red surface $v_{2a}=0$ does not intersect the blue surface 
$v_{1a}=0$ but instead lies entirely on one side of it. This implies 
that $v_{2a}$ retains a fixed sign whenever $v_{1a}=0$. 
A simple test confirms that $v_{2a}<0$, and thus $v_2<0$ when $v_{1a}=0$.

\begin{figure}[!h] 
\vspace*{0.00in}
\begin{center}
\hspace*{-0.70in}
\begin{overpic}[width=0.52\textwidth,height=0.27\textheight]{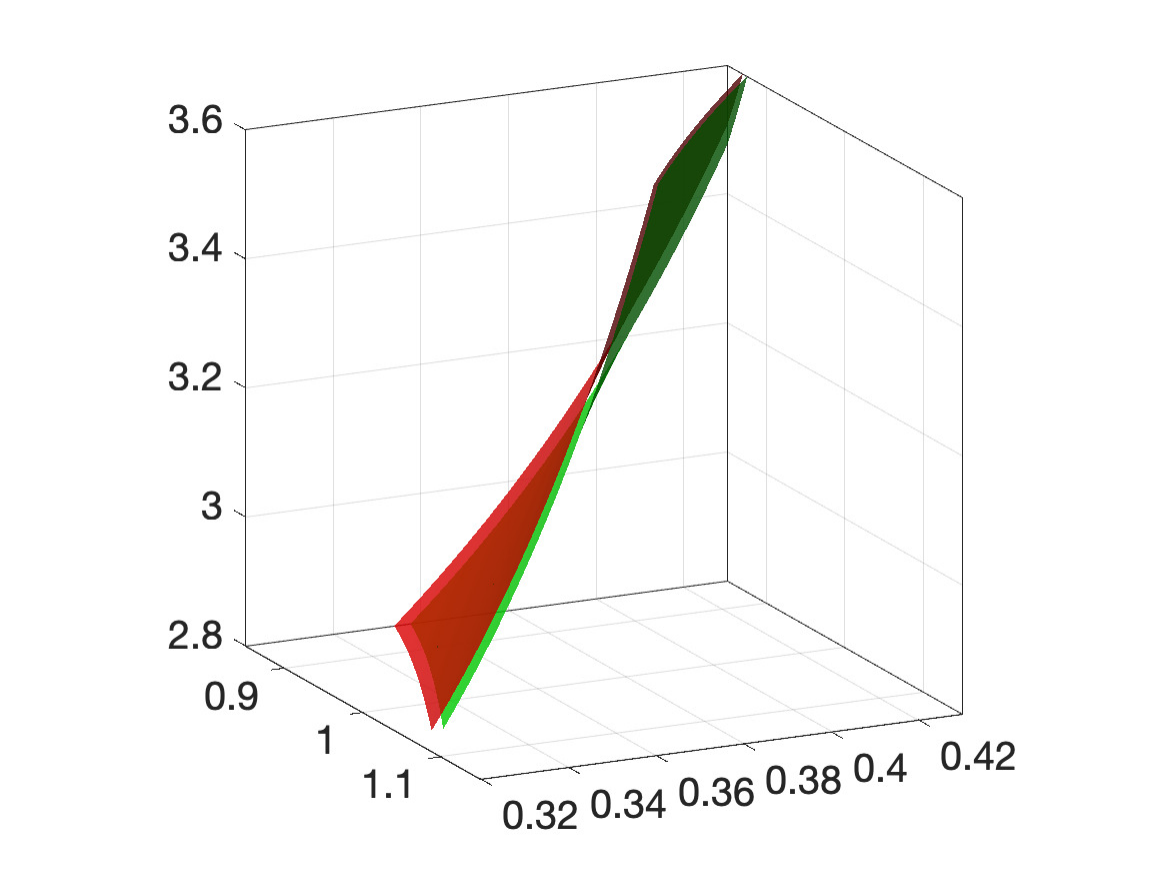}
\put(8,46){{$a$}} 
\put(20,5){\footnotesize{$X_2$}} 
\put(65,1){{$e$}} 
\end{overpic}
\hspace*{-0.20in}
\begin{overpic}[width=0.52\textwidth,height=0.27\textheight]{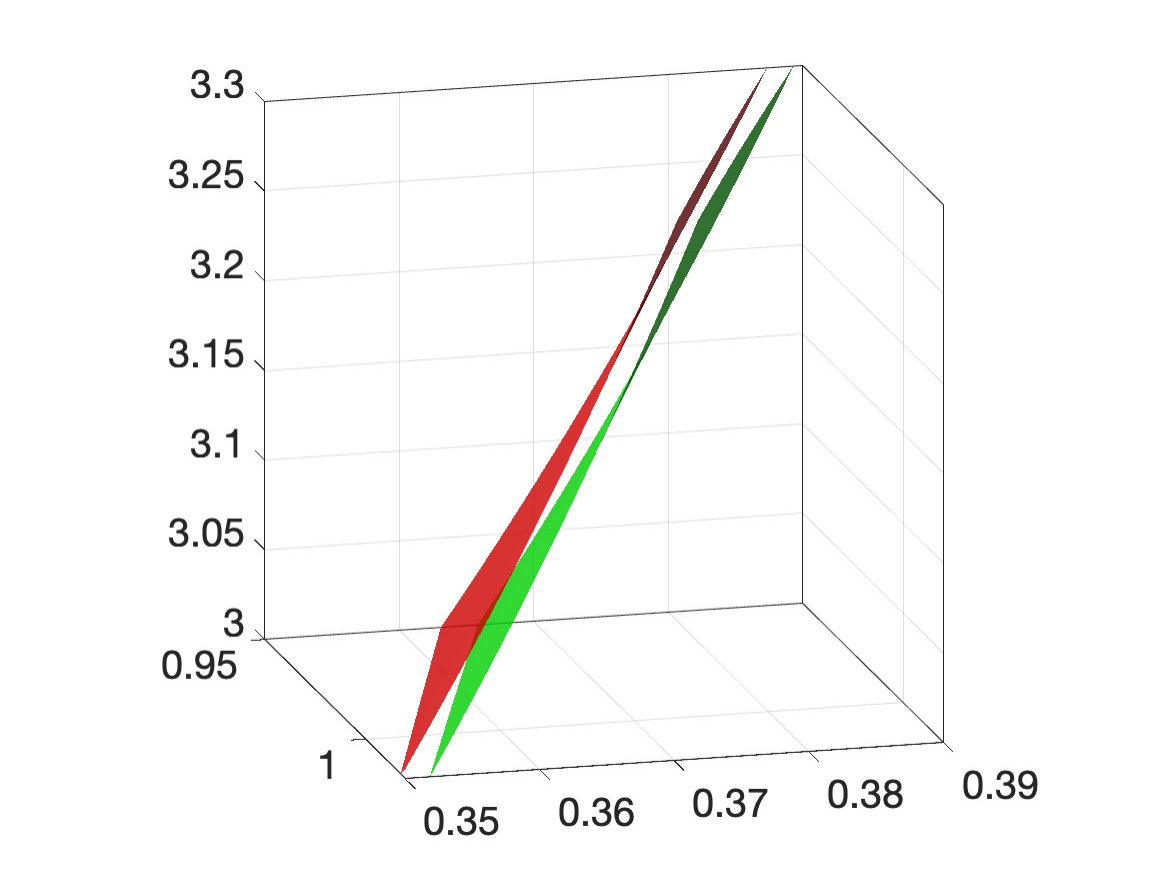}
\put(8,46){{$a$}} 
\put(20,5){\footnotesize{$X_2$}} 
\put(65,1){{$e$}} 
\end{overpic}\hspace*{-0.80in} 

\vspace*{0.00in} 
\hspace*{0.20in}(a)\hspace*{3.10in}(b) 

\caption{(a) Surfaces $v_{1a}=0$ (blue) and $v_{2a}=0$ (red) 
in $(e,X_2,a)$ space, showing that $v_{2a}=0$ does not intersect 
$v_{1a}=0$; (b) a enlarged view of the region highlighted in (a).}
\label{Fig9}
\end{center}
\vspace*{-0.10in} 
\end{figure}
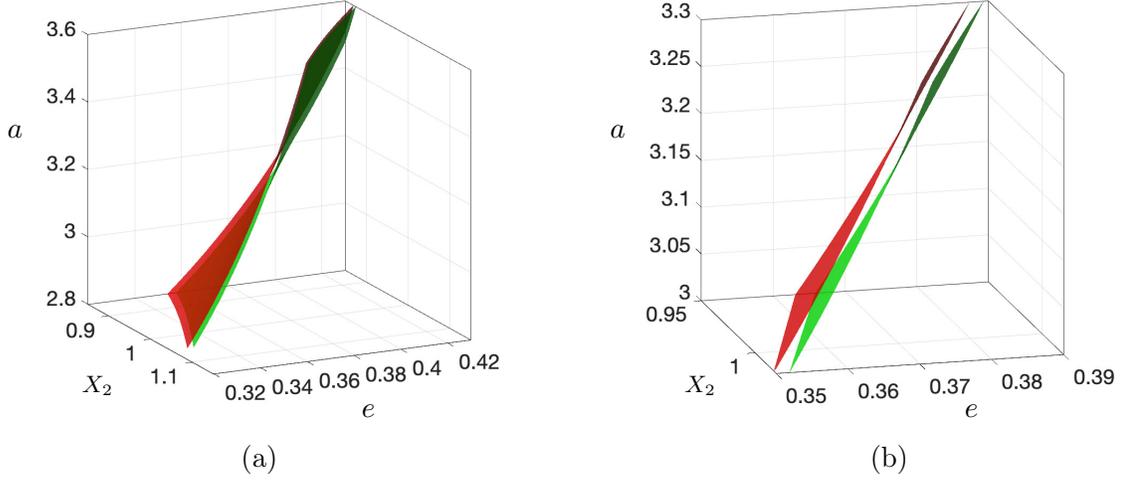

To further confirm this graphical result, we implemented a computational 
search for possible solutions with $v_{2a}>0$ under the condition 
$v_{1a}=0$. Note that the constraints in \eqref{Eqn55} can be 
divided into three categories:
\begin{enumerate}
\item[{(1)}] 
$0<e<1$, $0<X_2< \frac{1}{1+e}$, 
and $\frac{e}{1-2 e X_2} <a< \frac{e+1}{1-2e X_2}$; 
\item[{(2)}] 
$0<e<1$, $\frac{1}{1+e}<X_2<\frac{1}{2e}$, and $\frac{e}{1-2 e X_2} 
<a< \frac{e X_2+\sqrt{X_2 (1-e X_2)}}{X_2(1-2 e X_2)}$; 
\item[{(3)}] 
$e \ge 1$, $0<X_2< \frac{1}{2 e}$, and 
$\frac{e}{1-2eX_2}<a<\frac{e+1}{1-2e X_2}$. 
\end{enumerate}   

In categories (1) and (2), all three parameters $e$, $X_2$, and $a$ 
are finitely bounded. In contrast, in category (3), $e$ is unbounded, 
which complicates the definition of an upper boundary for $e$. 
Nevertheless, after applying Gr\"{o}bner basis reduction to $v_{2a}$, 
we find that the leading coefficient of $e$ in the resulting polynomial 
is negative. This implies that, in category (3), $v_{2a}<0$ for 
sufficiently large $e$. Consequently, finite computational boundaries 
can be imposed in all three cases.

For transparency, we list the expression for $v_{2a}$ obtained 
through Gr\"{o}bner basis reduction below. 
$$ 
\!\!
\begin{array}{rl}
v_{2a}=\!\!\!\!\! & -3 a X_2^5 (4986 a X_2 +3161) (2 a X_2+1) e^8 
       +3 a X_2^4 \big[ 2 a^2 X_2^2 (8614 X_2+55243)
\\[0.7ex]
& + a X_2 (14550 X_2+120199) +16 (102 X_2+2009)\big] e^7 
-\tfrac{3}{2} X_2^3 \big[ 4 a^3 X_2^2 (4278 X_2^2+81093 X_2 
\\[0.7ex] 
& +286303) +2 a^2 X_2 (2322 X_2^2+122518 X_2+594993) 
                  -a (5856 X_2^2-19102 X_2-302921)
\\[0.7ex] 
& +4204 \big] e^6 + \tfrac{3}{2} X_2^2 \big[ 2 a^3 X_2^2 
(1050 X_2^3+65574 X_2^2+705502 X_2+1847029)
\\[0.7ex]
& -2 a^2 X_2 (1564 X_2^3-4848 X_2^2-462565 X_2-1824251) 
-a (3514 X_2^3+52282 X_2^2-34261 X_2
\\[0.7ex]
& -879240) +4 (1603 X_2+10363) \big] e^5 
+\tfrac{3}{2} X_2 \big[ 2002162 a^4 X_2^3 
-2 a^3 X_2^2 (6041 X_2^3+228348 X_2^2 
\\[0.7ex]
& +1885222 X_2+659340) 
+ a^2 X_2 (288 X_2^4+25796 X_2^3+73572 X_2^2-2048969 X_2-1769020) 
\\[0.7ex]
& +a (216 X_2^4+21831 X_2^3+204597 X_2^2+31222 X_2-261725) 
-2 (1238 X_2^2+25972 X_2 
\\[0.7ex]
\end{array}
$$ 
$$ 
\!\!
\begin{array}{rl}
& +94065) \big] e^4 
+\tfrac{1}{4} \big[ 4178724 a^5 X_2^4 -12 a^4 X_2^3 (625651 X_2+680529) 
+6 a^3 X_2^2 (31061 X_2^3 
\\[0.7ex]
& +960102 X_2^2+2564445 X_2+249870) 
-6 a^2 X_2 (1380 X_2^4+91300 X_2^3+425444 X_2^2 
\\[0.7ex] 
& -1143225 X_2-380546) 
-3 a X_2 (1602 X_2^3+116529 X_2^2+932296 X_2-607790) +864 X_2^3 
\\[0.7ex]
& +93144 X_2^2+1144104 X_2+3140700 \big] e^3
-\tfrac{1}{4}\big[672 a^9 X_2^7 -72 a^8 X_2^6 (8 X_2+147)
\\[0.7ex]
& +132 a^7 X_2^5 (116 X_2+529) 
-12 a^6 X_2^4 (450 X_2^2+10804 X_2+29865) 
+6 a^5 X_2^3 (12506 X_2^2 
\\[0.7ex]
& +126974 X_2+616437) 
-6 a^4 X_2^2 (2175 X_2^3+93506 X_2^2+1245815 X_2+409266) 
\\[0.7ex]
& +6 a^3 X_2 (39863 X_2^3+903509 X_2^2+940962 X_2+15936) 
-6 a^2 X_2 (3011 X_2^3+194121 X_2^2 
\\[0.7ex]
& +244051 X_2-28258) -3 a (2517 X_2^3+179251 X_2^2-280126 X_2+571616) 
+4 (873 X_2^2 
\\[0.7ex]
& +64287 X_2+631698) \big] e^2 
+ \tfrac{1}{4} 
\big[384 a^9 X_2^6 -12 a^8 X_2^5 (28 X_2+611)+36 a^7 X_2^4 (225 X_2+1114) 
\\[0.7ex]
& -204 a^6 X_2^3 (19 X_2^2+490 X_2+1412)
+12 a^5 X_2^2 (4826 X_2^2+51612 X_2+111221) 
\\[0.7ex]
& -6 a^4 X_2 (1648 X_2^3+76234 X_2^2+521829 X_2+62866) 
+3 a^3 X_2 (43679 X_2^2+885834 X_2 
\\[0.7ex]
& +491182) -3 a^2 (3681 X_2^3+256434 X_2^2+366998 X_2-160076) 
+3 a (1033 X_2^2+49594 X_2 
\\[0.7ex]
& -262988) +6 (1033 X_2+69154) \big] e -\tfrac{3}{4} X_2 a^2 (a-1) 
(8 a^6 X_2^4 -316 a^5 X_2^3 
\\[0.7ex]
& +4 a^4 X_2^2 (59 X_2+699) 
-4 a^3 X_2 (809 X_2+4101) +2 a^2 (404 X_2^2+11830 X_2+40019) 
\\[0.7ex]
& -2 a (4890 X_2+65747)+1033 X_2+69154).
\end{array} 
$$ 
Using this expression together with $v_{1a}$ from \eqref{Eqn57}, 
our exhaustive computational search with a fine step size confirms 
that $v_{2a}<0$ across all three parameter categories.

This completes the proof of Theorem~\ref{Thm9}.
\end{proof} 

\begin{remark}\label{Rem3.6}
The analysis of system \eqref{Eqn8} demonstrates that even for a 
two-dimensional real-world problem, determining the codimension of 
BT and Hopf bifurcations can be highly intricate, especially when the 
number of system parameters exceeds the codimension of the bifurcation. 
In general, determining the codimension of a Hopf bifurcation is 
more challenging than that of a BT bifurcation.
\end{remark}

\section{Conclusion}

In this work, we applied hierarchical parametric analysis to study 
Bogdanov–Takens (BT) and generalized Hopf (GH) bifurcations,
with a particular focus on systematically determining their codimension.
For BT bifurcations, we introduced a one-step forward transformation
approach and highlighted its advantages and practicality for real-world
applications. To illustrate the method in detail, we analyzed two
population models, carefully outlining each step of the procedure.
In addition, numerical simulations were carried out to validate
the theoretical results.

The results presented in this paper clearly demonstrate the effectiveness 
of our one-step transformation approach for analyzing high-codimension 
BT bifurcations, as well as the usefulness of parametric analysis for 
determining the codimension of Hopf bifurcations, both of which pose 
significant challenges in real-world applications. These methods can 
also be extended to study other nonlinear systems arising in 
physical and engineering contexts.

\section*{Acknowledgment}

This work was supported by the Natural Sciences and Engineering 
Research Council of Canada, No.~R2686A02 (P. Yu),
and the National Natural Science Foundation of China,
No.~12571187 (M. Han).

\end{document}